\documentclass[12pt,twosides]{amsart}
\usepackage{amssymb,amsmath,amsthm, amscd, enumerate, mathrsfs}
\usepackage{graphicx, hhline}
\usepackage[all]{xy}
\usepackage{braket}
\usepackage[usenames]{color}
\usepackage{hyperref}
\usepackage{fancyhdr}
\usepackage{url}
\usepackage[top=30truemm,bottom=30truemm,left=30truemm,right=30truemm]{geometry}

\hypersetup{colorlinks=true}

\title[MMP for lc pairs on complex analytic spaces]{Minimal model program for log canonical pairs on complex analytic spaces}
\author{Makoto Enokizono and Kenta Hashizume}
\date{2025/12/09}
\address{Department of Mathematics, College of Science, Rikkyo University, Tokyo 171-8501, Japan}
\email{enokizono@rikkyo.ac.jp}
\address{Department of 
Mathematics, Faculty of Science, Niigata University, Niigata 950-2181, Japan}
\address{Institute for Research Administration, Niigata University, Niigata 950-2181, Japan}
\email{hkenta@math.sc.niigata-u.ac.jp}

\newtheorem{thm}{Theorem}[section]

\newtheorem{lem}[thm]{Lemma}
\newtheorem{cor}[thm]{Corollary}
\newtheorem{prop}[thm]{Proposition}

\theoremstyle{definition}
\newtheorem{defn}[thm]{Definition}
\newtheorem{rem}[thm]{Remark}

\newtheorem{exam}[thm]{Example}
\newtheorem*{ack}{Acknowledgments} 
 
\newtheorem*{b-divisor}{b-divisors} 
 
\newtheorem*{g-pair}{Generalized pairs} 
\newtheorem*{adj-g-pair}{Divisorial adjunction for generalized pairs} 
\newtheorem*{mmp-g-pair}{MMP for generalized pairs}

\newtheorem{step1}{Step}
\newtheorem{step2}{Step}
\newtheorem{step3}{Step}

\newtheorem*{claim*}{Claim}
\begin{document}

\begin{abstract}
We study the minimal model program for lc pairs on projective morphism between complex analytic spaces. 
More precisely, we generalize the results by Birkar \cite{birkar-flip} and the second author \cite{has-finite} to the setup by Fujino \cite{fujino-analytic-bchm}. 

\end{abstract}

\maketitle

\tableofcontents

\section{Introduction}

In \cite{fujino-analytic-bchm}, Fujino studied the minimal model theory for projective morphisms between complex analytic spaces. 
More precisely, he generalized the minimal model theory for Kawamata log terminal (klt, for short) pairs by Birkar, Cascini, Hacon, and M\textsuperscript{c}Kernan \cite{bchm} to complex analytic setting. 
Compared to the algebraic case \cite{bchm}, the Noetherian property of the structure sheaves does not hold in general, and some fundamental notions, for example, the canonical divisor and the $\mathbb{Q}$-factoriality, is more delicate in the analytic setting. 
These difficulties make the setup of the minimal model theory in the complex analytic case more complicated. 
By considering the minimal model theory over a Stein neighborhood of an appropriate Stein compact set, Fujino \cite{fujino-analytic-bchm} established the first step of the minimal model theory for projective morphisms between complex analytic spaces. 
For other generalizations of the minimal model theory in the complex analytic setting,  see \cite{hp-MFS}, \cite{hp-minmodel}, and \cite{chp} by Campana, H\"{o}ring, and Peternell, and \cite{lyu--murayama-mmp} by Lyu and Murayama. 

The goal of this paper is to develop the minimal model theory for log canonical (lc, for short) pairs in the complex analytic setting. 
In this paper, we will follow the setup in \cite{fujino-analytic-bchm}. 
As in \cite{fujino-analytic-bchm} or Definition \ref{defn--property(P)}, we say that a projective morphism $\pi \colon X \to Y$ of complex analytic spaces and a compact subset $W \subset Y$ satisfies the property (P) if the following conditions hold:
\begin{itemize}
\item[(P1)]
$X$ is a normal complex variety,
\item[(P2)]
$Y$ is a Stein space,
\item[(P3)]
$W$ is a Stein compact subset of $Y$, and
\item[(P4)]
$W \cap Z$ has only finitely many connected components for any analytic subset $Z$
which is defined over an open neighborhood of $W$. 
\end{itemize}
Let us quickly explain these properties. 
The property (P2) corresponds to the affine property of the base scheme in the algebraic case. 
The property (P3) says that $W$ has a fundamental system of Stein open neighborhood. 
This property is used to shrink $Y$ around $W$ in each step of the minimal model program (MMP, for short) keeping (P2). 
The property (P4) is used to suppose the Noetherian property of ``germ of $\mathcal{O}_{Y}$ around $W$'' as in \cite[Theorem 2.10]{fujino-analytic-bchm} and the finiteness of the relative Picard rank of $\pi$ over $W$.

The following theorems are the main results of this paper.

\begin{thm}[= Theorem \ref{thm--mmp-ample+eff}, cf.~{\cite{hashizumehu}}]\label{thm--polarizationtype}
Let $\pi \colon X \to Y$ be a projective morphism from a normal analytic variety $X$ to a Stein space $Y$, and let $W \subset Y$ be a compact subset such that $\pi$ and $W$ satisfy (P). 
Let $(X,B+A)$ be an lc pair, where $B$ is an effective $\mathbb{R}$-divisor on $X$ and $A$ is an effective $\pi$-ample $\mathbb{R}$-divisor on $X$.  
Let $f \colon (\tilde{X},\tilde{B}) \to (X,B)$ be a dlt blow-up of $(X,B)$, and we put $\tilde{\Gamma}=\tilde{B}+f^{*}A$. 
Let $\tilde{H}$ be a $(\pi \circ f)$-ample $\mathbb{R}$-divisor on $\tilde{X}$ such that $(\tilde{X},\tilde{\Gamma}+\tilde{H})$ is lc and $K_{\tilde{X}}+\tilde{\Gamma}+\tilde{H}$ is nef over $W$. 
Then there exists a sequence of steps of a $(K_{\tilde{X}}+\tilde{\Gamma})$-MMP over $Y$ around $W$ with scaling of $\tilde{H}$
$$(\tilde{X}_{0}:=\tilde{X},\tilde{\Gamma}_{0}:=\tilde{\Gamma}) \dashrightarrow (\tilde{X}_{1},\tilde{\Gamma}_{1}) \dashrightarrow\cdots \dashrightarrow (\tilde{X}_{m},\tilde{\Gamma}_{m})$$
such that after shrinking $Y$ around $W$ the lc pair $(\tilde{X}_{m},\tilde{\Gamma}_{m})$ is a log minimal model or a Mori fiber space of $(X,B+A)$ over $Y$ around $W$.  
\end{thm}

\begin{thm}[= Theorem \ref{thm--mmp-negativetrivial}, cf.~{\cite{birkar-flip}}, {\cite[Theorem 1.6]{haconxu-lcc}}, {\cite[Theorem 1.1]{has-mmp}}]\label{thm--lcfliptype}
Let $\pi \colon X \to Y$ be a projective morphism from a normal analytic variety $X$ to a Stein space $Y$, and let $W \subset Y$ be a compact subset such that $\pi$ and $W$ satisfy (P). 
Let $(X,B)$ be a dlt pair and $A$ an effective $\mathbb{R}$-Cartier divisor on $X$ such that $(X,B+A)$ is lc and $K_{X}+B+A\sim_{\mathbb{R},\,Y}0$.  
Let $H$ be a $\pi$-ample $\mathbb{R}$-divisor on $Y$ such that $(X,B+H)$ is lc and $K_{X}+B+H$ is nef over $W$. 
Then there exists a sequence of steps of a $(K_{X}+B)$-MMP over $Y$ around $W$ with scaling of $H$
$$(X_{0}:=X,B_{0}:=B) \dashrightarrow (X_{1},B_{1}) \dashrightarrow\cdots \dashrightarrow (X_{m},B_{m})$$
such that after shrinking $Y$ around $W$ the lc pair $(X_{m},B_{m})$ is a log minimal model or a Mori fiber space of $(X,B)$ over $Y$ around $W$.  
\end{thm}



Theorem \ref{thm--lcfliptype} is new and nontrivial even in the klt case because we  cannot directly apply the argument by Lai \cite{lai} to the $\mathbb{R}$-boundary divisor case, and we need the equi-dimensinoal reduction (cf.~\cite{ak}) in the analytic setting established in \cite{eh-semistablereduction}. 

To prove the main results, we will apply the ideas developed in \cite{has-finite} by the second author. 
The key technical result is the following:

\begin{thm}[= Theorem \ref{thm--mmp-ind2}, cf.~{\cite[Theorem 3.5]{has-finite}}]\label{thm--key-intro} 
Let $\pi \colon X \to Y$ be a projective morphism from a normal analytic variety $X$ to a Stein space $Y$, and let $W \subset Y$ be a compact subset such that $\pi$ and $W$ satisfy (P). 
Let $(X,\Delta)$ be an lc pair. 
Let $A$ be an effective $\mathbb{R}$-divisor on $X$ such that $(X,\Delta+A)$ is an lc pair and $K_{X}+\Delta+A$ is nef over $W$.
Then no infinite sequence of steps of a $(K_{X}+\Delta)$-MMP over $Y$ around $W$ with scaling of $A$
$$(X_{0}:=X,\Delta_{0}:=\Delta) \dashrightarrow (X_{1},\Delta_{1}) \dashrightarrow\cdots \dashrightarrow (X_{i},\Delta_{i})\dashrightarrow \cdots$$
satisfies the following properties.
\begin{itemize}
\item
If we define $\lambda_{i}={\rm inf}\set{\mu \in \mathbb{R}_{\geq0} \!|\! K_{X_{i}}+\Delta_{i}+\mu A_{i}\text{\rm \, is nef over }W}$, where $A_{i}$ is the birational transform of $A$ on $X_{i}$, then ${\rm lim}_{i\to\infty}\lambda_{i}=0$, and  
\item
there are infinitely many $i$ such that $(X_{i},\Delta_{i})$ is log abundant over $Y$ around $W$. 
\end{itemize}
\end{thm} 

For Theorem \ref{thm--key-intro}, we will establish fundamental results of the asymptotic vanishing order and the Nakayama--Zariski decomposition in \cite{nakayama} (see Section \ref{sec4}), techniques of MMP as in \cite{fujino-sp-ter} and \cite{birkar-flip} (see Section \ref{sec3}), and we will apply the equi-dimensinoal reduction (cf.~\cite{ak}) in the analytic setting established in \cite{eh-semistablereduction}. 
We also note that some important results on the log canonical abundance in the algebraic case are necessary to prove the main results of this paper. 
For example, to prove Lemma \ref{lem--mmp-genericCYfib}, which is crucial to the proof of Theorem \ref{thm--key-intro}, we need the abundance theorem for lc pairs of numerical dimension zero \cite{gongyo} by Gongyo. 
Moreover, we need \cite[Theorem 1.5]{hashizumehu} by the second author and Hu for the reduction of Theorem \ref{thm--polarizationtype} to Theorem \ref{thm--key-intro}.

The contents of this paper are as follows: 
In Section \ref{sec2}, we collect some notions on complex analytic spaces used in this paper. 
In Section \ref{sec3}, we define a sequence of steps of an MMP in the analytic setting and prove some fundamental results. 
In Section \ref{sec4}, we discuss the asymptotic vanishing order and the Nakayama--Zariski decomposition in the analytic setting. 
In Section \ref{sec5}, we prove the main results of this paper.

\begin{ack}
The first author was partially supported by JSPS KAKENHI Grant Number JP20K14297.
The second author was partially supported by JSPS KAKENHI Grant Number JP22K13887. 
The authors are grateful to Professor Osamu Fujino for fruitful discussions, constant support, and warm encouragement.   
\end{ack}

\section{Preliminaries}\label{sec2}

Throughout this paper, {\em complex analytic spaces} are always assumed to be Hausdorff and second countable. 
{\em Analytic varieties} are reduced and irreducible complex analytic spaces. 

In this section, we collect definitions and basic results used in this paper. 
For the notations and definitions, see \cite{fujino-analytic-bchm}. 

A {\em contraction} $\pi \colon X \to Y$ is a projective morphism between analytic varieties such that $\pi_{*}\mathcal{O}_{X} \cong \mathcal{O}_{Y}$. 
For any contraction $\pi \colon X \to Y$, if $X$ is a normal analytic variety then $Y$ is also a normal analytic variety. 
For an analytic variety $X$ and an $\mathbb{R}$-divisor $D$ on $X$, a {\em log resolution of} $(X,D)$ is a projective bimeromorphism $g \colon X' \to X$ from a non-singular analytic variety $X'$ such that the exceptional locus ${\rm Ex}(g)$ of $g$ is pure codimension one and ${\rm Ex}(g)\cup {\rm Supp}\,g_{*}^{-1}D$ is a simple normal crossing divisor. 

Let $X$ be a complex analytic space. 
A subset $S$ of $X$ is said to be {\em analytically meagre} if $S \subset \bigcup_{n \in \mathbb{Z}_{>0}}Y_{n}$, where each $Y_{n}$ is a locally closed analytic subset of $X$ such that ${\rm codim}_{X}(Y_{n})\geq 1$. 
An {\em analytically sufficiently general point of $X$} is a point which is not contained in an analytically meagre subset. 
Let $f \colon X \to Y$ be a morphism of analytic spaces. 
Then an {\em analytically sufficiently general fiber of $f$} is the inverse image of an  analytically sufficiently general point of $Y$.

\begin{defn}[Property (P), see {\cite{fujino-analytic-bchm}}]\label{defn--property(P)}
Let $\pi \colon X \to Y$ be a projective morphism of complex analytic spaces, and let $W \subset Y$ be a compact subset. 
In this paper, we will use the following conditions:
\begin{itemize}
\item[(P1)]
$X$ is a normal complex variety,
\item[(P2)]
$Y$ is a Stein space,
\item[(P3)]
$W$ is a Stein compact subset of $Y$, and
\item[(P4)]
$W \cap Z$ has only finitely many connected components for any analytic subset $Z$
which is defined over an open neighborhood of $W$. 
\end{itemize}
We say that {\em $\pi \colon X \to Y$ and $W \subset Y$ satisfy (P)} if the conditions (P1)--(P4) hold. 
\end{defn}

\subsection{Complex analytic space}

We collect basic facts on complex analytic varieties.

\begin{thm}\label{thm--birat-basic}
Let $f\colon X' \to X$ be a proper bimeromorphism between normal analytic varieties. 
Then $f_{*}\mathcal{O}_{X'}\cong \mathcal{O}_{X}$ and there exists an Zariski open subset $U \subset X$ over which $f$ is a biholomorphism and ${\rm codim}_{X}(X\setminus U)\geq 2$. 
\end{thm}

\begin{proof}
It is well known. 
\end{proof}

\begin{thm}[cf.~{\cite[Section 13]{resolution-1}}]\label{thm--resolution}
Let $X$ be an analytic variety and let $\mathcal{I}$ be a coherent sheaf on $X$. 
Then there exists a projective bimeromorphism $f \colon X' \to X$ from a non-singular analytic variety $X'$ such that $\mathcal{I}\cdot \mathcal{O}_{X'}$ is an invertible sheaf on $X'$ and there exists an effective $f$-exceptional Cartier divisor $E'$ on $X'$ such that ${\rm Supp}\,E'={\rm Ex}(f)$ and $-E'$ is $f$-ample.  
\end{thm}

\begin{rem}
The resolution $f \colon X' \to X$ in Theorem \ref{thm--resolution} is not necessarily a finite sequence of blow-ups.  
\end{rem}

\begin{lem}\label{lem--cont-shrink}
Let $\pi \colon X \to Y$ be a contraction from a normal analytic variety to an analytic space. 
Then $\pi^{-1}(U)$ is a normal variety for any connected open subset $U \subset Y$. 
\end{lem}

\begin{proof}
Fix a connected open subset $U \subset Y$ and we put $X_{U}=\pi^{-1}(U)$. 
For any point $x \in X$, the stalk $\mathcal{O}_{X,\, x}$ is an integral domain. 
Hence $\mathcal{O}_{X_{U},\, x'}$ is an integral domain for any point $x' \in X_{U}$. 
Therefore $X_{U}$ is a disjoint union of normal varieties. 
If there are at least two connected components, then we can write $X_{U}=X'_{U} \amalg X''_{U}$ for some closed subsets $X'_{U}$ and $X''_{U}$ of $X_{U}$. 
Since $\pi|_{X_{U}}$ is proper, $\pi(X'_{U})$ and $\pi(X''_{U})$ are closed subsets of $U$ and 
$U= \pi(X'_{U}) \cup \pi(X''_{U})$. 
Since $U$ is connected, there exists a point $y \in \pi(X'_{U}) \cap \pi(X''_{U})$. 
Then $f^{-1}(y)=(f^{-1}(y) \cap X'_{U}) \amalg (f^{-1}(y) \cap X''_{U})$, which contradicts the fact that $\pi$ is a contraction. 
Thus $X_{U}$ is a normal variety. 
\end{proof}

\begin{lem}\label{lem--shrink-irreducible}
Let $\pi \colon X \to Y$ be a projective morphism from a normal analytic variety $X$ to an analytic space $Y$, and let $W \subset Y$ be a compact subset such that $W \cap Z$ has only finitely many connected components for any analytic subset $Z$ which is defined over an open neighborhood of $W$ (see also the condition $({\rm P}4)$ in \cite{fujino-analytic-bchm}). 
Then there exists an open subset $Y' \subset Y$ containing $W$ such that if $\pi^{-1}(Y') =\coprod_{\lambda \in \Lambda}X'_{\lambda}$ is the decomposition of $\pi^{-1}(Y')$ into connected components $X'_{\lambda}$, then there are only finitely many connected components, which we denote by $X'_{1},\cdots,\, X'_{l}$, such that their images on $Y$ by $\pi$ intersect $W$. 
Furthermore, these $X'_{1},\cdots,\, X'_{l}$ satisfy the following properties.
\begin{itemize}
\item
$X'_{i}$ is a normal complex variety for all $1 \leq i \leq l$, and
\item
for any $1 \leq i \leq l$ and open subset $U \subset Y'$ containing $W$, there exists a unique connected component $X_{U}^{(i)}$ of $\pi^{-1}(U) \cap X'_{i}$ such that $\pi(X_{U}^{(i)}) \cap W \neq \emptyset$, and furthermore, $X_{U}^{(i)}$ is a normal analytic variety. 
\end{itemize} 
\end{lem}

\begin{proof}
We may assume $\pi(X) \cap W \neq \emptyset$ because otherwise we may set $Y'=Y\setminus \pi(X)$. 
Let $\widetilde{\pi} \colon X \to \widetilde{Y}$ be the Stein factorization of $\pi$, and let $f \colon \widetilde{Y} \to Y$ be the induced finite morphism. 
We set $\widetilde{W}=f^{-1}(W)$. 
Then $\widetilde{W}$ is compact, and \cite[Theorem 2.13]{fujino-analytic-bchm} shows that $\widetilde{W} \cap \widetilde{Z}$ has only finitely many connected components for any analytic subset $\widetilde{Z}$ which is defined over an open neighborhood of $\widetilde{W}$. 
In particular, $\widetilde{W}$ has only finitely many connected components $\widetilde{W}_{1},\,\cdots,\,\widetilde{W}_{l}$. 
Then we can take open subsets $\widetilde{Y}_{1},\,\cdots,\,\widetilde{Y}_{l}$ of $\widetilde{Y}$ such that $\widetilde{Y}_{i} \supset \widetilde{W}_{i}$ for every $i$ and $\widetilde{Y}_{i} \cap \widetilde{Y}_{i'}=\emptyset$ for any indices $i$ and $i'$. 
Let $Y'$ be the complement of $f(Y \setminus \coprod_{i=1}^{l}\widetilde{Y}_{i})$. 
Since $f$ is finite, $Y'$ is open, and $Y' \supset W$. 

We check that this $Y'$ satisfies the condition of Lemma \ref{lem--shrink-irreducible}.  
By construction, we have
$$\coprod_{i=1}^{l}\widetilde{W}_{i}=\widetilde{W} \subset f^{-1}(Y') \subset \coprod_{i=1}^{l}\widetilde{Y}_{i}.$$
By the connectedness of each $\widetilde{W}_{i}$ and considering the decomposition $f^{-1}(Y')=\coprod_{\gamma \in \Gamma}\widetilde{Y}'_{\gamma}$ of $f^{-1}(Y')$ into connected components, we can find connected components $\widetilde{Y}'_{1},\,\cdots,\,\widetilde{Y}'_{l}$ such that $\widetilde{W}_{i} \subset \widetilde{Y}'_{i} \subset \widetilde{Y}_{i}$ for every $1 \leq i \leq l$, and the other connected components of $f^{-1}(Y')$ cannot intersect $\widetilde{W}$. 
For each $\gamma \in \Gamma$, we put $X'_{\gamma}=\widetilde{\pi}^{-1}(\widetilde{Y}'_{\gamma})$. 
Then
$$X':=\widetilde{\pi}^{-1}(f^{-1}(Y'))=\widetilde{\pi}^{-1}\left(\coprod_{\gamma \in \Gamma}\widetilde{Y}'_{\gamma}\right)=\coprod_{\gamma \in \Gamma}X'_{\gamma}.$$
Since $\widetilde{\pi} \colon X \to \widetilde{Y}$ is a contraction, all $X'_{\gamma}$ are connected. 
Therefore, $X'=\coprod_{\gamma \in \Gamma}X'_{\gamma}$ is the decomposition of $X'$ into connected component, and a connected component $X'_{\gamma}$ of $X'$ satisfying $\pi(X'_{\gamma})\cap W \neq \emptyset$ is one of $X'_{1},\,\cdots,\,X'_{l}$.  

To finish the proof, it is sufficient to prove that $X'_{1},\,\cdots,\,X'_{l}$ satisfy the two properties of Lemma \ref{lem--shrink-irreducible}. 
By applying Lemma \ref{lem--cont-shrink} to $\widetilde{\pi} \colon X \to \widetilde{Y}$ and $\widetilde{Y}'_{1},\,\cdots,\,\widetilde{Y}'_{l}$, it follows that $X'_{1},\,\cdots,\,X'_{l}$ are normal analytic varieties. 
Therefore the first assertion of Lemma \ref{lem--shrink-irreducible} holds. 
To prove the second assertion of Lemma \ref{lem--shrink-irreducible}, we take an open subset $U \subset Y'$ containing $W$. 
Then
$$\coprod_{i=1}^{l}\widetilde{W}_{i}=\widetilde{W} \subset f^{-1}(U) \subset \coprod_{\gamma \in \Gamma}\widetilde{Y}_{\gamma}.$$
By considering the decomposition of $f^{-1}(U)$ into connected components, we may find connected components $\widetilde{U}_{1},\,\cdots,\,\widetilde{U}_{l}$ of $f^{-1}(U)$ such that $\widetilde{W}_{i} \subset \widetilde{U}_{i} \subset \widetilde{Y}'_{i}$ for all $1 \leq i \leq l$, and the other connected components of $f^{-1}(U)$ cannot intersect $\widetilde{W}$. 
For each $1 \leq i \leq l$, we put $X_{U}^{(i)}=\widetilde{\pi}^{-1}(\widetilde{U}_{i})$. 
By the same argument as in the case of $\widetilde{Y}'_{1},\,\cdots,\,\widetilde{Y}'_{l}$, we see that $X_{U}^{(1)},\, \cdots,\, X_{U}^{(l)}$ are normal analytic varieties and connected components of $\pi^{-1}(U)$, and each $X_{U}^{(i)}$ is the unique connected component of $\pi^{-1}(U) \cap X'_{i}$ such that $\pi(X_{U}^{(i)}) \cap W \neq \emptyset$. 
From these facts, the second assertion of Lemma \ref{lem--shrink-irreducible} holds. 
We finish the proof. 
\end{proof}

\subsection{Stable base locus}

We collect basic results on stable base loci of $\mathbb{R}$-divisors. 

\begin{lem}\label{lem--base-locus-1}
Let $\pi\colon X \to Y$ be a projective morphism from a normal analytic variety $X$ to a Stein space $Y$. 
Let $D$ be a Cartier divisor on $X$. 
Then 
$$\bigcap_{D\sim E \geq0}{\rm Supp}\,E={\rm Supp}\,\bigl({\rm Coker}(\pi^{*}\pi_{*}\mathcal{O}_{X}(D)\otimes \mathcal{O}_{X}(-D)\to \mathcal{O}_{X})\bigr)$$
set-theoretically. 
The left hand hand side is denoted by ${\boldsymbol{\rm Bs}}|D|$. 
\end{lem}

\begin{proof}
Pick $x \in X$. 
We put
$$T:={\rm Supp}\,\bigl({\rm Coker}(\pi^{*}\pi_{*}\mathcal{O}_{X}(D)\otimes \mathcal{O}_{X}(-D)\to \mathcal{O}_{X})\bigr).$$
If $x \not\in {\boldsymbol{\rm Bs}}|D|$, then there is an effective divisor $E \sim D$ such that $x \not\in {\rm Supp}\,E$. 
Taking $\phi \in \Gamma(X, \mathcal{O}_{X}(D))$ corresponding to $E$, then we can easily check $\phi\cdot \mathcal{O}_{X}(-D)_{x}=\mathcal{O}_{X, x}$. 
Therefore, $x \not\in T$. 
From this, we have 
${\boldsymbol{\rm Bs}}|D| \supset T$. 

If $x \not \in T$, then there exist an open subset $U \subset Y$ and an element $\psi \in \Gamma(U, \pi_{*}\mathcal{O}_{X}(D))$ such that ${\rm Supp}\,({\rm div}(\psi)+D|_{\pi^{-1}(U)})\not\ni x$. 
Since $Y$ is Stein, $\pi_{*}\mathcal{O}_{X}(D)$ is generated by global sections, in other words, $\Gamma(Y, \pi_{*}\mathcal{O}_{X}(D))\cdot \mathcal{O}_{Y, y}=(\pi_{*}\mathcal{O}_{X}(D))_{y}$ for any $y \in Y$. 
By considering the case $y=f(x)$, we can find an open subset $V \subset U$ and elements $\varphi \in  \Gamma(Y, \pi_{*}\mathcal{O}_{X}(D))=\Gamma(X,\mathcal{O}_{X}(D))$ and $\tau \in \mathcal{O}_{Y}(V)$ such that $\psi|_{V}=\varphi|_{V}\cdot \tau$. 
We put 
$$E':= {\rm div}(\varphi)+D.$$ 
Then $E' \geq 0$. 
We also have ${\rm Supp}\,E' \not\ni x$ because
$$E'|_{\pi^{-1}(V)}=({\rm div}(\varphi)+D)|_{\pi^{-1}(V)}\leq {\rm div}(\varphi|_{V}\cdot \pi^{*}\tau)+D|_{\pi^{-1}(V)}={\rm div}(\psi|_{V})+D|_{\pi^{-1}(V)}$$
and ${\rm Supp}\,({\rm div}(\psi)+D|_{\pi^{-1}(U)})\not\ni x$. 
From this, we have $x \not \in {\boldsymbol{\rm Bs}}|D|$. 
Hence ${\boldsymbol{\rm Bs}}|D| \subset T$. 

By the above argument, we have ${\boldsymbol{\rm Bs}}|D| = T$, which we wanted to prove. 
\end{proof}

\begin{prop}\label{prop--q-base-locus}
Let $\pi\colon X \to Y$ be a projective morphism from a normal analytic variety $X$ to a Stein space $Y$. 
Let $D$ be a globally $\mathbb{Q}$-Cartier divisor on $X$. 
We put 
$$|D|_{\mathbb{Q}}:=\{E\,|\,D\sim_{\mathbb{Q}}E\geq 0\},\quad {\rm and} \quad |D/Y|_{\mathbb{Q}}:=\{E\,|\,D\sim_{\mathbb{Q},\,Y}E\geq 0\}.$$
The notations are different from \cite[Section 10]{fujino-analytic-bchm}. 
Then we have
\begin{equation*}
\begin{split}
{\boldsymbol{\rm Bs}}|D|_{\mathbb{Q}}=&{\boldsymbol{\rm Bs}}|D/Y|_{\mathbb{Q}}\\
=&\bigcap_{mD\,:\,{\rm Cartier}}{\rm Supp}\,\bigl({\rm Coker}(\pi^{*}\pi_{*}\mathcal{O}_{X}(mD)\otimes \mathcal{O}_{X}(-mD)\to \mathcal{O}_{X})\bigr)
\end{split}
\end{equation*}
set-theoretically. 
\end{prop}

\begin{proof}
For each $m >0$ such that $mD$ is Cartier, we put
$$T_{m}:={\rm Supp}\,\bigl({\rm Coker}(\pi^{*}\pi_{*}\mathcal{O}_{X}(mD)\otimes \mathcal{O}_{X}(-mD)\to \mathcal{O}_{X})\bigr).$$
By Lemma \ref{lem--base-locus-1}, we have 
$$\bigcap_{mD\,:\,{\rm Cartier}}T_{m}=\bigcap_{mD\,:\,{\rm Cartier}}{\boldsymbol{\rm Bs}}|mD|,$$
and we can easily check that the right hand side is equal to ${\boldsymbol{\rm Bs}}|D|_{\mathbb{Q}}$. 
It is also clear that ${\boldsymbol{\rm Bs}}|D|_{\mathbb{Q}}\supset{\boldsymbol{\rm Bs}}|D/Y|_{\mathbb{Q}}$ since we have $|D|_{\mathbb{Q}} \subset |D/Y|_{\mathbb{Q}}$. 
Therefore, we only need to prove ${\boldsymbol{\rm Bs}}|D/Y|_{\mathbb{Q}}\supset \bigcap_{m}T_{m}$. 
Pick a point $x \not\in {\boldsymbol{\rm Bs}}|D/Y|_{\mathbb{Q}}$. 
Then there exist an effective $\mathbb{Q}$-divisor $E$ on $X$ and a $\mathbb{Q}$-Cartier divisor $G$ on $Y$ such that $D +\pi^{*}G\sim_{\mathbb{Q}}E$ and ${\rm Supp}\,E \not\ni x$. 
Let $U \subset Y$ be an open subset such that $U \ni x$ and some multiple of $G$ is a principal divisor on $U$. 
Such $U$ exists since $G$ is globally $\mathbb{Q}$-Cartier on any relatively compact open subset. 
Then there exist a positive integer $l$, a meromorphic function $\phi$ on $X$, and a meromorphic function $\psi$ on $Y$ such that $lD$ is Cartier, $lG|_{U}={\rm div}(\psi)$, and
$$(lD+{\rm div}(\phi \cdot \pi^{*}\psi))|_{\pi^{-1}(U)}=lE|_{\pi^{-1}(U)}.$$
This implies $(\phi \cdot \pi^{*}\psi))|_{\pi^{-1}(U)} \in \Gamma(U, \pi_{*}\mathcal{O}_{X}(lD))$. 
Therefore, we see that $x \not\in T_{l}$. 
From this fact, we have ${\boldsymbol{\rm Bs}}|D/Y|_{\mathbb{Q}}\supset \bigcap_{m}T_{m}$. 
Now we have 
$$\bigcap_{mD\,:\,{\rm Cartier}}T_{m}= {\boldsymbol{\rm Bs}}|D|_{\mathbb{Q}} \supset {\boldsymbol{\rm Bs}}|D/Y|_{\mathbb{Q}} \supset \bigcap_{mD\,:\,{\rm Cartier}}T_{m},$$ 
and therefore Proposition \ref{prop--q-base-locus} holds.
\end{proof}

\begin{prop}\label{prop--r-base-locus}
Let $\pi\colon X \to Y$ be a projective morphism from a normal analytic variety $X$ to a Stein space $Y$. 
Let $D$ be a globally $\mathbb{R}$-Cartier divisor on $X$. 
We set 
$$|D|_{\mathbb{R}}:=\{E\,|\,D\sim_{\mathbb{R}}E\geq 0\},\quad {\rm and} \quad |D/Y|_{\mathbb{R}}:=\{E\,|\,D\sim_{\mathbb{R},\,Y}E\geq 0\}.$$ 
The notations are different from \cite[Section 10]{fujino-analytic-bchm}. 
Then, for any non-empty open subset $U \subset Y$, we have
$$({\boldsymbol{\rm Bs}}|D|_{\mathbb{R}})|_{\pi^{-1}(U)}={\boldsymbol{\rm Bs}}|D|_{\pi^{-1}(U)}|_{\mathbb{R}}=({\boldsymbol{\rm Bs}}|D/Y|_{\mathbb{R}})|_{\pi^{-1}(U)}={\boldsymbol{\rm Bs}}|D|_{\pi^{-1}(U)}/U|_{\mathbb{R}}$$ set-theoretically. 
\end{prop}

\begin{proof} 
Let $D=\sum_{i=1}^{p}r_{i}D_{i}$ be the decomposition of $D$ into Cartier divisors on $X$. 
It is easy to check that we only need to prove that $({\boldsymbol{\rm Bs}}|D|_{\mathbb{R}})|_{\pi^{-1}(U)}\subset{\boldsymbol{\rm Bs}}|D|_{\pi^{-1}(U)}/U|_{\mathbb{R}}$. 
We denote the morphism $\pi^{-1}(U) \to U$ by $\pi_{U}$. 

Fix a point $x \in \pi^{-1}(U)$ that is not contained in ${\boldsymbol{\rm Bs}}|D|_{\pi^{-1}(U)}/U|_{\mathbb{R}}$. 
Then there exist finitely many real numbers $s_{1},\,\cdots,\,s_{m}$, bimeromorphic functions $\phi_{1},\cdots,\,\phi_{m}$ on $U$, an $\mathbb{R}$-Cartier divisor $G$ on $U$, and an effective $\mathbb{R}$-divisor $E$ on $\pi^{-1}(U)$ such that 
$$D|_{\pi^{-1}(U)}+\pi^{*}_{U}G+\sum_{j=1}^{m}s_{j}\cdot {\rm div}(\pi^{*}_{U}\phi_{j})=E$$
and ${\rm Supp}\,E \not\ni x$. 
Let $V \subset U$ be a relatively compact open subset such that $V \ni f(x)$. 
Then we can write $G|_{V}=\sum_{k=1}^{n}t_{k}G_{k}$, where $G_{k}$ are Cartier divisors on $V$. 
By shrinking $V$ around $x$ if necessary, we may assume that $G_{k}$ are principal divisors, which we denote by ${\rm div}(\gamma_{k})$, for all $1 \leq k \leq n$. 
Now we have
$$\sum_{i=1}^{p}r_{i}D_{i}|_{\pi^{-1}(V)}+\sum_{k=1}^{n}t_{k}\cdot {\rm div}(\pi^{*}_{V}\gamma_{k})+\sum_{j=1}^{m}s_{j}\cdot {\rm div}(\pi^{*}_{U}\phi_{j})|_{\pi^{-1}(V)}=E|_{\pi^{-1}(V)}.$$
Since $V$ is relatively compact, $D_{i}|_{\pi^{-1}(V)}$, ${\rm div}(\pi^{*}_{V}\gamma_{k})$, ${\rm div}(\pi^{*}_{U}\phi_{j})|_{\pi^{-1}(V)}$, and $E|_{\pi^{-1}(V)}$ have only finitely many components. 
By the argument from convex geometry, we can find finitely many positive real numbers 
$r'_{1},\cdots, r'_{q}$, $\mathbb{Q}$-Cartier divisors $D'_{1},\cdots ,D'_{q}$ on $X$ that are finite $\mathbb{Q}$-linear combinations of $D_{1},\cdots,\,D_{p}$, and an effective $\mathbb{Q}$-divisors $E'_{1},\cdots ,E'_{q}$ on $\pi^{-1}(V)$ such that 
\begin{itemize}
\item
$\sum_{l=1}^{q}r'_{l}=1$ and $\sum_{l=1}^{q}r'_{l}D'_{l}=D$, 
\item
${\rm Supp}\,E'_{l}={\rm Supp}\,E|_{\pi^{-1}(V)}$ for all $1 \leq l \leq p$ and $\sum_{l=1}^{q}r'_{l}E'_{l}=E|_{\pi^{-1}(V)}$, and 
\item
$D'_{l}|_{\pi^{-1}(V)} \sim_{\mathbb{Q}}E'_{l}$. 
\end{itemize}
By the second property, it follows that ${\rm Supp}\,E'_{l} \not\ni x$. 
By this fact and the third property, there exists a positive integer $I$ such that all $ID_{l}$ are Cartier and
$$x \not\in{\rm Supp}\,\bigl({\rm Coker}(\pi^{*}\pi_{*}\mathcal{O}_{X}(ID'_{l})\otimes \mathcal{O}_{X}(-ID'_{l})\to \mathcal{O}_{X})\bigr)$$
for all $1 \leq l \leq q$. 
By Proposition \ref{prop--q-base-locus}, we have $x \not\in {\boldsymbol{\rm Bs}}|D_{l}|_{\mathbb{Q}}$ for all $1 \leq l \leq p$. 
Since $r_{1},\cdots,\,r_{q}$ are positive real numbers and $\sum_{l=1}^{q}r'_{l}D'_{l}=D$, which is the first property stated above, we have $x \not\in {\boldsymbol{\rm Bs}}|D|_{\mathbb{R}}$. 
This shows $x \not\in ({\boldsymbol{\rm Bs}}|D|_{\mathbb{R}})|_{\pi^{-1}(U)}$. 

By the above argument, for any $x \not\in {\boldsymbol{\rm Bs}}|D|_{\pi^{-1}(U)}/U|_{\mathbb{R}}$, we have $x \not\in ({\boldsymbol{\rm Bs}}|D|_{\mathbb{R}})|_{\pi^{-1}(U)}$. 
Therefore we have $({\boldsymbol{\rm Bs}}|D|_{\mathbb{R}})|_{\pi^{-1}(U)}\subset{\boldsymbol{\rm Bs}}|D|_{\pi^{-1}(U)}/U|_{\mathbb{R}}$. 
\end{proof}

\subsection{Iitaka fibration}\label{subsec--iitaka-fib}

Let $\pi\colon X \to Y$ be a projective morphism from a non-singular analytic variety $X$ to an analytic space $Y$, and let $W \subset Y$ be a compact subset. 
Let $D$ be an $\mathbb{R}$-Cartier divisor on $X$ such that $D\sim_{\mathbb{R},\, Y}E$ for some effective $\mathbb{R}$-divisor $E$ on $X$. 
For each $l \in \mathbb{Z}_{>0}$, we consider the ideal sheaf $$\mathcal{I}_{l}:={\rm Im} \bigl( \pi^{*}\pi_{*}\mathcal{O}_{X}( \lfloor lE\rfloor ) \otimes \mathcal{O}_{X}( -\lfloor lE\rfloor ) \longrightarrow \mathcal{O}_{X} \bigr) \subset \mathcal{O}_{X}.$$
There exists a resolution $f_{l} \colon X'_{l} \to X$ of $X$ such that $\mathcal{I}_{l} \cdot \mathcal{O}_{X'_{l}}$ is an invertible sheaf on $X'_{l}$. 
We may write $\mathcal{I}_{l} \cdot \mathcal{O}_{X'_{l}}=\mathcal{O}_{X'_{l}}(-F'_{l})$ with an effective Cartier divisor $F'_{l}$ on $X'_{l}$. 
We put $M'_{l}:=f_{l}^{*}\lfloor lE \rfloor-F'_{l}$. 
Then there is a natural isomorphism $$(\pi \circ f_{l})_{*}\mathcal{O}_{X'_{l}}(M'_{l}) \simeq \pi_{*}\mathcal{O}_{X}( \lfloor lE\rfloor )$$ 
and the morphism 
$$(\pi \circ f_{l})^{*}(\pi \circ f_{l})_{*}\mathcal{O}_{X'_{l}}(M'_{l}) \otimes \mathcal{O}_{X'_{l}}(-M'_{l})\longrightarrow \mathcal{O}_{X'_{l}}$$ is surjective, in other words, $M'_{l}$ is globally generated over $Y$. 
Hence we can construct a contraction $\phi_{l} \colon X'_{l} \to Z_{l}$ over $Y$ between normal analytic varieties and a Cartier divisor $H_{l}$ on $Z_{l}$ such that $Z_{l}$ is projective over $Y$, $H_{l}$ is very ample over $Y$, and $M_{l}\sim \phi_{l}^{*}H_{l}$. 
We take $m \in \mathbb{Z}_{>0}$ such that ${\rm dim}\,Z_{m}$ takes the maximum among the dimension of $Z_{l}$ varying $l \in \mathbb{Z}_{>0}$. We call the meromorphic map $X \dashrightarrow Z:=Z_{m}$ an {\em Iitaka fibration over $Y$ associated to $D$}. 
Note that $Z$ is not unique, but $Z$ is determined up to bimeromorhic equivalence. 

Let $X \to Y'$ be the Stein factorization of $\pi \colon X \to Y$. 
Then the induced meromorphic map $Z \dashrightarrow Y'$ is a contraction and for any analytically sufficiently general point $p \in Y'$, the map $X_{p} \dashrightarrow Z_{p}$ of the fibers over $p$ is the usual Iitaka fibration associated to $D|_{X_{p}}$.

\subsection{Negativity lemma}

In this subsection, we prove variants of negativity lemma.

\begin{defn}[Movable divisor]
Let $\pi\colon X \to Y$ be a projective morphism from a normal analytic variety to an analytic space, and let $U \subset Y$ be an open subset. 
We say that an $\mathbb{R}$-Cartier divisor $D$ on $X$ is {\em movable over $U$} if any irreducible component of ${\boldsymbol{\rm Bs}}|D|_{\pi^{-1}(U)}/U|_{\mathbb{R}}$ has codimension at least two. 
Let $W \subset Y$ be a subset. 
We say that an $\mathbb{R}$-Cartier divisor $D$ on $X$ is {\em movable over a neighborhood of $W$} if $D$ is movable over an open subset $V$ containing $W$. 
\end{defn}

\begin{defn}[Very exceptional divisor, {\cite[Definition 3.1]{birkar-flip}}]\label{defn--veryexcepdiv}
Let $\pi \colon X \to Y$ be a contraction of normal analytic varieties. 
We say that an $\mathbb{R}$-divisor $D$ on $X$ is {\em very exceptional over $Y$} if $D$ is vertical and  for any prime divisor $Q$ on $Y$, there is a prime divisor $P$ on $X$ such that $\pi(P)=Q$ and $P$ is not a component of $D$. 
\end{defn}

\begin{rem}\label{rem--very ecepdiv-shrink}
Let $\pi \colon X \to Y$ be a contraction of normal analytic varieties and $D$ a very exceptional $\mathbb{R}$-divisor on $X$. 
Let $U \subset Y$ be an connected open subset, and we put $X_{U}=\pi^{-1}(U)$, $D_{U}=D|_{X_{U}}$, and $\pi_{U}=\pi|_{X_{U}} \colon X_{U} \to U$. 
Then $D_{U}$ is very exceptional over $U$. 
Indeed, picking any prime divisor $Q'$ on $U$, it is sufficient to find a prime divisor $P'$ on $X_{U}$ such that $\pi_{U}(P')=Q'$ and $P'$ is not a component of $D_{U}$. 
We may assume the existence of a prime divisor $\tilde{P} \subset {\rm Supp}\,D_{U}$ such that $\pi_{U}(\tilde{P})=Q'$ because otherwise the existence of such $P'$ as above is obvious. 
Then $Q'$ is an irreducible component of $\pi({\rm Supp}\,D)|_{U}$. 
In particular, there is a prime divisor $Q_{0}$ on $Y$ such that $Q'$ is an irreducible comonent of $Q_{0}|_{U}$. 
Since $D$ is very exceptional over $Y$, we can find a prime divisor $P_{0}$ on $X$ such that $\pi(P_{0})=Q_{0}$ and $P_{0}$ is not a component of $D$. 
Then $P_{0}|_{X_{U}} \not\subset D_{U}$ over general points of $Q'$, and therefore we can find a prime divisor $P'$ on $X_{U}$ such that $\pi_{U}(P')=Q'$ and $P'$ is not a component of $D_{U}$. 
Hence, we see that $D_{U}$ is very exceptional over $U$. 
\end{rem}

\begin{exam}
Let $\pi \colon X \to Y$ be a contraction of normal analytic varieties, and let $D$ be an $\mathbb{R}$-divisor on $X$. 
If $\pi({\rm Supp}\,D)$ has codimension at least two in $Y$, then $D$ is very exceptional over $Y$. 
In particular, any exceptional divisors in bimeromorphisms are also very exceptional divisors. 
\end{exam}

\begin{lem}[cf.~{\cite[Lemma 3.3]{birkar-flip}}]\label{lem--negativity-veryexc}
Let $\pi \colon X \to Y$ be a contraction of normal analytic varieties. 
Let $D=E-F$ be an $\mathbb{R}$-Cartier divisor on $X$, where $E$ and $F$ are effective $\mathbb{R}$-divisors having no common components. 
Suppose that $E$ is very exceptional over $Y$ and there is an analytically meagre subset $S \subset {\rm Supp}\,E$ such that for any curve $C \subset {\rm Supp}\,E$ contained in a fiber of $\pi$, if $C \not\subset S$ then $(D\cdot C)\geq 0$.
Then $-D \geq0$. 
\end{lem}

\begin{proof}
We will get a contradiction by assuming $E \neq 0$. 
By taking a Stein factorization of $\pi$, we may assume that $\pi$ is a contraction. 
By shrinking $Y$ around a point, we may assume that $Y$ is Stein.

\begin{step2}\label{lem--negativity-veryexc-step1}
In this step, we reduce the problem to the case ${\rm dim}\,\pi({\rm Supp}\,E)=0$. 

If ${\rm dim}\,\pi({\rm Supp}\,E)>0$, we take a sufficiently general hyperplane section $H$ of $Y$. 
Then $H$ and $\pi^{-1}(H)$ are normal analytic varieties (see, for example, \cite[(II.5) Theorem]{analytic-bertini}), and $\pi^{-1}(H) \to H$ is a contraction. 
It also follows that $E|_{\pi^{-1}(H)}$ is very exceptional over $H$. 
It is easy to check that $S \cap \pi^{-1}(H)$ is an analytically meagre subset of ${\rm Supp}\,E|_{\pi^{-1}(H)}$, and for any curve $C' \subset {\rm Supp}\,E|_{\pi^{-1}(H)}$ contained in a fiber of $\pi|_{\pi^{-1}(H)}$, if $C' \not\subset S \cap \pi^{-1}(H)$ then $(D|_{\pi^{-1}(H)}\cdot C')\geq 0$.
Therefore, it is sufficient to get a contradiction after replacing $\pi \colon X \to Y$, $D$, and $S$ with $\pi^{-1}(H) \to H$, $D|_{\pi^{-1}(H)}$, and $S \cap \pi^{-1}(H)$ respectively. 

Repeating this discussion, we may assume ${\rm dim}\,\pi({\rm Supp}\,E)=0$. 
\end{step2}

\begin{step2}\label{lem--negativity-veryexc-step2}
In this step, we treat the case ${\rm dim}\,Y \geq 2$. 

We set $d:={\rm dim}\,X$. 
We take a sufficiently general hyperplane sections $A_{1},\, \cdots, A_{d-2}$ of $X$. 
We put $X'=A_{1} \cap \cdots \cap A_{d-2}$, and let $\pi' \colon X' \to Y'$ be the Stein factorization of $X' \hookrightarrow X \to Y$. 
Then $\pi'$ is a bimeromorphism between normal analytic surfaces. 
We put $D'=D|_{X'}$, $E'=E|_{X'}$, and $F'=F|_{X'}$. 
Then $E'$ and $F'$ are the effective and the negative part of $D'$, respectively, and $E'$ is $\pi'$-exceptional since ${\rm dim}\,\pi({\rm Supp}\,E)=0$. 
Moreover, $S \cap X'$ is a union of points, and we have $(D'\cdot E_{i})\geq 0$ for every component $E_{i}$ of $E'$.
This contradicts the analytic analog of Hodge index theorem \cite[p367]{grauert-hodgeindex}. 

We finish the case ${\rm dim}\,Y \geq 2$. 
\end{step2}

\begin{step2}\label{lem--negativity-veryexc-step3}
In this step, we treat the case ${\rm dim}\,Y =1$. 

In this case, $Y$ is a non-singular curve. 
We put $y:=\pi({\rm Supp}\,E)$, and we define
$$t:={\rm inf}\{u \in \mathbb{R}_{\geq0}| -D+u\pi^{*}y\geq 0\}.$$
Then $-D+t \pi^{*}y \geq 0$, and there exist a component $P$ of $E$ and a component $Q$ of $\pi^{*}y$ such that $P \cap Q \neq \emptyset$, ${\rm coeff}_{P}(-D+t \pi^{*}y)=0$, and ${\rm coeff}_{Q}(-D+t \pi^{*}y)>0$. 
By cutting out $P$ with sufficiently general hyperplane sections, we can find a curve $C \subset P$ such that
$$(-D+t \pi^{*}y)\cdot C=((F-E+t \pi^{*}y)\cdot C)>0.$$
Now recall the hypothesis that for any curve $C \subset {\rm Supp}\,E$ contained in a fiber of $\pi$, if $C \not\subset S$ then $(D\cdot C)\geq 0$.
Since $\pi(C)$ is a point and $C$ passes through a very general point, we have
$$(-D+t \pi^{*}y)\cdot C=(-D\cdot C)\leq 0.$$
Thus, we get a contradiction. 
\end{step2}
By the above argument, we get a contradiction. 
Therefore, Lemma \ref{lem--negativity-veryexc} holds.
\end{proof}

\begin{cor}\label{cor--negativity-veryexc-2}
Let $\pi \colon X \to Y$ be a contraction of normal analytic varieties and $W \subset Y$ a subset. 
Let $D=E-F$ be an $\mathbb{R}$-Cartier divisor on $X$, where $E$ and $F$ are effective $\mathbb{R}$-divisors having no common components. 
Suppose that $E$ is very exceptional over $Y$ and there exist an $\mathbb{R}$-Cartier divisor $A$ on $X$ and a sequence of non-negative real numbers $\{\epsilon_{i}\}_{i \geq 0}$ such that ${\rm lim}_{i \to \infty}\epsilon_{i}=0$ and every $D+\epsilon_{i} A$ is movable over an open neighborhood of $W$.
Then there exists an open subset $U \subset Y$ containing $W$ such that $-D|_{\pi^{-1}(U)} \geq0$. 
\end{cor}

\begin{proof}
For any $y \in W$, if there is an open subset $U_{y} \subset Y$ containing $y$ such that $-D|_{\pi^{-1}(U_{y})} \geq0$, then $U:=\bigcup_{y \in W}U_{y}$ is the desired open subset. 
Thus, we may assume that $W$ is a point. 
By shrinking $Y$ around $W$, we may assume that $Y$ is Stein and there is a Zariski open subset $Y' \subset Y$ such that the image of all components of $D':=D|_{\pi^{-1}(Y')}$ contain $W$. 

By definition of movable divisors, there is an open subset $U_{i} \subset Y'$ containing $W$ such that any irreducible component of ${\boldsymbol{\rm Bs}}|(D+\epsilon_{i}A)|_{\pi^{-1}(U_{i})}/U_{i}|_{\mathbb{R}}$ has codimension $\geq 2$. 
Since $Y$ is Stein, by Proposition \ref{prop--r-base-locus} and the fact that the image of all components of $D'$ contain $W$, any component of $E':=E|_{\pi^{-1}(Y')}$  is not contained in ${\boldsymbol{\rm Bs}}|D+\epsilon_{i}A|_{\mathbb{R}}$. 
Therefore, we can find an analytically meagre subset $S \subset {\rm Supp}\,E'$ such that any point $x \in {\rm Supp}\,E'\setminus S$ is not contained in any ${\boldsymbol{\rm Bs}}|D+\epsilon_{i}A|_{\mathbb{R}}$. 
From this, for any curve $C \subset {\rm Supp}\,E'$ contained in a fiber of $\pi$, if $C \not\subset S$ then $(D\cdot C)={\rm lim}_{i \to \infty}(D+\epsilon_{i}A)\cdot C\geq 0$. 
By Lemma \ref{lem--negativity-veryexc}, we have $-D|_{\pi^{-1}(Y')} \geq0$. 
Thus, Corollary \ref{cor--negativity-veryexc-2} holds. 
\end{proof}

\begin{cor}\label{cor--negativity-veryexc-nef}
Let $\pi \colon X \to Y$ be a contraction of normal analytic varieties and $W \subset Y$ a subset. 
Let $D=E-F$ be an $\mathbb{R}$-Cartier divisor on $X$, where $E$ and $F$ are effective $\mathbb{R}$-divisors having no common components. 
Suppose that $E$ is very exceptional over $Y$ and $D$ is nef over $W$. 
Then, there is an open subset $U \supset W$ of $Y$ such that $-D|_{\pi^{-1}(U)} \geq0$. 
In particular, for any bimeromorphism $\pi\colon X \to Y$ of normal analytic varieties and any $\mathbb{R}$-Cartier divisor $D$ on $X$ such that $D$ is nef over $W$ and $-\pi_{*}D \geq 0$, there exists an open subset $U \subset Y$ containing $W$ such that  such that $-D|_{\pi^{-1}(U)} \geq0$. 
\end{cor}

\begin{proof}
For any $y \in W$, if there is an open subset $U_{y} \subset Y$ containing $y$ such that $-D|_{\pi^{-1}(U_{y})} \geq0$, then $U:=\bigcup_{y \in W}U_{y}$ is the desired open subset. 
Thus, we may assume that $W$ is a point. 
Let $A$ be a $\pi$-ample Cartier divisor on $X$. 
By \cite[Lemma 4.10]{fujino-analytic-bchm}, for any $\epsilon \in \mathbb{R}_{>0}$, there is a Stein open subset $U_{\epsilon}\subset Y$ containing $W$ such that $D+\epsilon A$ is ample over $U_{\epsilon}$. 
In particular, ${\boldsymbol{\rm Bs}}|(D+\epsilon A)|_{\pi^{-1}(U_{\epsilon})}/U_{\epsilon}|_{\mathbb{R}}$ is empty. 
Corollary \ref{cor--negativity-veryexc-nef} follows from this fact and Corollary \ref{cor--negativity-veryexc-2}. 
\end{proof}

\subsection{Singularities of pairs} 

In this subsection, we define singularities of pairs. 
We recommend the reader to read \cite[Section 3]{fujino-analytic-bchm} for basic properties of pairs. 

Let $X$ be a normal analytic variety. 
The {\em canonical sheaf} $\omega_{X}$ of $X$ is the unique reflexive sheaf whose restriction to the non-singular locus $X_{\rm sm}$ is isomorphic to $\Omega_{X_{\rm sm}}^{{\rm dim}X}$. 
Let $\Delta$ be an $\mathbb{R}$-divisor on $X$. 
We say that $K_{X}+\Delta$ is $\mathbb{R}$-Cartier at $x \in X$ if there exist an open neighborhood $U_{x}$ of $x$ and a Weil divisor $K_{U_{x}}$ on $U_{x}$ with $\mathcal{O}_{U_{x}}(K_{U_{x}}) \simeq \omega_{X}|_{U_{x}}$ such that $K_{U_{x}}+\Delta|_{U_{x}}$ is $\mathbb{R}$-Cartier at $x$. 
We simply say that $K_{X}+\Delta$ is $\mathbb{R}$-Cartier when $K_{X}+\Delta$ is $\mathbb{R}$-Cartier at any point $x \in X$. 
Unfortunately, we can not define $K_{X}$ globally with $\mathcal{O}_{X}(K_{X}) \simeq \omega_{X}$. 
It only exists locally on $X$. 
However, we use the symbol $K_{X}$ as a formal divisor class with
an isomorphism $\mathcal{O}_{X}(K_{X}) \simeq \omega_{X}$ and call it the {\em canonical divisor} of $X$ if there is no danger of confusion. 

A {\em pair} $(X,\Delta)$ consists of a normal analytic variety $X$ and an effective $\mathbb{R}$-divisor $\Delta$ on $X$ such that $K_{X}+\Delta$ is $\mathbb{R}$-Cartier. 
Let $(X,\Delta)$ be a pair. 
Let $f \colon Y \to X$ be a proper bimeromorphism from a normal analytic variety $Y$.
We take a Stein open subset $U$ of $X$ where $K_{U}+\Delta|_{U}$ is a well defined $\mathbb{R}$-Cartier divisor on $U$. 
Then we
can define $K_{f^{-1}(U)}$ and $K_{U}$ such that $f_{*}K_{f^{-1}(U)}=K_{U}$, and we can write
$K_{f^{-1}(U)}=f^{*}(K_{U}+\Delta|_{U})+E_{U}$
as usual. 
Note that $E_{U}$ is a well defined $\mathbb{R}$-divisor on $f^{-1}(U)$ such that $f_{*}E_{U}=-\Delta|_{U}$.
Then we have 
$$K_{Y}=f^{*}(K_{X}+\Delta)+\sum_{E}a(E,X,\Delta)E$$
such that $\bigl( \sum_{E}a(E,X,\Delta)E \bigr)|_{U}=E_{U}$. 
We note that $\sum_{E}a(E,X,\Delta)E$ is a globally well defined although $K_{X}$ and $K_{Y}$ are well defined only locally. 

A pair $(X,\Delta)$ is called a {\em log canonical} ({\em lc}, for short) {\em pair} if $a(E,X,\Delta) \geq -1$ for any $f \colon Y \to X$ and every prime divisor divisor $E$ on $Y$. 
A pair $(X,\Delta)$ is called a {\em Kawamata log terminal} ({\em klt}, for short) {\em pair} if $a(E,X,\Delta) > -1$ for any $f \colon Y \to X$ and every prime  divisor $E$ on $Y$. 
A pair $(X,\Delta)$ is called a {\em divisorial log terminal} ({\em dlt}, for short) {\em pair} if there exists a log resolution $f\colon Y \to X$ of $(X,\Delta)$ such that $a(E,X,\Delta)>-1$ for every $f$-exceptional prime divisor $E$. 
The image of $E$ with $a(E,X,\Delta) = -1$ for some $f \colon Y \to X$ such that $(X,\Delta)$ is lc around general points of $f(E)$ is called an {\em lc center} of $(X,\Delta)$.

\begin{lem}\label{lem--finite-lccenter}
Let $(X,\Delta)$ be an lc pair and $K \subset X$ a compact subset. 
Then there are only finitely many lc centers of $(X,\Delta)$ intersecting $K$.   
\end{lem}

\begin{proof}
Let $f \colon Y \to X$ be a log resolution of $(X,\Delta)$. 
We can write
$$K_{Y}+\Gamma=f^{*}(K_{X}+\Delta)+E$$
for some effective $\mathbb{R}$-divisors $\Gamma$ and $E$ on $Y$ having no common components. 
Since any lc center of $(X,\Delta)$ is the image of an lc center of $(Y,\Gamma)$, it is sufficient to the finiteness of the lc centers of $(Y,\Gamma)$ intersecting $f^{-1}(K)$. 
By replacing $(X,\Delta)$ and $K$ with $(Y,\Gamma)$ and $f^{-1}(K)$ respectively, we may assume that $(X,\Delta)$ is log smooth. 

We prove the log smooth case by induction on a dimension of $X$. 
If ${\rm dim}\,X=1$, then the assertion is clear since ${\rm Supp}\,\Delta \cap K$ contains only finitely many points. 
In the general case, there are only finitely components $T_{1},\,\cdots,\,T_{m}$ of $\lfloor \Delta \rfloor$ intersecting $K$. 
By the induction hypothesis, for every $1 \leq i \leq m$ the lc pair $(T_{i},(\Delta-T_{i})|_{T_{i}})$ has only finitely many lc centers intersecting $T_{i} \cap K$. 
Since any lc center of $(X,\Delta)$ is $T_{1},\,\cdots,\,T_{m}$ or an lc center of $(T_{i},(\Delta-T_{i})|_{T_{i}})$ for some $i$, Lemma \ref{lem--finite-lccenter} holds.  
\end{proof}

\begin{thm}\label{thm--dlt-restriction}
Let $\pi \colon X \to Y$ be a projective morphism from a normal analytic variety $X$ to an analytic space $Y$, and let $W \subset Y$ be a compact subset such that $W \cap Z$ has only finitely many connected components for any analytic subset $Z$ which is defined over an open neighborhood of $W$ (see also the condition $({\rm P}4)$ in \cite{fujino-analytic-bchm}). 
Let $(X,\Delta)$ be a dlt pair.
Then there exists an open subset $Y' \subset Y$ containing $W$ such that every connected component $X'$ of $\pi^{-1}(Y')$ satisfying the following conditions.
\begin{itemize}
\item
$(X',\Delta|_{X'})$ is a dlt pair, and
\item
for any lc center $S'$ of $(X',\Delta|_{X'})$ such that $\pi(S') \cap W \neq \emptyset$ and any open subset $U \subset Y'$ containing $W$, there exists a unique lc center $S_{U}$ of $(\pi^{-1}(U), \Delta|_{\pi^{-1}(U)})$ such that $S_{U} \subset S'$ and $\pi(S_{U})\cap W \neq \emptyset$. 
\end{itemize} 
In particular, for any open subset $U \subset Y'$ containing $W$, there is a natural bijection between the set of the lc centers $S'$ of $(X',\Delta|_{X'})$ satisfying $\pi(S') \cap W \neq \emptyset$ and the set of the lc centers $S_{U}$ of $(\pi^{-1}(U), \Delta|_{\pi^{-1}(U)})$ satisfying $\pi(S_{U})\cap W \neq \emptyset$. 
\end{thm}

\begin{proof}
By Lemma \ref{lem--finite-lccenter}, there are finitely many lc centers of $(X,\Delta)$ whose images on $Y$ by $\pi$ intersect $W$. 
Let $S$ be $X$ or an lc center of $(X,\Delta)$ such that $\pi(S) \cap W \neq \emptyset$, and let $\pi_{S}\colon S \to Y$ be the induced morphism. 
Then $S$ is a normal analytic variety. 
By applying Lemma \ref{lem--shrink-irreducible} to $\pi_{S}$ and $W$, we obtain an open subset $Y'_{S}$ containing $W$ such that if $\pi^{-1}(Y'_{S}) =\coprod_{\lambda \in \Lambda}S'_{\lambda}$ is the decomposition of $\pi^{-1}(Y'_{S})$ into connected components $S'_{\lambda}$, then there are only finitely many connected components, which we denote by $S'_{1},\cdots,\, S'_{l}$, such that their images on $Y$ by $\pi$ intersect $W$ and $S'_{1},\cdots,\, S'_{l}$ satisfy the following properties.
\begin{itemize}
\item
$S'_{i}$ is a normal complex variety for all $1 \leq i \leq l$, and
\item
for any $1 \leq i \leq l$ and open subset $U \subset Y'_{S}$ containing $W$, there exists a unique connected component $S_{U}^{(i)}$ of $\pi^{-1}(U) \cap S'_{i}$ such that $\pi(S_{U}^{(i)}) \cap W \neq \emptyset$, and furthermore, $S_{U}^{(i)}$ is a normal analytic variety. 
\end{itemize} 
We set $Y':=\bigcap_{S}Y'_{S}$, where $S$ runs over $X$ and lc centers of $(X,\Delta)$ whose images on $Y$ by $\pi$ intersect $W$. 
Then $Y'$ is open and $Y' \supset W$. 
Is is easy to check that the above two properties imply the two conditions of Theorem \ref{thm--dlt-restriction}. 
Note that for any open subset $U \subset Y$, an lc center of $(\pi^{-1}(U), \Delta|_{\pi^{-1}(U)})$ can be written as an irreducible component of the intersection of $\pi^{-1}(U)$ and an lc center of $(X,\Delta)$. 
Thus, $Y'$ is the desired open subset. 
\end{proof}

\begin{lem}\label{lem--dlt-perturbation-coefficients}
Let $\pi \colon X \to Y$ be a contraction from a normal analytic variety $X$ to a Stein space $Y$, and let $(X,\Delta)$ be a dlt pair. 
Let $W \subset Y$ be a compact subset. 
Let $H$ be a $\pi$-ample $\mathbb{R}$-divisor on $X$. 
Then after shrinking $Y$ around $W$, we can find a klt pair $(X,B)$ such that $K_{X}+B\sim_{\mathbb{R}}K_{X}+\Delta+H$. 
\end{lem}

\begin{proof}
By using \cite[Lemma 3.9]{fujino-analytic-bchm}, the argument in the algebraic case works. 
\end{proof}

We close this subsection with the definition of log smooth models.

\begin{defn}[Log smooth model, {\cite[Definition 2.9]{has-trivial}}, cf.~{\cite[Definition 2.3]{birkar-flip}}]\label{defnlogsmoothmodel}
Let $(X,\Delta)$ be an lc pair and $f:Y \to X$ a log resolution of $(X,\Delta)$. 
Let $\Gamma$ be a boundary $\mathbb{R}$-divisor on $Y$ such that $(Y,\Gamma)$ is log smooth. 
Then $(Y,\Gamma)$ is a {\it log smooth model} of $(X,\Delta)$ if we write 
$$K_{Y}+\Gamma=f^{*}(K_{X}+\Delta)+F, $$
then
\begin{enumerate}
\item[(i)]
$F$ is an effective $f$-exceptional divisor, and 
\item[(ii)] 
every $f$-exceptional prime divisor $E$ satisfying $a(E,X,\Delta)>-1$ is a component of $F$ and $\Gamma-\lfloor \Gamma \rfloor$.  
\end{enumerate}
\end{defn}

\subsection{Abundant divisor}

In this subsection, we define the property of being abundant and the property of being log abundant with respect to an lc pair. 

\begin{defn}[Invariant Iitaka dimension, cf.~{\cite[Definition 2.2.1]{choi}}]\label{defn--inv-iitaka-dim}
Let $\pi \colon X \to Y$ be a projective morphism from a normal analytic variety $X$ to an analytic variety $Y$, and let $D$ be an $\mathbb{R}$-Cartier divisor on $X$. 
Then the {\em relative invariant Iitaka dimension} of $D$, denoted by $\kappa_{\iota}(X/Y,D)$, is defined as follows: 
If there is an $\mathbb{R}$-divisor $E\geq 0$ such that $D\sim_{\mathbb{R},Y}E$ then we set $\kappa_{\iota}(X/Y,D)=\kappa_{\iota}(F,D|_{F})$, where $F$ is an analytically sufficiently general fiber of the Stein factorization of $\pi$, and otherwise we set $\kappa_{\iota}(X/Y,D)=-\infty$. 
As the algebraic case (\cite{choi}), this definition is independent of the choice of $E$ and $F$, hence $\kappa_{\iota}(X/Y,D)$ is well defined. 
\end{defn}

\begin{defn}[Numerical dimension, cf.~{\cite{nakayama}}]\label{defn--num-dim}
Let $\pi \colon X \to Y$ be a projective morphism from a normal analytic variety $X$ to an analytic variety $Y$, and let $D$ be an $\mathbb{R}$-Cartier divisor on $X$. 
Then, the {\em relative numerical dimension} of $D$ over $Y$ is defined by the numerical dimension $\kappa_{\sigma}(F,D|_{F})$ defined in \cite[V, 2.5 Definition]{nakayama}, where $F$ is an analytically sufficiently general fiber of the Stein factorization of $\pi$. 
As the algebraic case \cite[2.2]{has-mmp}, $\kappa_{\sigma}(F,D|_{F})$ is independent of the choice of $F$.  
 In this paper, we denote $\kappa_{\sigma}(F,D|_{F})$ by $\kappa_{\sigma}(X/Y,D)$. 
\end{defn}

\begin{defn}[Abundant divisor, log abundant divisor, cf.~{\cite[Definition 2.16]{has-finite}}]\label{defn--abund}
Let $\pi\colon X\to Y$ be a projective morphism from a normal analytic variety $X$ to an analytic variety $Y$, and let $D$ be an $\mathbb{R}$-Cartier divisor on $X$. 
We say that $D$ is $\pi$-{\em abundant} (or {\em abundant over} $Y$) if the equality $\kappa_{\iota}(X/Y,D)=\kappa_{\sigma}(X/Y,D)$ holds. 

Let $\pi\colon X\to Y$ and $D$ be as above. 
Let $W \subset Y$ be a subset and $(X,\Delta)$ an lc pair. 
We say that $D$ is $\pi$-{\em log abundant} (or {\em log abundant over} $Y$) {\em around} $W$ {\em with respect to} $(X,\Delta)$ if $D$ satisfies the following: Let $S$ be $X$ or an lc center of $(X,\Delta)$ with the normalization $S^{\nu}\to S$. 
If $\pi(S) \cap W \neq \emptyset$, then the pullback $D|_{S^{\nu}}$ is abundant over $Y$.

Let $\pi\colon X\to Y$ be a projective morphism from a normal analytic variety $X$ to an analytic variety $Y$, and let $W \subset Y$ be a subset. 
Let $(X,\Delta)$ an lc (resp.~dlt) pair. 
When $K_{X}+\Delta$ is log abundant over $Y$ around $W$ with respect to $(X,\Delta)$, we call $(X,\Delta)$ a {\em log abundant lc} (resp.~{\em log abundant dlt}) {\em pair over} $Y$ {\em around} $W$ or we say that $(X,\Delta)$ is {\em log abundant over} $Y$ around $W$.  
\end{defn}

\begin{lem}[cf.~{\cite[Lemma 2.11]{hashizumehu}}]\label{lem--iitakafib} 
Let $\pi\colon X\to Y$ be a projective morphism from a normal analytic variety $X$ to an analytic variety $Y$. 
Let $(X,\Delta)$ be an lc pair with an $\mathbb{R}$-divisor $\Delta$. 
Suppose that $K_{X}+\Delta$ is abundant over $Y$ and there is an effective $\mathbb{R}$-divisor $D\sim_{\mathbb{R},\,Y}K_{X}+\Delta$. 
Let $X\dashrightarrow V$ be the Iitaka fibration over $Y$ associated to $D$. 
Pick a log resolution $f\colon \overline{X} \to X$ of $(X,\Delta)$ such that the induced meromorphic map $\overline{X} \dashrightarrow V$ is a morphism, and let $(\overline{X},\overline{\Delta})$ be a projective lc pair such that we can write $K_{\overline{X}}+\overline{\Delta}=f^{*}(K_{X}+\Delta)+E$ for an effective $f$-exceptional $\mathbb{R}$-divisor $E$. 
Then, we have $\kappa_{\sigma}(\overline{X}/V,K_{\overline{X}}+\overline{\Delta})=0$. 
\end{lem}

\begin{proof}
By the discussion in Subsection \ref{subsec--iitaka-fib} and considering the fibers over an analytically sufficiently general point of $Y$, we may reduce the lemma to the algebraic case, and the the lemma follows from \cite[Lemma 2.11]{hashizumehu}. 
\end{proof}

\subsection{Bimeromorphic map}
The following lemma is well known to experts but useful to discuss steps of MMP.

\begin{lem}\label{lem--basic-1}
Let $f\colon X \to Y$ and $f'\colon X' \to Y$ be projective morphisms from normal analytic varieties to an analytic space $Y$, and let $\phi \colon X\dashrightarrow X'$ be a bimeromorphic contraction over $Y$. 
Let $D$ be an $\mathbb{R}$-Cartier divisor on $X$ such that $-D$ is ample over $Y$. 
Suppose that $\phi_{*}D$ is  $\mathbb{R}$-Cartier and ample over $Y$. 
Let $g \colon W \to X$ and $g' \colon W \to X'$ be the normalization of the graph of $\phi$. 
We put 
$$E:=g^{*}D-g'^{*}(\phi_{*}D).$$
Then $E$ is effective and ${\rm Supp}\,E={\rm Ex}(g) \cup {\rm Ex}(g')$. 
\end{lem}

\begin{proof}
Note that $g$ and $g'$ are bimeromorphisms. 
The effectivity of $E$ follows from the negativity lemma (Corollary \ref{cor--negativity-veryexc-nef}). 
We also have ${\rm Supp}\,E \subset {\rm Ex}(g) \cup {\rm Ex}(g')$. 
Therefore, it is sufficient to prove ${\rm Supp}\,E \supset {\rm Ex}(g) \cup {\rm Ex}(g')$. 

Shrinking $Y$ if necessary, we may assume that $Y$ is a Stein space.
By Theorem \ref{thm--birat-basic}, all fibers of $g$ and $g'$ are connected. 
Pick any $x \in g({\rm Ex}(g))$. 
Then there exists a Zariski open subset $U$ of $g^{-1}(x)$ such that for any $u \in U$, by cutting out by hyperplane sections on $W$ passing through $u$, we can find a curve $C \subset g^{-1}(x)$ such that $C \ni u$ and $g'(C)$ is not a point. 
Here, we used Bertini type theorem for Cartier divisors on $W$ which are very ample over $Y$. 
Then 
$$(E \cdot C)=\bigl(g^{*}D-g'^{*}(\phi_{*}D)\bigr)\cdot C=-(\phi_{*}D\cdot g'(C))<0.$$
Therefore ${\rm Supp}\,E \ni u$. 
Then ${\rm Supp}\,E \supset U$. 
Since ${\rm Supp}\,E \cap g^{-1}(x)$ is closed in $g^{-1}(x)$, we have ${\rm Supp}\,E \supset g^{-1}(x)$. 
Thus, ${\rm Supp}\,E \supset {\rm Ex}(g)$. 
Similarly, we have ${\rm Supp}\,E \supset {\rm Ex}(g')$. 
From these facts, we have ${\rm Supp}\,E \supset {\rm Ex}(g) \cup {\rm Ex}(g')$. 
\end{proof}

\begin{lem}\label{lem--basic-2}
Let $f\colon X \to Y$ and $f'\colon X' \to Y$ be projective morphisms from normal analytic varieties to an analytic space $Y$, and let $\phi \colon X\dashrightarrow X'$ be a bimeromorphic contraction over $Y$. 
Let $(X,\Delta)$ and $(X',\Delta')$ be dlt pairs such that $\Delta'=\phi_{*}\Delta$. 
Suppose that $-(K_{X}+\Delta)$ and $K_{X'}+\Delta'$ are ample over $Y$. 
Then the following hold.
\begin{enumerate}[(i)]
\item\label{lem--basic-2-(i)}
For any prime divisor $P$ over $X$, the inequality $a(P,X,\Delta) \leq a(P,X',\Delta')$ holds, and the equality holds if and only if the center of $P$ on $X$, denoted by $c_{X}(P)$, intersects the biholomorphic locus of $\phi$.
\item\label{lem--basic-2-(ii)}
Let $S$ (resp.~$S'$) be an lc center of $(X,\Delta)$ (resp.~$(X',\Delta')$) such that $S$ intersects the biholomorphic locus of $\phi$ and $\phi|_{S}$ defines a bimeromorphic map from $S$ to $S'$. 
Let $(S,\Delta_{S})$ and $(S',\Delta_{S'})$ be dlt pairs defined by adjunctions. 
If $\phi|_{S}\colon S \dashrightarrow S'$ is a biholomorphism over $Y$ and $(\phi|_{S})_{*}\Delta_{S}=\Delta_{S'}$, then there exists a Zariski open subset $U\subset X$ containing $S$ such that $\phi|_{U}$ is a biholomorphism to its image. 
\end{enumerate} 
\end{lem}

\begin{proof}
$g \colon W \to X$ and $g' \colon W \to X'$ be the normalization of the graph of $\phi$. 
We put 
$$E:=g^{*}(K_{X}+\Delta)-g'^{*}(K_{X'}+\Delta').$$
By Lemma \ref{lem--basic-1},  $E \geq 0$ and ${\rm Supp}\,E={\rm Ex}(g) \cup {\rm Ex}(g')$. 
Thus $a(P,X,\Delta) \leq a(P,X',\Delta')$ for all prime divisor $P$ over $X$, and the equality holds if and only if 
$$c_{W}(P)\not\subset {\rm Supp}E\,= {\rm Ex}(g) \cup {\rm Ex}(g')$$
This implies the first statement. 

Let $(S,\Delta_{S})$ and $(S',\Delta_{S'})$ be the dlt pairs in the second statement. 
Then there is an analytic subvariety $T$ of $W$ such that $g|_{T}\colon T \to S$ and $g'|_{T}\colon T \to S'$ are bimeromorphic morphism. 
Then
$$E|_{T}=g|_{T}^{*}(K_{S}+\Delta_{S})-g'|_{T}^{*}(K_{S'}+\Delta_{S'}).$$
If $\phi|_{S}\colon S \dashrightarrow S'$ is a biholomorphism over $Y$ and $(\phi|_{S})_{*}\Delta_{S}=\Delta_{S'}$, then all prime divisor $Q$ over $S$ satisfy 
$$a(Q,S,\Delta_{S})=a(Q,S',\Delta_{S'}).$$
This shows $E|_{T}=0$. 
Therefore, $T$ is disjoint from ${\rm Ex}(g) \cup {\rm Ex}(g')$, and therefore $S$ is disjoint from $g\bigl({\rm Ex}(g) \cup {\rm Ex}(g')\bigr)$. 
Then 
$$U:=X\setminus g\bigl({\rm Ex}(g) \cup {\rm Ex}(g')\bigr)$$
is the desired Zariski open subset. 
\end{proof}

\section{Fundamental results of minimal model program}\label{sec3}

\subsection{Models}
In this subsection, we define some models. 

\begin{defn}[Models]\label{defn--models}
Let $\pi \colon X \to Y$ be a projective morphism from a normal analytic variety $X$ to an analytic space $Y$, and let $(X,\Delta)$ be an lc pair. 
Let $W \subset Y$ be a subset. 
Let $\pi' \colon X' \to Y$ be a projective morphism from a normal analytic variety $X'$ to $Y$, and let $\phi \colon X \dashrightarrow X'$ be a bimeromorphic map over $Y$. 
We set $\Delta':=\phi_{*}\Delta+E$, where $E$ is the sum of all $\phi^{-1}$-exceptional prime divisors with the coefficients one. 
When $K_{X'}+\Delta'$ is $\mathbb{R}$-Cartier, we say that $(X',\Delta')$ is a {\em log birational model of $(X,\Delta)$ over $Y$}. 

A log birational model of $(X,\Delta)$ over $Y$ is called a {\em weak log canonical model} ({\em weak lc model}, for short) {\em of $(X,\Delta)$ over $Y$ around $W$} if
\begin{itemize}
\item
$K_{X'}+\Delta'$ is nef over $W$, and
\item
for any prime divisor $P$ on $X$ that is exceptional over $X'$, we have 
$$a(P,X,\Delta) \leq a(P,X',\Delta').$$ 
\end{itemize} 
A weak lc model $(X',\Delta')$ of $(X,\Delta)$ over $Y$ around $W$ is a {\em log minimal model} if 
\begin{itemize}
\item
the above inequality on discrepancies is strict. 
\end{itemize}
A log minimal model $(X',\Delta')$ of $(X,\Delta)$ over $Y$ around $W$ is called a {\em good minimal model} if $K_{X'}+\Delta'$ is semi-ample over a neighborhood of $W$. 

Suppose that $W \subset Y$ is a compact subset such that $\pi \colon X \to Y$ and $W$ satisfy (P). 
Then, a log birational model $(X',\Delta')$ of $(X,\Delta)$ over $Y$ is called a {\em Mori fiber space over $Y$ around $W$} if there exists a contraction $X' \to Z$ over $Y$ such that
\begin{itemize}
\item
${\rm dim}\,X'>{\rm dim}\,Z$, 
\item
$\rho(X'/Y;W)-\rho(Z/Y;W)=1$ and $-(K_{X'}+\Delta')$ is ample over $Z$, and 
\item
for any prime divisor $P$ over $X$, we have 
$$a(P,X,\Delta) \leq a(P,X',\Delta'),$$
and the strict inequality holds if $P$ is a $\phi$-exceptional prime divisor on $X$.  
\end{itemize}

If $W=\emptyset$, then we formally define $(X,\Delta)$ itself to be a log minimal model (and thus a good minimal model) of $(X,\Delta)$ over $Y$ around $W$. 
Let $(\widetilde{X},\widetilde{\Delta})$ be a disjoint union $\coprod_{\lambda \in \Lambda}(\widetilde{X}_{\lambda},\widetilde{\Delta}_{\lambda})$ of lc pairs, $\widetilde{\pi} \colon \widetilde{X} \to Y$ a projective morphism to an analytic space $Y$, and $W \subset Y$ a subset. 
Then a {\rm log birational model of $(\widetilde{X},\widetilde{\Delta})$ over $Y$} (resp.~a {\em log minimal model of $(\widetilde{X},\widetilde{\Delta})$ over $Y$ around $W$}) is a disjoint union of lc pairs $\coprod_{\lambda \in \Lambda}(\widetilde{X}'_{\lambda},\widetilde{\Delta}'_{\lambda})$ such that $(\widetilde{X}'_{\lambda},\widetilde{\Delta}'_{\lambda})$ is a log birational model of $(\widetilde{X}_{\lambda},\widetilde{\Delta}_{\lambda})$ over $Y$ (resp.~a log minimal model of $(\widetilde{X}_{\lambda},\widetilde{\Delta}_{\lambda})$ over $Y$ around $W$) for all $\lambda \in \Lambda$. 
\end{defn}

\begin{rem}
With notation as in Definition \ref{defn--models}, let $(X',\Delta')$ be a weak lc model of an lc pair $(X,\Delta)$ over $Y$ around $W$. 
By the negativity lemma (Corollary \ref{cor--negativity-veryexc-nef}), we can check that 
$$a(P,X,\Delta) \leq a(P,X',\Delta')$$ 
for all prime divisors $P$ over $X$ whose image on $Y$ intersects $W$.
\end{rem}

\begin{lem}\label{lem--two-logminmodel}
Let $\pi \colon X \to Y$ be a projective morphism from a normal analytic variety $X$ to an analytic space $Y$, and let $W \subset Y$ be a compact subset. 
Let $(X,\Delta)$ be an lc pair. 
Let $(X_{1},\Delta_{1})$ and $(X_{2},\Delta_{2})$ be two weak lc models of $(X,\Delta)$ over $Y$ around $W$. 
Let $f_{1}\colon X' \to X_{1}$ and $f_{2}\colon X' \to X_{2}$ be the bimeromorphisms from a normal variety $X'$ that resolves the indeterminacy of the induced bimeromorphic map $X_{1}\dashrightarrow X_{2}$. 
Then, after shrinking $Y$ around $W$, we have $f_{1}^{*}(K_{X_{1}}+\Delta_{1})=f_{2}^{*}(K_{X_{2}}+\Delta_{2})$. 
\end{lem}

\begin{proof}
By replacing $X'$ if necessary, we may assume that the induced bimeromorphic map $f \colon X' \dashrightarrow X$ is a morphism. 
We set
$$E_{i}:=f^{*}(K_{X}+\Delta)-f_{i}^{*}(K_{X_{i}}+\Delta_{i})=\sum_{P}\bigl(a(P,X_{i},\Delta_{i})-a(P,X,\Delta)\bigr)P$$
for $i=1$ and $2$, where $P$ runs over prime divisors on $X'$. 
By the definition of weak lc models, $E_{i}$ is effective. 
When a prime divisor $P$ on $X'$ is not exceptional over $X$ or $X_{i}$, obviously ${\rm coeff}_{P}(E_{i})=0$. 
When $P$ is exceptional over $X$ but not exceptional over $X_{i}$, then $a(P,X_{i},\Delta_{i})=-1$. 
Then
$-1=a(P,X_{i},\Delta_{i}) \geq a(P,X,\Delta) \geq -1$ since $(X,\Delta)$ is lc. 
Therefore, we have ${\rm coeff}_{P}(E_{i})=0$. 
From these facts, $E_{i}$ is $f_{i}$-exceptional. 

Let $W_{1}$ (resp.~$W_{2}$) be the inverse image of $W$ to $X_{1}$ (resp.~$X_{2}$). 
Then the divisor
$$E_{1}-E_{2}=f_{2}^{*}(K_{X_{2}}+\Delta_{2})-f_{1}^{*}(K_{X_{1}}+\Delta_{1})$$
is nef over $W_{1}$ and we have $f_{1*}(E_{2}-E_{1})\geq 0$. 
Applying Corollary \ref{cor--negativity-veryexc-nef} to $f_{1}\colon X' \to X_{1}$ and $E_{1}-E_{2}$, we can find an open subset $U_{1}\supset W_{1}$ such that $(E_{2}-E_{1})|_{f^{-1}_{1}(U_{1})}\geq 0$. 
Therefore, after shrinking $Y$ around $W$, we have $E_{2}-E_{1}\geq 0$. 
By the similar argument, shrinking $Y$ around $W$, we may assume that $E_{1}-E_{2}\geq 0$. 
Then $E_{1}=E_{2}$, and therefore $f_{1}^{*}(K_{X_{1}}+\Delta_{1})=f_{2}^{*}(K_{X_{2}}+\Delta_{2}).$
\end{proof}

\begin{lem}\label{lem--exist-model-birat-I}
Let $\pi \colon X \to Y$ be a projective morphism from a normal analytic variety $X$ to a Stein space $Y$, and let $W \subset Y$ be a compact subset such that $\pi$ and $W$ satisfy (P). 
Let $(X,\Delta)$ be an lc pair and let $(X',\Delta')$ be an lc pair together with a projective bimeromorphism $f \colon X' \to X$ such that $f^{-1}\colon (X,\Delta) \dashrightarrow (X',\Delta')$ is a log birational model of $(X,\Delta)$ over $Y$. 
Then the following conditions are equivalent.
\begin{itemize}
\item
After shrinking $Y$ around $W$, $(X,\Delta)$ has a weak lc model over $Y$ around $W$.  
\item
After shrinking $Y$ around $W$, $(X',\Delta')$ has a weak lc model over $Y$ around $W$. 
\end{itemize}
Moreover, the same statement holds even if we replace weak lc model by log minimal model or good minimal model.
\end{lem}

\begin{proof}
The proof is the same as the algebraic case. 
Let $(\widetilde{X},\widetilde{\Delta})$ be an lc pair. 
By the negativity lemma (Corollary \ref{cor--negativity-veryexc-nef}), we can easily check that $(\widetilde{X},\widetilde{\Delta})$ is a weak lc model of $(X,\Delta)$ over $Y$ around $W$ if and only of $(\widetilde{X},\widetilde{\Delta})$ is a weak lc model of $(X',\Delta')$ over $Y$ around $W$. 
We omit the details.
\end{proof}

\subsection{Definition of log MMP}

In this subsection, we define a sequence of steps of a log MMP.  
Compared to the algebraic case, the definition is much more complicated.

\begin{defn}[Log MMP]\label{defn--mmp}
Let $\pi \colon X \to Y$ be a projective morphism from a normal analytic variety $X$ to a Stein space $Y$, and let $(X,\Delta)$ be an lc pair. 
Let $W \subset Y$ be a compact subset such that $\pi$ and $W$ satisfy (P).
\begin{enumerate}[(1)]
\item\label{defn--mmp-(1)}
A {\em step of a $(K_{X}+\Delta)$-MMP over $Y$ around $W$} is a diagram
 $$
\xymatrix{
(X,\Delta)\ar@{-->}[rr]^-{\phi}\ar[dr]^-{f}\ar[ddr]_-{\pi}&&(X',\Delta':=\phi_{*}\Delta)\ar[dl]_-{f'}\ar[ddl]^-{\pi'}\\
&Z\ar[d]\\
&Y
}
$$
consisting of projective morphisms such that 
\begin{itemize}
\item
$(X',\Delta')$ is an lc pair and $Z$ is a normal analytic variety,
\item
$\phi \colon X \dashrightarrow X'$ is a bimeromorphic contraction over $Y$ and $f\colon X \to Z$ and $f' \colon X' \to Z$ are bimeromorphisms, 
\item
$f$ is a contraction of a $(K_{X}+\Delta)$-negative extremal ray of $\overline{\rm NE}(X/Y;W)$, in particular, $\rho(X/Y;W)-\rho(Z/Y;W)=1$ (cf.~\cite[Theorem 1.2 (4)(iii)]{fujino-analytic-conethm}) and $-(K_{X}+\Delta)$ is ample over $Z$, and
\item
$K_{X'}+\Delta'$ is ample over $Z$. 
\end{itemize}
Let $H$ be an $\mathbb{R}$-Cartier divisor on $X$. 
A {\em step of a $(K_{X}+\Delta)$-MMP over $Y$ around $W$ with scaling of $H$} is the above diagram satisfying the above conditions and
\begin{itemize}
\item
$K_{X}+\Delta+t H$ is nef over $W$ for some $t \in \mathbb{R}_{> 0}$, and 
\item
if we put 
$$\lambda:={\rm inf}\{t \in \mathbb{R}_{\geq 0}\,|\, \text{$K_{X}+\Delta+tH$ {\rm is nef over} $W$}\},$$
then $(K_{X}+\Delta+\lambda H)\cdot C= 0$ for any curve $C \subset \pi^{-1}(W)$ contracted by $f$. 
\end{itemize}

\item\label{defn--mmp-(2)}
A {\em sequence of steps of a $(K_{X}+\Delta)$-MMP over $Y$ around $W$} is a pair of sequences $\{Y_{i}\}_{i \geq 0}$ and $\{\phi_{i}\colon X_{i} \dashrightarrow X'_{i}\}_{i\geq 0}$, where $Y_{i} \subset Y$ are Stein open subsets and $\phi_{i}$ are bimeomorphic contractions of normal analytic varieties over $Y_{i}$, such that
\begin{itemize}
\item
$Y_{i}\supset W$ and $Y_{i}\supset Y_{i+1}$ for every $i\geq 0$, 
\item
$X_{0}=\pi^{-1}(Y_{0})$ and $X_{i+1}=X'_{i}\times_{Y_{i}}Y_{i+1}$ for every $i\geq 0$, and
\item
if we put $\Delta_{0}:=\Delta|_{X_{0}}$ and $\Delta_{i+1}:=(\phi_{i*}\Delta_{i})|_{X_{i+1}}$ for every $i\geq 0$, then 
$$(X_{i},\Delta_{i}) \dashrightarrow (X'_{i},\phi_{i*}\Delta_{i})$$ 
is a step of a $(K_{X_{i}}+\Delta_{i})$-MMP over $Y_{i}$ around $W$.  
\end{itemize}
For the simplicity of notation, a sequence of steps of a $(K_{X}+\Delta)$-MMP over $Y$ around $W$ is denoted by
$$(X_{0},\Delta_{0})\dashrightarrow (X_{1},\Delta_{1})\dashrightarrow \cdots \dashrightarrow (X_{i},\Delta_{i})\dashrightarrow \cdots.$$

\item\label{defn--mmp-(3)}
Let $H$ be an $\mathbb{R}$-Cartier divisor on $X$. 
We put 
$$H_{0}:=H|_{X_{0}} \qquad {\rm and} \qquad H_{i+1}:=(\phi_{i*}H_{i})|_{X_{i+1}}$$
for each $i \geq 0$. 
Then a {\em sequence of steps of a $(K_{X}+\Delta)$-MMP over $Y$ around $W$ with scaling of $H$} is a sequence of steps of a $(K_{X}+\Delta)$-MMP over $Y$ around $W$
$$(X_{0},\Delta_{0})\dashrightarrow (X_{1},\Delta_{1})\dashrightarrow \cdots \dashrightarrow (X_{i},\Delta_{i})\dashrightarrow \cdots$$
with the data $\{Y_{i}\}_{i \geq 0}$ and $\{\phi_{i}\colon X_{i} \dashrightarrow X'_{i}\}_{i\geq 0}$ such that $(X_{i},\Delta_{i})\dashrightarrow (X'_{i},\phi_{i*}\Delta_{i})$ is a step of a $(K_{X_{i}}+\Delta_{i})$-MMP over $Y$ around $W$ with scaling of $H_{i}$ for every $i$.
\end{enumerate}
\end{defn}

\begin{defn}
With notation as in Definition \ref{defn--mmp}, let
$$(X_{0},\Delta_{0})\dashrightarrow (X_{1},\Delta_{1})\dashrightarrow \cdots \dashrightarrow (X_{i},\Delta_{i})\dashrightarrow \cdots$$
be a sequence of steps of a $(K_{X}+\Delta)$-MMP over $Y$ around $W$ defined with $\{Y_{i}\}_{i \geq 0}$ and $\{\phi_{i}\colon X_{i} \dashrightarrow X'_{i}\}_{i\geq 0}$. 
We say that the $(K_{X}+\Delta)$-MMP is {\em represented by bimeromorphic contractions over $Y$} if $Y_{i}=Y$ for all $i \geq 0$. 
\end{defn}

\begin{rem}[Reductions towards existence of a log minimal model]\label{rem--mmp-reduction-basic}
Let $\pi \colon X \to Y$ be a projective morphism from a normal analytic variety $X$ to a Stein space $Y$, and let $W \subset Y$ be a compact subset such that $\pi$ and $W$ satisfy (P). 
Let $(X,\Delta)$ be an lc pair. 
We are usually interested in the existence of an open subset $U \supset W$ of $Y$ and log minimal models or Mori fiber spaces of all the connected components of $(\pi^{-1}(U),\Delta|_{\pi^{-1}(U)})$ over $U$ around $W$. 
To find such an open subset and log minimal models or Mori fiber spaces, we often use the following reductions. 
 
Firstly, by \cite[Theorem 2.8 and Theorem 2.13]{fujino-analytic-bchm} and replacing $\pi \colon X \to Y$ with the Stein factorization, we may assume that $\pi$ is a contraction. 
We can easily check that every connected component of $W$ is Stein compact and satisfies the property (P4) in \cite{fujino-analytic-bchm}. 
Therefore, replacing $W$ by any connected component of $W$, we may assume that $W$ is connected. 
By the property (P3) for $W$, any open subset $U \supset W$ of $Y$ contains a connected Stein open subset $Y' \subset U$ containing $W$. 
Then $\pi^{-1}(Y')$ is a normal variety by Lemma \ref{lem--cont-shrink}, in particular $(\pi^{-1}(Y'), \Delta|_{\pi^{-1}(Y')})$ is an lc pair. 
Hence, the reduction enables us to freely shrink $Y$ around $W$ keeping the irreducibility of $X$. 

Secondly, taking a dlt blow-up of $(X,\Delta)$ after shrinking $Y$ around $W$, we may assume that $(X,\Delta)$ is dlt. 
By using Theorem \ref{thm--dlt-restriction} and shrinking $Y$ around $W$ again, we may assume that for any lc center $S$ of $(X,\Delta)$ such that $\pi(S) \cap W \neq \emptyset$ and any open subset $U \subset Y$ containing $W$, there exists a unique lc center $S_{U}$ of $(\pi^{-1}(U), \Delta|_{\pi^{-1}(U)})$ such that $S_{U} \subset S$ and $\pi(S_{U})\cap W \neq \emptyset$.  
By this reduction, we may freely shrink $Y$ around $W$ keeping the number of lc centers $S$ of $(X,\Delta)$ of fixed dimension such that $\pi(S)\cap W \neq \emptyset$. 

Finally, suppose that after the above two reductions we are given data $\{Y_{i}\}_{i \geq 0}$ and $\{\phi_{i}\colon X_{i} \dashrightarrow X'_{i}\}_{i\geq 0}$ of a sequence of steps of a $(K_{X}+\Delta)$-MMP over $Y$ around $W$. 
Let $\pi_{i} \colon X_{i} \to Y_{i}$ be the induced morphism. 
Let $\{Y'_{i}\}_{i \geq 0}$ be a sequence of a Stein open subset of $Y$ such that $W\subset Y'_{i} \subset Y_{i}$ and $Y'_{i+1} \subset Y'_{i}$ for all $i \geq 0$.  
Then 
$$\{Y'_{i}\}_{i \geq 0}\qquad {\rm and} \qquad \{\phi_{i}|_{X_{i}\times_{Y_{i}}Y'_{i}}\colon X_{i}\times_{Y_{i}}Y'_{i} \dashrightarrow X'_{i}\times_{Y_{i}}Y'_{i}\}_{i\geq 0}$$
are the data of a sequence of steps of a $(K_{X}+\Delta)$-MMP over $Y$ around $W$.  
However, we identify the new data with the originally given data because we are interested in the existence of a log minimal model or a Mori fiber space of $(X,\Delta)$ over $Y$ around $W$. 
In the rest of this paper, we will apply the above reductions to all $X_{i} \to Y_{i}$ and $(X_{i},\Delta_{i})$ without detailed explanation.  
Then, for every $i$, we may freely shrink $Y_{i}$ around $W$ keeping the irreducibility of $X_{i}$ and the number of lc centers $S_{i}$ of $(X_{i},\Delta_{i})$ of fixed dimension such that the image of $\pi_{i}(S_{i}) \cap W \neq \emptyset$.  
\end{rem}

\subsection{Basics of log MMP}

In this subsection, we collect basic properties of log MMP. 
Those results are well known in the algebraic case.

\begin{thm}\label{thm--mmp-Cartier}
Let $\pi \colon X \to Y$ be a contraction from a normal analytic variety $X$ to a Stein space $Y$, and let $W \subset Y$ be a connected compact subset such that $\pi$ and $W$ satisfy (P). 
Let $(X,\Delta)$ be an lc pair. 
Let
$$
\xymatrix{
(X,\Delta)\ar@{-->}[rr]^-{\phi}\ar[dr]_-{f}&&(X',\Delta':=\phi_{*}\Delta)\ar[dl]^-{f'}\\
&Z
}
$$
be a step of a $(K_{X}+\Delta)$-MMP over $Y$ around $W$. 
Let $D$ be a Cartier divisor on $X$ such that $(D \cdot C)=0$ for any curve $C \subset \pi^{-1}(W)$ contracted by $f$. 
Then, after shrinking $Y$ around $W$, the divisor $\phi_{*}D$ is Cartier.  
\end{thm}

\begin{proof}
By the ampleness of $-(K_{X}+\Delta)$ over $Z$ and \cite[Theorem 9.1]{fujino-analytic-conethm}, after shrinking $Y$ around $W$, we can find $m_{0} \in \mathbb{Z}_{>0}$ such that $mD$ is base point free for all $m \geq m_{0}$. 
Let $X \to Z_{1}$ and $X \to Z_{2}$ be the contractions over $Z$ induced by $(m_{0}+1)D$ and $(m_{0}+2)D$, respectively. 
Since $D \equiv_{W} 0$, the induced morphisms $Z_{1} \to Z$ and $Z_{2} \to Z$ are biholomorphisms near the inverse image of $W$. 
By shrinking $Y$ around $W$ if necessary, we may assume $Z_{1}=Z_{2}=Z$. 
Then there are Cartier divisors $H_{1}$ and $H_{2}$ on $Z$ such that $(m_{0}+1)D\sim f^{*}H_{1}$ and $(m_{0}+2)D\sim f^{*}H_{2}$. 
Then 
$$\phi_{*}D=(m_{0}+2)\phi_{*}D-(m_{0}+1)\phi_{*}D \sim f'^{*}H_{2}-f'^{*}H_{1}$$
which is Cartier. 
\end{proof}

\begin{thm}\label{thm--mmp-R-Cartier}
Let $\pi \colon X \to Y$ be a contraction from a normal analytic variety $X$ to a Stein space $Y$, and let $W \subset Y$ be a connected compact subset such that $\pi$ and $W$ satisfy (P). 
Let $(X,\Delta)$ be an lc pair.
Let
$$
\xymatrix{
(X,\Delta)\ar@{-->}[rr]^-{\phi}\ar[dr]_-{f}&&(X',\Delta':=\phi_{*}\Delta)\ar[dl]^-{f'}\\
&Z
}
$$
be a step of a $(K_{X}+\Delta)$-MMP over $Y$ around $W$. 
Then, for any $\mathbb{R}$-Cartier divisor $D$ on $X$, after shrinking $Y$ around $W$, it follows that $\phi_{*}D$ is a $\mathbb{R}$-Cartier divisor on $X'$. 
In particular, if $X$ is $\mathbb{Q}$-factorial over $W$ then $X'$ is also $\mathbb{Q}$-factorial over $W$. 
\end{thm}

\begin{proof}
Since $\rho(X/Y;W)-\rho(Z/Y;W)=1$ and $-(K_{X}+\Delta)$ is ample over $Z$, by shrinking $Y$ around $W$, we can find $r \in \mathbb{R}$ and an $\mathbb{R}$-Cartier divisor $L$ on $Z$ such that $$D-r(K_{X}+\Delta)-f^{*}L \equiv_{W}0.$$ Then $\phi_{*}D$ is $\mathbb{R}$-Cartier if and only if $$\phi_{*}(D-r(K_{X}+\Delta)-f^{*}L)=\phi_{*}D-r(K_{X'}+\Delta')-f'^{*}L$$ is $\mathbb{R}$-Cartier. By replacing $D$ with $D-r(K_{X}+\Delta)-f^{*}L$, we may assume $D\equiv_{W}0$. 

By shrinking $Y$ around $W$, we may assume that $D$ is a finite sum $\sum_{i=1}^{n}r_{i}D_{i}$ of Cartier divisors $D_{i}$ such that $r_{1},\cdots, r_{n}$ are linearly independent over $\mathbb{Q}$. 
Then, for any curve $C \subset X$ contained in a fiber of $\pi$ over $W$, we have
$$\sum_{i=1}^{n}r_{i}(D_{i} \cdot C)=(D \cdot C)=0.$$
By the linear independence of $r_{1},\cdots ,r_{n}$ over $\mathbb{Q}$, we have $(D_{i} \cdot C)=0$ for all $i$. 
This shows $D_{i}\equiv_{W}0$ for all $i$.  
If $\phi_{*}D_{i}$ are Cartier for all $i$, then clearly $\phi_{*}D$ is $\mathbb{R}$-Cartier. 
By replacing $D$ with $D_{i}$ for an arbitrary $i$, we may assume that $D$ is Cartier. 

Now $D$ is Cartier and $D \equiv_{W}0$, so $\phi_{*}D$ is Cartier by Theorem \ref{thm--mmp-Cartier}. 
\end{proof}

\begin{lem}\label{lem--mmp-triviallyintersect}
Let $\pi \colon X \to Y$ be a contraction from a normal analytic variety $X$ to a Stein space $Y$, and let $W \subset Y$ be a connected compact subset such that $\pi$ and $W$ satisfy (P). 
Let $(X,\Delta)$ be an lc pair. 
Suppose that $K_{X}+\Delta$ is nef over $W$. 
Let $H\geq 0$ be an $\mathbb{R}$-Cartier divisor on $X$ such that $(X,\Delta+H)$ is an lc pair.
Then there exists $\epsilon \in \mathbb{R}_{>0}$ such that for any $t \in (0,\epsilon]$, any sequence of steps of a $(K_{X}+\Delta+tH)$-MMP over $Y$ around $W$
$$(X_{0},\Delta_{0}+tH_{0})\dashrightarrow (X_{1},\Delta_{1}+tH_{1})\dashrightarrow \cdots \dashrightarrow (X_{i},\Delta_{i}+tH_{i})\dashrightarrow \cdots$$
and the $(K_{X_{i}}+\Delta_{i}+tH_{i})$-extremal contraction $f_{i} \colon X_{i} \to Z_{i}$ in the  $(i+1)$-th step of the MMP, $K_{X_{i}}+\Delta_{i}$ trivially intersects any curve over $W$ contracted by $f_{i}$. 
Moreover, $(X_{i},\Delta_{i}+H_{i})$ is an lc pair and $K_{X_{i}}+\Delta_{i}$ is nef over $W$ for every $i \geq 0$. 
\end{lem}

\begin{proof}
By Theorem \ref{thm--mmp-R-Cartier}, $K_{X_{i}}+\Delta_{i}$ is $\mathbb{R}$-Cartier for every $i \geq 0$. 
Let $f_{i} \colon X_{i} \to Z_{i}$ be the $(K_{X_{i}}+\Delta_{i}+tH_{i})$-extremal contraction of the $(i+1)$-th step of the log MMP. 

It is sufficient to prove that for every $i \geq 0$, $K_{X_{i}}+\Delta_{i}$ is nef over $W$ and there is a curve $C_{i} \subset X_{i}$ contracted by $f_{i}$ such that $(K_{X_{i}}+\Delta_{i})\cdot C_{i}=0$. 
Indeed, if the two properties holds, then $K_{X_{i}}+\Delta_{i}$ is nef over $W$  and numerically trivial over $Z_{i}$, and furthermore the $(K_{X}+\Delta+tH)$-MMP in the lemma is a sequence of steps of a $(K_{X}+\Delta+H)$-MMP over $Y$ around $W$. 
Then $(X_{i},\Delta_{i}+H_{i})$ are lc pairs, and therefore Lemma \ref{lem--mmp-triviallyintersect} holds true. 
Thus, we will prove that $K_{X_{i}}+\Delta_{i}$ is nef over $W$ and there exists a curve $C_{i} \subset X_{i}$ contracted by $f_{i}$ such that $(K_{X_{i}}+\Delta_{i})\cdot C_{i}=0$. 

Since the problem is local around $W$ and $W$ is compact, by shrinking $Y$ around $W$, we may assume that 
\begin{itemize}
\item
$K_{X}+\Delta$ and $H$ are globally $\mathbb{R}$-Cartier, i.e., $K_{X}+\Delta$ and $H$ are finite $\mathbb{R}$-linear combination of Cartier divisors $D^{(j)}$ ($1\leq j \leq m$) and $H^{(j')}$ ($1 \leq j' \leq m' $), respectively, and
\item
all $D^{(j)}$ and $H^{(j')}$ contain only finitely many components, in particular, 
$K_{X}$, $\Delta$, and $H$ contain only finitely many components.  
\end{itemize}
Then we can use the length of extremal rational curves (\cite[Section 14]{fujino-analytic-conethm}) and the Shokurov's polytopes (\cite[Section 14]{fujino-analytic-conethm}), and the argument in the algebraic case (see, for example, \cite[Proposition 3.2]{birkar-existII} or \cite[Lemma 2.12]{has-mmp}) works with no changes.  
\end{proof}

\begin{lem}\label{lem--mmp-weak-lc-model}
Let $\pi \colon X \to Y$ be a contraction from a normal analytic variety $X$ to a Stein space $Y$, and let $W \subset Y$ be a connected compact subset such that $\pi$ and $W$ satisfy (P). 
Let $(X,\Delta)$ be an lc pair and
$$(X_{0},\Delta_{0})\dashrightarrow (X_{1},\Delta_{1})\dashrightarrow \cdots \dashrightarrow (X_{n},\Delta_{n})$$
a finite sequence of steps of a $(K_{X}+\Delta)$-MMP over $Y$ around $W$. 
Let $(X',\Delta')$ be an lc pair. 
Then the following conditions are equivalent.
\begin{itemize}
\item
After shrinking $Y$ around $W$, $(X',\Delta')$ is a weak lc model of $(X,\Delta)$ over $Y$ around $W$. 
\item
After shrinking $Y$ around $W$, $(X',\Delta')$ is a weak lc model of $(X_{n},\Delta_{n})$ over $Y$ around $W$. 
\end{itemize}
\end{lem}

\begin{proof}
We may assume that $K_{X'}+\Delta'$ is nef over $W$. 
By shrinking $Y$ around $W$, we may assume that $X=X_{0}$ and the $(K_{X}+\Delta)$-MMP in the lemma is represented by a bimeromorphic contraction $X\dashrightarrow X_{n}$ over $Y$. 
Let $f \colon V \to X$ be a resolution of $X$ such that the induced bimeromorphic maps $f_{n}\colon V \dashrightarrow X_{n}$ and $f' \colon V \dashrightarrow X'$ are morphisms. 
Then we can write
\begin{equation*}
\begin{split}
f^{*}(K_{X}+\Delta)=f_{n}^{*}(K_{X_{n}}+\Delta_{n})+E
\end{split}
\end{equation*}
for some effective $f_{n}$-exceptional $\mathbb{R}$-divisor $E$ on $V$ such that the strict transform of any exceptional divisor of $X\dashrightarrow X_{n}$ appear in $E$. 

Suppose that $(X',\Delta')$ is a weak lc model of $(X,\Delta)$ over $Y$ around $W$. 
Then there is an effective $\mathbb{R}$-divisor $F$ on $V$ such that $f^{*}(K_{X}+\Delta)=f'^{*}(K_{X'}+\Delta')+F$. 
Then 
$$F-E=f_{n}^{*}(K_{X_{n}}+\Delta_{n})-f'^{*}(K_{X'}+\Delta'),$$
and $E$ is $f_{n}$-exceptional. 
By the negativity lemma (Corollary \ref{cor--negativity-veryexc-nef}), after shrinking $Y$ around $W$ we have $F-E \geq 0$. 
Moreover, for any prime divisor $P$ over $X$ which is exceptional over $X_{n}$, the relation $F-E \geq 0$ implies
$$a(P,X',\Delta')-a(P,X,\Delta)\geq a(P,X_{n},\Delta_{n})-a(P,X,\Delta)>0.$$
Therefore, $P$ is exceptional over $X'$. 
From this, we see that $(X',\Delta')$ is a log birational model of $(X_{n},\Delta_{n})$ over $Y$ and $a(Q,X_{n},\Delta_{n})-a(Q,X,\Delta)\geq 0$ for all prime divisors $Q$ over $X_{n}$. 
Therefore, $(X',\Delta')$ is a weak lc model of $(X_{n},\Delta_{n})$ over $Y$ around $W$. 

Conversely, suppose that $(X',\Delta')$ is a weak lc model of $(X_{n},\Delta_{n})$ over $Y$ around $W$. 
Then we have $f_{n}^{*}(K_{X_{n}}+\Delta_{n})=f'^{*}(K_{X'}+\Delta')+F'$
for some effective $\mathbb{R}$-divisor $F'$ on $V$. 
Then
$$f^{*}(K_{X}+\Delta)=f_{n}^{*}(K_{X_{n}}+\Delta_{n})+E=f'^{*}(K_{X'}+\Delta')+E+F'.$$
Moreover, for any prime divisor $P$ over $X$ which is exceptional over $X_{n}$, the relation $E+F' \geq E$ shows
$$a(P,X',\Delta')-a(P,X,\Delta)\geq a(P,X_{n},\Delta_{n})-a(P,X,\Delta)>0.$$
If $P$ is not exceptional over $X'$, then the definition of log birational models implies $a(P,X',\Delta')=-1$, which implies $a(P,X,\Delta)<-1$, a contradiction. 
Therefore, $P$ is exceptional over $X'$. 
From this, we see that $(X',\Delta')$ is a log birational model of $(X,\Delta)$ over $Y$ and $a(Q,X_{n},\Delta_{n})-a(Q,X,\Delta)\geq 0$ for all prime divisors $Q$ over $X_{n}$. 
Therefore, $(X',\Delta')$ is a weak lc model of $(X,\Delta)$ over $Y$ around $W$. 
\end{proof}

\subsection{Nef thresholds in log MMP}

In this subsection, we discuss nef thresholds of log MMP.

\begin{thm}\label{thm--mmp-nefthreshold}
Let $\pi \colon X \to Y$ be a contraction from a normal analytic variety $X$ to a Stein space $Y$, and let $W \subset Y$ be a connected compact subset such that $\pi$ and $W$ satisfy (P). 
Let $(X,\Delta)$ be an lc pair. 
Let $H$ be an $\mathbb{R}$-Cartier divisor on $X$, and let
$$(X_{0},\Delta_{0})\dashrightarrow (X_{1},\Delta_{1})\dashrightarrow \cdots \dashrightarrow (X_{i},\Delta_{i})\dashrightarrow \cdots$$
be a sequence of steps of a $(K_{X}+\Delta)$-MMP over $Y$ around $W$ with scaling of $H$. 
Put
$$\lambda_{i}:={\rm inf}\{t \in \mathbb{R}_{\geq 0}\,|\, \text{$K_{X_{i}}+\Delta_{i}+tH_{i}$ {\rm is nef over} $W$}\}$$
for each $i$. 
Then $\lambda_{i}\geq \lambda_{i+1}$ for every $i \geq 0$.  
\end{thm}

\begin{proof}
Fix $i \geq 0$. 
Note that $\lambda_{i}>0$ because otherwise $(X_{i},\Delta_{i}) \dashrightarrow (X_{i+1},\Delta_{i+1})$ cannot be defined. 
Let
$$
\xymatrix{
(X_{i},\Delta_{i})\ar@{-->}[rr]^-{\phi_{i}}\ar[dr]_-{f_{i}}&&(X'_{i},\phi_{i*}\Delta_{i})\ar[dl]^-{f'_{i}}\\
&Z_{i}
}
$$
be the $(i+1)$-th step of a $(K_{X}+\Delta)$-MMP over $Y$ around $W$ with scaling of $H$. 

By shrinking $Y$ around $W$, we may write 
$$K_{X_{i}}+\Delta_{i}+\lambda_{i}H_{i}=\sum_{j}r_{j}D_{j}$$
such that $D_{j}$ are Cartier and trivially intersects any curve over $W$ contracted by $f_{i}$. 
By the ampleness of $-(K_{X_{i}}+\Delta_{i})$ over $Z_{i}$ and \cite[Theorem 9.1]{fujino-analytic-conethm}, after shrinking $Y$ around $W$, it follows that $D_{j}$ are semi-ample over $Z_{i}$ for all $j$. 
Let $X_{i} \to V_{j}$ be the contractions over $Z_{i}$ induced by $D_{j}$. 
By the above condition for $D_{j}$, the induced morphisms $V_{j} \to Z_{i}$ are biholomorphisms near the inverse image of $W$. 
By shrinking $Y$ around $W$, we may assume $V_{j}=Z_{i}$ for all $j$. 
Then there are $\mathbb{Q}$-Cartier divisors $H_{j}$ on $Z_{i}$ such that $D_{j}\sim_{\mathbb{Q}}f_{i}^{*}H_{j}$ for all $j$. 
Put $A:=\sum_{j}r_{j}H_{j}$. 
Then
$$K_{X_{i}}+\Delta_{i}+\lambda_{i}H_{i}\sim_{\mathbb{R}} f_{i}^{*}A.$$
Since $K_{X_{i}}+\Delta_{i}+\lambda_{i}H_{i}$ is nef over $W$, we see that $A$ is also nef over $W$. 
Now we have 
$$\phi_{i*}(K_{X_{i}}+\Delta_{i}+\lambda_{i}H_{i})\sim_{\mathbb{R}} f'^{*}_{i}A,$$
and therefore $\phi_{i*}(K_{X_{i}}+\Delta_{i}+\lambda_{i}H_{i})$ is nef over $W$. 
This implies $\lambda_{i}\geq \lambda_{i+1}$. 
\end{proof}

\begin{lem}\label{lem--mmp-ample-scaling-nefthreshold}
Let $\pi \colon X \to Y$ be a projective morphism from a normal analytic variety $X$ to a Stein space $Y$, and let $W \subset Y$ be a compact subset such that $\pi$ and $W$ satisfy (P). 
Let $(X,\Delta)$ be an lc pair such that $K_{X}+\Delta$ is $\pi$-pseudo-effective. 
Suppose that there is a klt pair $(X,\Gamma)$. 
Let $H$ be a $\pi$-ample $\mathbb{R}$-divisor on $X$, and let
$$(X_{0},\Delta_{0})\dashrightarrow (X_{1},\Delta_{1})\dashrightarrow \cdots \dashrightarrow (X_{i},\Delta_{i})\dashrightarrow \cdots$$
be a sequence of steps of a $(K_{X}+\Delta)$-MMP over $Y$ around $W$ with scaling of $H$. 
Put
$$\lambda_{i}:={\rm inf}\{t \in \mathbb{R}_{\geq 0}\,|\, \text{$K_{X_{i}}+\Delta_{i}+tH_{i}$ {\rm is nef over} $W$}\}$$
for each $i \geq 0$. 
Then ${\rm lim}_{i \to \infty}\lambda_{i}=0$.  
\end{lem}

\begin{proof}
We take the Stein factorization $X \to Y'$ of $\pi$ and put $W'$ as the inverse image of $W$ by the induced morphism $Y' \to Y$. 
Then the $(K_{X}+\Delta)$-MMP over $Y$ around $W$ is a $(K_{X}+\Delta)$-MMP over $Y'$ around $W'$. 
By replacing $\pi$ and $W$ with $X \to Y'$ and $W'$ respectively, we may assume that $\pi$ is a contraction. 
Since the property (P4) of \cite{fujino-analytic-bchm} shows that $W$ has only finitely connected components, we may replace $W$ by any its connected component, and therefore, we may assume that $W$ is connected. 
As in Remark \ref{rem--mmp-reduction-basic}, we may freely shrink $Y$ around $W$ keeping the irreducibility of $X'$ and any $X_{i}$ if necessary.  

By replacing $H$, we may assume that $H$ is effective and $(X,\Gamma+H)$ is klt. 
We put $\lambda_{\infty}:={\rm lim}_{i \to \infty}\lambda_{i}$. 
Suppose by contradiction that $\lambda_{\infty}>0$. 
We pick $\epsilon \in \mathbb{R}_{>0}$ such that $\frac{1}{2}\lambda_{\infty}H+\epsilon(K_{X}+\Delta)-\epsilon(K_{X}+\Gamma)$ is $\pi$-ample. 
By shrinking $Y$ around $W$, we can find a general $\pi$-ample $\mathbb{R}$-divisor
 $H'$ on $X$ such that $H' \sim_{\mathbb{R}} \frac{1}{2}\lambda_{\infty}H+\epsilon(K_{X}+\Delta)-\epsilon(K_{X}+\Gamma)$. 
 We put $B:=(1-\epsilon)\Delta+\epsilon\Gamma+H'$. 
  Since $H'$ is general, $(X,B)$ is a klt pair such that 
$$K_{X}+\Delta+\frac{1}{2}\lambda_{\infty}H\sim_{\mathbb{R}}K_{X}+B.$$ 
Since $K_{X}+B+\lambda_{0}H\sim_{\mathbb{R}}(K_{X}+\Delta+\lambda_{0}H)+\frac{1}{2}\lambda_{\infty}H$ is $\pi$-ample, the $(K_{X}+\Delta)$-MMP induces a sequence of steps of a $(K_{X}+B)$-MMP over $Y$ around $W$ with scaling of $H$ 
$$(X_{0},B_{0})\dashrightarrow (X_{1},B_{1})\dashrightarrow \cdots \dashrightarrow (X_{i},B_{i})\dashrightarrow \cdots.$$
By Definition \ref{defn--mmp} (\ref{defn--mmp-(2)}), for every $i \geq0$, there is an open subset $Y_{i}\subset Y$ such that $Y_{i} \supset W$ and $(X_{i},B_{i}+(\lambda_{i}-\frac{1}{2}\lambda_{\infty})H_{i})$ is a weak lc model of $(\pi^{-1}(Y_{i}), (B+(\lambda_{i}-\frac{1}{2}\lambda_{\infty})H)|_{\pi^{-1}(Y_{i})})$ over $Y_{i}$ around $W$. 

By shrinking $Y$ around $W$, we may assume that $K_{X}$, $B$, and $H$ have only finitely many components. 
We fix an effective $\pi$-ample $\mathbb{Q}$-divisor $A \leq \frac{1}{2}\lambda_{\infty}H$, and let $V \subset {\rm WDiv}_{\mathbb{R}}(X)$ be an $\mathbb{R}$-vector space generated by the components of $B$ and $H$. 
Then 
$$\mathcal{L}_{A}(V;\pi^{-1}(W))\supset\left\{B+tH\,\middle|\, \frac{1}{2}\lambda_{\infty} \leq t \leq \lambda_{0}\right\},$$
where the left hand side is the set in \cite[Definition 11.8]{fujino-analytic-bchm}.  
Then the last sentence of the previous paragraph contradicts the finiteness of models \cite[Theorem E]{fujino-analytic-bchm}. 
Therefore, we have $\lambda_{\infty}={\rm lim}_{i \to \infty}\lambda_{i}=0$.  
\end{proof}

\subsection{From existence of weak lc model to termination}

In this subsection we prove Theorem \ref{thm--termination-mmp} and Theorem \ref{thm--mmpsequence-2}, which are the analytic case of \cite[Theorem 4.1 (iii)]{birkar-flip} and \cite[Lemma 2.14]{has-mmp} respectively.

\begin{lem}\label{lem--relative-mmp}
Let $\pi \colon X \to Y$ be a contraction from a normal analytic variety $X$ to a Stein space $Y$, and let $W \subset Y$ be a compact subset such that $\pi$ and $W$ satisfy (P). 
Let $(X,\Delta)$ be an lc pair. 
Let $f \colon X \to X'$ be a projective morphism of analytic varieties over $Y$ with the structure morphism $\pi' \colon X' \to Y$. 
Let $H'$ be a $\pi'$-ample Cartier divisor on $X'$, and let $H \geq 0$ be an $\mathbb{R}$-divisor on $X$ such that $H \sim_{\mathbb{R}} rf^{*}H'$ for some $r>2\cdot{\rm dim}\,X$ and $(X,\Delta+H)$ is lc. 
Then, for any sequence of steps of a $(K_{X}+\Delta+H)$-MMP over $Y$ around $W$
$$(X_{0},\Delta_{0}+H_{0})\dashrightarrow (X_{1},\Delta_{1}+H_{1})\dashrightarrow \cdots \dashrightarrow (X_{i},\Delta_{i}+H_{i})\dashrightarrow \cdots$$
and $i \geq 0$, the induced map $X_{i}\dashrightarrow X'$ is a morphism, and the $(K_{X}+\Delta+H)$-MMP is also a $(K_{X}+\Delta)$-MMP over $Y$ around $W$. 
\end{lem}

\begin{proof}
The arguments of \cite[Theorem 9.3]{fujino-analytic-bchm} and \cite[Lemma 9.4]{fujino-analytic-bchm} work in our situation. 
\end{proof}

\begin{thm}[cf.~{\cite[Theorem 4.1 (iii)]{birkar-flip}}]\label{thm--termination-mmp}
Let $\pi \colon X \to Y$ be a projective morphism from a normal analytic variety $X$ to a Stein space $Y$, and let $W \subset Y$ be a compact subset such that $\pi$ and $W$ satisfy (P). 
Let $(X,\Delta)$ be an lc pair. 
Let $H$ be an effective $\mathbb{R}$-Cartier divisor on $X$ such that $(X,\Delta+H)$ is an lc pair, and let
$$(X_{0},\Delta_{0})\dashrightarrow (X_{1},\Delta_{1})\dashrightarrow \cdots \dashrightarrow (X_{i},\Delta_{i})\dashrightarrow \cdots$$
be a sequence of steps of a $(K_{X}+\Delta)$-MMP over $Y$ around $W$ with scaling of $H$. 
Put
$$\lambda_{i}:={\rm inf}\{t \in \mathbb{R}_{\geq 0}\,|\, \text{$K_{X_{i}}+\Delta_{i}+tH_{i}$ {\rm is nef over} $W$}\}$$
for each $i$. 
If ${\rm lim}_{i \to \infty}\lambda_{i}=0$ and $(X,\Delta)$ has a weak lc model over $Y$ around $W$, then we have $\lambda_{n}=0$ for some $n \in \mathbb{Z}_{\geq 0}$. 
\end{thm}

\begin{proof}
Let $(X',\Delta')$ be a weak lc model of $(X,\Delta)$ over $Y$ around $W$. 
Let $H'$ be the strict transform of $H$ on $X'$. 
We take the Stein factorization $X \to Y'$ of $\pi$ and put $W'$ as the inverse image of $W$ by the induced morphism $Y' \to W$. 
Then the $(K_{X}+\Delta)$-MMP over $Y$ around $W$ is a $(K_{X}+\Delta)$-MMP over $Y'$ around $W'$, and $(X',\Delta')$ is a weak lc model of $(X,\Delta)$ over $Y'$ around $W'$. 
By replacing $\pi$ and $W$ with $X \to Y'$ and $W'$ respectively, we may assume that $\pi$ is a contraction. 
Since the property (P4) of \cite{fujino-analytic-bchm} shows that $W$ has only finitely connected components, we may replace $W$ by any its connected component, and therefore, we may assume that $W$ is connected. 
As in Remark \ref{rem--mmp-reduction-basic}, we may freely shrink $Y$ around $W$ keeping the irreducibility of $X'$ and any $X_{i}$ if necessary.

\begin{step1}\label{step1-thm--termination-mmp}
In this step, after shrinking $Y$ around $W$ we will construct a good common resolution $\widetilde{X} \to X$ and $\widetilde{X} \to X'$ of the bimeromorphic map $X \dashrightarrow X'$. 

After shrinking $Y$ around $W$, we take a common log resolution $f \colon \widetilde{X} \to X$ and $f' \colon \widetilde{X} \to X'$ of the bimeromorphic map $(X,\Delta+H) \dashrightarrow (X',\Delta'+H')$. 
Let $\widetilde{\Delta}$ be the sum of $f^{-1}_{*}\Delta$ and the reduced $f$-exceptional divisor. 
Put $\widetilde{H}:=f^{-1}_{*}H$. 
Since $(X,\Delta+tH)$ is lc for all $t \in [0,1]$, for each $t \in [0,1]$ there is an effective $f$-exceptional $\mathbb{R}$-divisor $\widetilde{E}_{t}$ on $\widetilde{X}$ such that 
\begin{equation*}\tag{$*$}\label{thm--termination-mmp-(*)}
\begin{split}
K_{\widetilde{X}}+\widetilde{\Delta}+t\widetilde{H}=f^{*}(K_{X}+\Delta+tH)+\widetilde{E}_{t}.
\end{split}
\end{equation*} 
By definition, we have $\widetilde{E}_{t'}\leq \widetilde{E}_{t}$ when $t \leq t'$. 
By the definition of weak lc models, after shrinking $Y$ around $W$ we have an effective $f'$-exceptional $\mathbb{R}$-divisor $\widetilde{F}$ on $\widetilde{X}$ such that
$$f^{*}(K_{X}+\Delta)=f'^{*}(K_{X'}+\Delta')+\widetilde{F}.$$
From these relations, we have
\begin{equation*}\tag{$**$}\label{thm--termination-mmp-(**)}
\begin{split}
K_{\widetilde{X}}+\widetilde{\Delta}=f'^{*}(K_{X'}+\Delta')+\widetilde{E}_{0}+\widetilde{F}.
\end{split}
\end{equation*} 
We show that $\widetilde{E}_{0}+\widetilde{F}$ is $f'$-exceptional. 
It is sufficient to prove that $\widetilde{E}_{0}$ is $f'$-exceptional. 
For any prime divisor $P$ on $\widetilde{X}$, we have 
$${\rm coeff}_{P}(\widetilde{E}_{0})=a(P,X,\Delta)+{\rm coeff}_{P}(\widetilde{\Delta}).$$
When $P$ is not $f$-exceptional, then ${\rm coeff}_{P}(\widetilde{E}_{0})$ is clearly zero. 
When $P$ is $f$-exceptional but not $f'$-exceptional, we have ${\rm coeff}_{P}(\widetilde{\Delta})=1$ and 
$$-1=a(P,X',\Delta')\geq a(P,X,\Delta)\geq-1,$$
thus we have ${\rm coeff}_{P}(\widetilde{E}_{0})=a(P,X,\Delta)+{\rm coeff}_{P}(\widetilde{\Delta})=0$. 
From these discussions, $\widetilde{E}_{0}$ is $f'$-exceptional. 
Therefore, $\widetilde{E}_{0}+\widetilde{F}$ is $f'$-exceptional. 
\end{step1}

\begin{step1}\label{step2-thm--termination-mmp}
In this step, we run a $(K_{\widetilde{X}}+\widetilde{\Delta})$-MMP by using Lemma \ref{lem--relative-mmp} and construct an lc pair $(\widetilde{X}',\widetilde{\Delta}')$ with a projective bimeromorphism $\widetilde{X}' \to X'$. 

Let $\pi' \colon X' \to Y$ be the structure morphism. 
Let $L'$ be a $\pi'$-ample Cartier divisor on $X'$, and let $\widetilde{L} \geq 0$ be an $\mathbb{R}$-divisor on $\widetilde{X}$ such that $(\widetilde{X},\widetilde{\Delta}+\widetilde{H}+\widetilde{L})$ is an lc pair and $\widetilde{L} \sim_{\mathbb{R}} (3\cdot{\rm dim}\,X) f'^{*}L'$. 
We run a $(K_{\widetilde{X}}+\widetilde{\Delta}+\widetilde{L})$-MMP over $Y$ around $W$ with scaling of a $(\pi'\circ f')$-ample divisor
$$(\widetilde{X}_{0},\widetilde{\Delta}_{0}+\widetilde{L}_{0})\dashrightarrow (\widetilde{X}_{1},\widetilde{\Delta}_{1}+\widetilde{L}_{1})\dashrightarrow \cdots \dashrightarrow (\widetilde{X}_{j},\widetilde{\Delta}_{j}+\widetilde{L}_{j})\dashrightarrow \cdots.$$
Since $(X,\Delta)$ has a weak lc model over $Y$ around $W$, $K_{X}+\Delta$ is $\pi$-pseudo-effective. 
Thus $K_{\widetilde{X}}+\widetilde{\Delta}+\widetilde{L}$ is $(\pi'\circ f')$-pseudo-effective. 
By \cite[Lemma 13.7]{fujino-analytic-bchm}, there is $j' \gg 0$ such that $K_{\widetilde{X}_{j'}}+\widetilde{\Delta}_{j'}+\widetilde{L}_{j'} \in \overline{\rm Mov}(\widetilde{X}_{j'}/Y;W)$. 
After shrinking $Y$ around $W$, we may assume that the $(K_{\widetilde{X}}+\widetilde{\Delta}+\widetilde{L})$-MMP $(\widetilde{X}_{0},\widetilde{\Delta}_{0}+\widetilde{L}_{0})\dashrightarrow (\widetilde{X}_{j'},\widetilde{\Delta}_{j'}+\widetilde{L}_{j'})$ is a bimeromorphic contraction over $Y$. 
Put 
$$\widetilde{X}':=\widetilde{X}_{j'}, \quad \widetilde{\Delta}':=\widetilde{\Delta}_{j'}, \quad {\rm and} \quad \widetilde{L}':=\widetilde{L}_{j'}.$$
By Lemma \ref{lem--relative-mmp}, the induced bimeromorphic map $\widetilde{f}' \colon \widetilde{X}' \dashrightarrow X'$ is a morphism. 
Let $\widetilde{E}'$ and $\widetilde{F}'$ be the strict transforms of $\widetilde{E}_{0}$ and $\widetilde{F}$ on $\widetilde{X}'$, respectively. 
Then we have $\widetilde{L}' \sim_{\mathbb{R}} (3\cdot{\rm dim}\,X) \widetilde{f}'^{*}L'$, and (\ref{thm--termination-mmp-(**)}) implies
$$K_{\widetilde{X}'}+\widetilde{\Delta}'=\widetilde{f}'^{*}(K_{X'}+\Delta')+\widetilde{E}'+\widetilde{F}'.$$
Therefore $K_{\widetilde{X}'}+\widetilde{\Delta}' \in \overline{\rm Mov}(\widetilde{X}'/X';\pi'^{-1}(W))$. 
Moreover, $\widetilde{E}'+\widetilde{F}'$ is $\widetilde{f}'$-exceptional by construction. 
By the negativity lemma (Corollary \ref{cor--negativity-veryexc-2}), after shrinking $Y$ around $W$ we have $\widetilde{E}'+\widetilde{F}'=0$. 
Therefore, $K_{\widetilde{X}'}+\widetilde{\Delta}'=\widetilde{f}'^{*}(K_{X'}+\Delta')$ is nef over $W$. 
Now recall by Lemma \ref{lem--relative-mmp} that the $(K_{\widetilde{X}}+\widetilde{\Delta}+\widetilde{L})$-MMP over $Y$ around $W$ is also a $(K_{\widetilde{X}}+\widetilde{\Delta})$-MMP over $Y$ around $W$. 
Hence, $(\widetilde{X}',\widetilde{\Delta}')$ is a log minimal model of $(\widetilde{X},\widetilde{\Delta})$ over $Y$ around $W$. 
\end{step1}

\begin{step1}\label{step3-thm--termination-mmp}
In this step, we will construct a diagram used in the rest of the proof.  

Let $\widetilde{H}'$ be the strict transform of $\widetilde{H}$ on $\widetilde{X}'$. 
We take a positive real number $\epsilon$ such that $\widetilde{X} \dashrightarrow \widetilde{X}'$ is a sequence of steps of a $(K_{\widetilde{X}}+\widetilde{\Delta}+\epsilon\widetilde{H})$-MMP over $Y$ around $W$. 
By Lemma \ref{lem--mmp-triviallyintersect} and taking $\epsilon$ sufficiently small, we may assume that for any sequence of steps of a $(K_{\widetilde{X}'}+\widetilde{\Delta}'+\epsilon\widetilde{H}')$-MMP over $Y$ around $W$
$$(\widetilde{X}'_{0},\widetilde{\Delta}'_{0}+\epsilon\widetilde{H}'_{0})\dashrightarrow (\widetilde{X}'_{1},\widetilde{\Delta}'_{1}+\epsilon\widetilde{H}'_{1})\dashrightarrow \cdots \dashrightarrow (\widetilde{X}'_{k},\widetilde{\Delta}'_{k}+\epsilon\widetilde{H}'_{k})\dashrightarrow \cdots,$$
$K_{\widetilde{X}'_{k}}+\widetilde{\Delta}'_{k}$ trivially intersects the curves over $W$ that are contracted by the $(k+1)$-th steps of the MMP. 
In particular, for every $k \geq 0$, after shrinking $Y$ around $W$, it follows that $(\widetilde{X}'_{k},\widetilde{\Delta}'_{k})$ is a weak lc model of $(\widetilde{X}', \widetilde{\Delta}')$ over $Y$ around $W$. 

Since ${\rm lim}_{i \to \infty}\lambda_{i}=0$, there is an index $n \geq 0$ such that $\lambda_{n}<\lambda_{n-1}$ and $\lambda_{n}<\epsilon$. 
By shrinking $Y$ around $W$, we may assume that the sequence of the $(K_{X}+\Delta)$-MMP
$$(X_{0},\Delta_{0})\dashrightarrow (X_{1},\Delta_{1})\dashrightarrow \cdots \dashrightarrow (X_{n},\Delta_{n})$$
is represented by bimeromorphic contractions over $Y$ and $X=X_{0}$. 

By the above discussions, after shrinking $Y$ around $W$ suitably, we get 
$$
\xymatrix{
(\widetilde{X}, \widetilde{\Delta})\ar[d]_{f}\ar@{-->}[r]&(\widetilde{X}', \widetilde{\Delta}')\\
(X,\Delta)\ar@{-->}[rr]&&(X_{n},\Delta_{n})
}
$$
over $Y$ and a positive real number $\epsilon<1 $ satisfying the following properties: 
\begin{enumerate}[(i)]
\item\label{thm--termination-mmp-(i)}
$f \colon \widetilde{X} \to X$ is a log resolution of $(X,\Delta+H)$ such that for any $t \in [0,1]$, we can write
$$K_{\widetilde{X}}+\widetilde{\Delta}+t\widetilde{H}=f^{*}(K_{X}+\Delta+tH)+\widetilde{E}_{t}$$
for some effective $f$-exceptional $\mathbb{R}$-divisor $\widetilde{E}_{t}$ on $\widetilde{X}$, where $\widetilde{H}=f^{-1}_{*}H$,  
\item\label{thm--termination-mmp-(ii)}
$\widetilde{X} \dashrightarrow \widetilde{X}'$ is a bimeromorphic contraction over $Y$ that defines:
\begin{itemize}
\item
a finite sequence of steps of a $(K_{\widetilde{X}}+\widetilde{\Delta})$-MMP over $Y$ around $W$ 
$$(\widetilde{X}, \widetilde{\Delta}) \dashrightarrow (\widetilde{X}', \widetilde{\Delta}')$$
terminating with a log minimal model $(\widetilde{X}', \widetilde{\Delta}')$ over $Y$ around $W$, and
\item
a finite sequence of steps of a $(K_{\widetilde{X}}+\widetilde{\Delta}+\epsilon \widetilde{H})$-MMP over $Y$ around $W$
$$(\widetilde{X}, \widetilde{\Delta}+\epsilon \widetilde{H}) \dashrightarrow (\widetilde{X}', \widetilde{\Delta}'+\epsilon \widetilde{H}'),$$ 
\end{itemize}
\item\label{thm--termination-mmp-(iii)}
for any sequence of steps of a $(K_{\widetilde{X}'}+\widetilde{\Delta}'+\epsilon\widetilde{H}')$-MMP over $Y$ around $W$
$$(\widetilde{X}'_{0},\widetilde{\Delta}'_{0}+\epsilon\widetilde{H}'_{0})\dashrightarrow (\widetilde{X}'_{1},\widetilde{\Delta}'_{1}+\epsilon\widetilde{H}'_{1})\dashrightarrow \cdots \dashrightarrow (\widetilde{X}'_{k},\widetilde{\Delta}'_{k}+\epsilon\widetilde{H}'_{k})\dashrightarrow \cdots$$
and any $k \geq0$, after shrinking $Y$ around $W$ suitably, 
$(\widetilde{X}'_{k},\widetilde{\Delta}'_{k})$ is a weak lc model of $(\widetilde{X}', \widetilde{\Delta}')$ over $Y$ around $W$, and 
\item\label{thm--termination-mmp-(iv)}
$X \dashrightarrow X_{n}$ is a bimeromorphic contraction over $Y$ that defines a finite sequence of steps of a $(K_{X}+\Delta)$-MMP over $Y$ around $W$ with scaling of $H$
$$(X, \Delta) \dashrightarrow (X_{n}, \Delta_{n})$$
such that $\lambda_{n}<{\rm min}\{\lambda_{n-1},\epsilon\}$. 
\end{enumerate}
\end{step1}

\begin{step1}\label{step4-thm--termination-mmp}
We put 
$$\lambda':={\rm min}\{\lambda_{n-1},\epsilon\}.$$
In this step, we will prove that after shrinking $Y$ around $W$ suitably, $(X_{n},\Delta_{n}+\lambda'H_{n})$ is a weak lc model of $(\widetilde{X}',\widetilde{\Delta}'+\lambda'\widetilde{H}')$ over $Y$ around $W$. 

By definition and (\ref{thm--termination-mmp-(i)}) in Step \ref{step3-thm--termination-mmp}, $(X_{n},\Delta_{n}+\lambda'H_{n})$ is a weak lc model of $(\widetilde{X},\widetilde{\Delta}+\lambda'\widetilde{H})$ over $Y$ around $W$. 
Moreover, since $\lambda' \leq \epsilon$, by (\ref{thm--termination-mmp-(ii)}) in Step \ref{step3-thm--termination-mmp}, the bimeromorphic map $\widetilde{X} \dashrightarrow \widetilde{X}'$ defines a sequence of steps of a $(K_{\widetilde{X}}+\widetilde{\Delta}+\lambda' \widetilde{H})$-MMP over $Y$ around $W$
$$(\widetilde{X}, \widetilde{\Delta}+\lambda' \widetilde{H}) \dashrightarrow (\widetilde{X}', \widetilde{\Delta}'+\lambda' \widetilde{H}').$$ 
By Lemma \ref{lem--mmp-weak-lc-model}, after shrinking $Y$ around $W$ suitably, $(X_{n},\Delta_{n}+\lambda'H_{n})$ is a weak lc model of $(\widetilde{X}',\widetilde{\Delta}'+\lambda'\widetilde{H}')$ over $Y$ around $W$. 
\end{step1}

\begin{step1}\label{step5-thm--termination-mmp}
In this step, we will prove Theorem \ref{thm--termination-mmp} in the case where $X$ is $\mathbb{Q}$-factorial over $W$, $(X,0)$ is klt, and $H$ is a general $\pi$-ample $\mathbb{R}$-divisor. 

By shrinking $Y$ around $W$, we may assume that $\Delta$ is globally $\mathbb{R}$-Cartier. 
We take $\delta \in \mathbb{R}_{>0}$ such that $\delta \Delta+\epsilon H$ is ample. 
By shrinking $Y$ around $W$, we can find a general $\pi$-ample $\mathbb{R}$-divisor $\Gamma$ on $X$ such that $\Gamma\sim_{\mathbb{R}} \delta \Delta+\epsilon H$ and $f^{*}\Gamma=f_{*}^{-1}\Gamma$. 
We may assume that the pair $\bigl(X,\Delta+\epsilon H+\gamma(\Gamma-\delta \Delta-\epsilon H)\bigr)$ is klt for any $\gamma\in (0,1)$. 
We take $\gamma>0$ sufficiently small and we define 
$$\widetilde{\Gamma}:=\widetilde{\Delta}+\epsilon\widetilde{H}+\gamma f^{*}(\Gamma-\delta \Delta-\epsilon H)-\gamma\widetilde{E}_{\epsilon}.$$ 
Since ${\rm Supp}\,(\widetilde{\Delta}+\epsilon\widetilde{H})$ contains all $f$-exceptional divisors and components of $f^{-1}_{*}(\Delta+H)$, we may assume that $\widetilde{\Gamma} \geq 0$. 
By (\ref{thm--termination-mmp-(i)}) in Step \ref{step3-thm--termination-mmp}, we have 
$$K_{\widetilde{X}}+\widetilde{\Gamma}=f^{*}\bigl(K_{X}+\Delta+\epsilon H+\gamma(\Gamma-\delta \Delta-\epsilon H)\bigr)+(1-\gamma)\widetilde{E}_{\epsilon}.$$
Since $\Gamma$ is general, we may assume that $(\widetilde{X},\widetilde{\Gamma})$ is log smooth and lc, and $\Gamma\sim_{\mathbb{R}} \delta \Delta+\epsilon H$ shows
$$K_{\widetilde{X}}+\widetilde{\Gamma}\sim_{\mathbb{R}}
K_{\widetilde{X}}+\widetilde{\Delta}+\epsilon\widetilde{H}-\gamma\widetilde{E}_{\epsilon}.$$
Then $(\widetilde{X},\widetilde{\Gamma})$ is klt. 
Indeed, if a prime divisor $Q$ on $\widetilde{X}$ is not a component of $\widetilde{E}_{\epsilon}$, then 
$${\rm coeff}_{Q}(\widetilde{\Gamma})=-a\bigl(Q,X,\Delta+\epsilon H+\gamma(\Gamma-\delta \Delta-\epsilon H)\bigr)<1.$$ 
If $Q$ is a component of $\widetilde{E}_{\epsilon}$, then $Q$ is not a component of $f^{*}\Gamma$, and therefore
$${\rm coeff}_{Q}(\widetilde{\Gamma})= {\rm coeff}_{Q}(\widetilde{\Delta}+\epsilon\widetilde{H}+\gamma f^{*}(\Gamma-\delta \Delta-\epsilon H)-\gamma\widetilde{E}_{\epsilon})< {\rm coeff}_{Q}(\widetilde{\Delta}+\epsilon\widetilde{H})\leq 1.$$ 
Since $(\widetilde{X},\widetilde{\Gamma})$ is log smooth, we see that $(\widetilde{X},\widetilde{\Gamma})$ is klt. 

Now $\widetilde{\Gamma}$ is $(\pi \circ f)$-big.
Since $\gamma>0$ is sufficiently small, we may assume that $\widetilde{X}\dashrightarrow \widetilde{X}'$ is a finite sequence of steps of a $(K_{\widetilde{X}}+\widetilde{\Delta}+\epsilon \widetilde{H}-\gamma\widetilde{E}_{\epsilon})$-MMP over $Y$ around $W$. 
Here, we used the second part of (\ref{thm--termination-mmp-(ii)}) in Step \ref{step3-thm--termination-mmp}. 
Then $\widetilde{X}\dashrightarrow \widetilde{X}'$ is also a $(K_{\widetilde{X}}+\widetilde{\Gamma})$-MMP over $Y$ around $W$. 
Let $\widetilde{\Gamma}'$ be the strict transform of $\widetilde{\Gamma}$ on $\widetilde{X}'$. 
Then $(\widetilde{X}',\widetilde{\Gamma}')$ is klt. 
Moreover, by the argument of Step \ref{step2-thm--termination-mmp}, the strict transform of $\widetilde{E}_{0}$ on $\widetilde{X}'$ is zero. 
Since $\widetilde{E}_{\epsilon}\leq \widetilde{E}_{0}$, the strict transform of $\widetilde{E}_{\epsilon}$ on $\widetilde{X}'$ is zero. 
This implies
$$K_{\widetilde{X}'}+\widetilde{\Gamma}'\sim_{\mathbb{R}}K_{\widetilde{X}'}+\widetilde{\Delta}'+\epsilon\widetilde{H}'.$$ 
Since $\widetilde{\Gamma}'$ is big over $Y$, by \cite[Theorem 1.7]{fujino-analytic-bchm} and shrinking $Y$ around $W$, we get a bimeromorphic contraction $\widetilde{X}'\dashrightarrow \widetilde{X}''$ over $Y$ that defines a finite sequence of steps of a $(K_{\widetilde{X}'}+\widetilde{\Gamma}')$-MMP over $Y$ around $W$ terminating with a log minimal model $(\widetilde{X}'',\widetilde{\Gamma}'')$. 
Let $\widetilde{\Delta}''$ (resp.~$\widetilde{H}''$) be the strict transform of $\widetilde{\Delta}'$ (resp.~$\widetilde{H}'$) on $\widetilde{X}''$. 
Then $\widetilde{X}'\dashrightarrow \widetilde{X}''$ defines a sequence of steps of a $(K_{\widetilde{X}'}+\widetilde{\Delta}'+\epsilon\widetilde{H}')$-MMP over $Y$ around $W$ terminating with a log minimal model $(\widetilde{X}'',\widetilde{\Delta}''+\epsilon\widetilde{H}'')$ over $Y$ around $W$. 
Therefore, we have a bimeromorphic contraction
$$\widetilde{X}\dashrightarrow \widetilde{X}''$$
over $Y$ that defines a sequence of steps of a $(K_{\widetilde{X}}+\widetilde{\Delta}+\epsilon\widetilde{H})$-MMP over $Y$ around $W$ terminating with a log minimal model $(\widetilde{X}'',\widetilde{\Delta}''+\epsilon\widetilde{H}'')$ over $Y$ around $W$. 
By (\ref{thm--termination-mmp-(iii)}) in Step \ref{step3-thm--termination-mmp}, $(\widetilde{X}'',\widetilde{\Delta}'')$ is a weak lc model of $(\widetilde{X},\widetilde{\Delta})$ over $Y$ around $W$. 
Since
$0\leq \lambda_{n}<\lambda'\leq\epsilon$ 
by (\ref{thm--termination-mmp-(iv)}) in Step \ref{step3-thm--termination-mmp},  we see that $(\widetilde{X}'',\widetilde{\Delta}''+\lambda'\widetilde{H}'')$ and $(\widetilde{X}'',\widetilde{\Delta}''+\lambda_{n}\widetilde{H}'')$ are weak lc models of $(\widetilde{X},\widetilde{\Delta}+\lambda'\widetilde{H})$ and $(\widetilde{X},\widetilde{\Delta}+\lambda_{n}\widetilde{H})$ over $Y$ around $W$, respectively. 

We use the diagram
$$
\xymatrix{
\widetilde{X}\ar[d]_{f}\ar@{-->}[rr]&&\widetilde{X}''\\
X\ar@{-->}[rr]&&X_{n}.
}
$$
By (\ref{thm--termination-mmp-(i)}) and (\ref{thm--termination-mmp-(iv)}) in Step \ref{step3-thm--termination-mmp} and the definition of weak lc models, $(X_{n},\Delta_{n}+\lambda'H_{n})$ and $(X_{n},\Delta_{n}+\lambda_{n}H_{n})$ are also weak lc models of $(\widetilde{X},\widetilde{\Delta}+\lambda'\widetilde{H})$ and $(\widetilde{X},\widetilde{\Delta}+\lambda_{n}\widetilde{H})$ over $Y$ around $W$, respectively. 
Let $g \colon V \to X_{n}$ and $\widetilde{g} \colon V \to \widetilde{X}''$ be a common resolution of the induced bimeromorphic map $X_{n} \dashrightarrow \widetilde{X}''$. 
By Lemma \ref{lem--two-logminmodel} and shrinking $Y$ around $W$, we have
\begin{equation*}
\begin{split}
g^{*}(K_{X_{n}}+\Delta_{n}+\lambda'H_{n})&=\widetilde{g}^{*}(K_{\widetilde{X}''}+\widetilde{\Delta}''+\lambda'\widetilde{H}''),\quad {\rm and}\\
g^{*}(K_{X_{n}}+\Delta_{n}+\lambda_{n}H_{n})&=\widetilde{g}^{*}(K_{\widetilde{X}''}+\widetilde{\Delta}''+\lambda_{n}\widetilde{H}'').
\end{split}
\end{equation*}
Since $\lambda'\neq \lambda_{n}$, we have 
$$g^{*}(K_{X_{n}}+\Delta_{n})=\widetilde{g}^{*}(K_{\widetilde{X}''}+\widetilde{\Delta}'').$$
By (\ref{thm--termination-mmp-(iii)}) in Step \ref{step3-thm--termination-mmp}, $K_{\widetilde{X}''}+\widetilde{\Delta}''$ is nef over $W$. 
Therefore, $K_{X_{n}}+\Delta_{n}$ is nef over $W$. 
This implies $\lambda_{n}=0$, and therefore Theorem \ref{thm--termination-mmp} holds if $X$ is $\mathbb{Q}$-factorial over $W$, $(X,0)$ is klt, and $H$ is a general $\pi$-ample $\mathbb{R}$-divisor. 
\end{step1}

\begin{step1}\label{step6-thm--termination-mmp}
Finally, we prove Theorem \ref{thm--termination-mmp} in the full general case. 

The argument is very similar to that in Step \ref{step5-thm--termination-mmp}. 
We will use the diagram
$$
\xymatrix{
(\widetilde{X}, \widetilde{\Delta})\ar[d]_{f}\ar@{-->}[r]&(\widetilde{X}', \widetilde{\Delta}')\\
(X,\Delta)\ar@{-->}[rr]&&(X_{n},\Delta_{n})
}
$$
and conditions (\ref{thm--termination-mmp-(i)})--(\ref{thm--termination-mmp-(iv)}) in Step \ref{step3-thm--termination-mmp}. 

By Step \ref{step4-thm--termination-mmp}, after shrinking $Y$ around $W$, there is a weak lc model of $(\widetilde{X}',\widetilde{\Delta}'+\lambda'\widetilde{H}')$ over $Y$ around $W$. 
We run a $(K_{\widetilde{X}'}+\widetilde{\Delta}'+\lambda'\widetilde{H}')$-MMP over $Y$ around $W$ with scaling of an ample divisor. 
By Lemma \ref{lem--mmp-ample-scaling-nefthreshold} and the conclusion of Step \ref{step5-thm--termination-mmp}, the MMP terminates with a log minimal model   $(\widetilde{X}''',\widetilde{\Delta}'''+\lambda'\widetilde{H}''')$ over $Y$ around $W$. 
By (\ref{thm--termination-mmp-(i)}) in Step \ref{step3-thm--termination-mmp} and shrinking $Y$ around $W$, we get a bimeromorphic contraction
$$\widetilde{X}\dashrightarrow \widetilde{X}'''$$
over $Y$ that defines a sequence of steps of a $(K_{\widetilde{X}}+\widetilde{\Delta}+\lambda'\widetilde{H})$-MMP over $Y$ around $W$ terminating with a log minimal model $(\widetilde{X}''',\widetilde{\Delta}'''+\lambda'\widetilde{H}''')$ over $Y$ around $W$. 
By (\ref{thm--termination-mmp-(iii)}) in Step \ref{step3-thm--termination-mmp}, $(\widetilde{X}''',\widetilde{\Delta}''')$ is a weak lc model of $(\widetilde{X},\widetilde{\Delta})$ over $Y$ around $W$. 
Since $0\leq \lambda_{n}<\lambda'$, we see that
 $(\widetilde{X}''',\widetilde{\Delta}'''+\lambda'\widetilde{H}''')$ and $(\widetilde{X}''',\widetilde{\Delta}'''+\lambda_{n}\widetilde{H}''')$ are weak lc models of $(\widetilde{X},\widetilde{\Delta}+\lambda'\widetilde{H})$ and $(\widetilde{X},\widetilde{\Delta}+\lambda_{n}\widetilde{H})$ over $Y$ around $W$, respectively. 

We use the diagram
$$
\xymatrix{
\widetilde{X}\ar[d]_{f}\ar@{-->}[rr]&&\widetilde{X}'''\\
X\ar@{-->}[rr]&&X_{n}.
}
$$
By (\ref{thm--termination-mmp-(i)}) and (\ref{thm--termination-mmp-(iv)}) in Step \ref{step3-thm--termination-mmp} and the definition of weak lc models, $(X_{n},\Delta_{n}+\lambda'H_{n})$ and $(X_{n},\Delta_{n}+\lambda_{n}H_{n})$ are also weak lc models of $(\widetilde{X},\widetilde{\Delta}+\lambda'\widetilde{H})$ and $(\widetilde{X},\widetilde{\Delta}+\lambda_{n}\widetilde{H})$ over $Y$ around $W$, respectively. 
Let $h \colon V' \to X_{n}$ and $\widetilde{h} \colon V' \to \widetilde{X}'''$ be a common resolution of the induced bimeromorphic map $X_{n} \dashrightarrow \widetilde{X}'''$. 
By Lemma \ref{lem--two-logminmodel} and shrinking $Y$ around $W$, we have
\begin{equation*}
\begin{split}
h^{*}(K_{X_{n}}+\Delta_{n}+\lambda'H_{n})&=\widetilde{h}^{*}(K_{\widetilde{X}'''}+\widetilde{\Delta}'''+\lambda'\widetilde{H}'''),\quad {\rm and}\\
h^{*}(K_{X_{n}}+\Delta_{n}+\lambda_{n}H_{n})&=\widetilde{h}^{*}(K_{\widetilde{X}'''}+\widetilde{\Delta}'''+\lambda_{n}\widetilde{H}''').
\end{split}
\end{equation*}
Since $\lambda'\neq \lambda_{n}$, we have 
$$h^{*}(K_{X_{n}}+\Delta_{n})=\widetilde{h}^{*}(K_{\widetilde{X}'''}+\widetilde{\Delta}''').$$
By (\ref{thm--termination-mmp-(iii)}) in Step \ref{step3-thm--termination-mmp}, $K_{\widetilde{X}'''}+\widetilde{\Delta}'''$ is nef over $W$. 
Therefore, $K_{X_{n}}+\Delta_{n}$ is nef over $W$. 
This implies $\lambda_{n}=0$, and therefore Theorem \ref{thm--termination-mmp} holds.
\end{step1}
We complete the proof. 
\end{proof}

\begin{thm}[cf.~{\cite[Lemma 2.14]{has-mmp}}]\label{thm--mmpsequence-2}
Let $\pi \colon X \to Y$ be a contraction from a normal analytic variety $X$ to a Stein space $Y$, and let $W \subset Y$ be a connected compact subset such that $\pi$ and $W$ satisfy (P). 
Let $(X,\Delta)$ be an lc pair. 
Suppose that $(X,B)$ is klt for some $\mathbb{R}$-divisor $B$ on $X$. 
Let $H$ be an effective $\mathbb{R}$-Cartier divisor on $X$ such that $(X,\Delta+H)$ is an lc pair and $K_{X}+\Delta+H$ is nef over $W$. 
Suppose that for any $\mu \in (0,1]$, after shrinking $Y$ around $W$, the lc pair $(X,\Delta+\mu H)$ has a log minimal model over $Y$ around $W$.  
Then, we may construct  a sequence of steps of a $(K_{X}+\Delta)$-MMP over $Y$ around $W$ with scaling of $H$
$$(X_{0},\Delta_{0})\dashrightarrow (X_{1},\Delta_{1})\dashrightarrow \cdots \dashrightarrow (X_{i},\Delta_{i})\dashrightarrow \cdots$$
such that if we put
$$\lambda_{i}:={\rm inf}\{t \in \mathbb{R}_{\geq 0}\,|\, \text{$K_{X_{i}}+\Delta_{i}+tH_{i}$ {\rm is nef over} $W$}\}$$
for each $i \geq 0$, then either the log MMP terminates after finitely many steps or we have ${\rm lim}_{i \to \infty}\lambda_{i}=0$ even if the log MMP terminates does not terminate. 
\end{thm}

\begin{proof}
The argument of \cite[Lemma 2.14]{has-mmp} works in the analytic setting. 
\end{proof}

\subsection{Special termination}\label{subsec-special-termi}

Let $\pi \colon X \to Y$ be a contraction from a normal analytic variety $X$ to a Stein space $Y$, and let $W \subset Y$ be a connected compact subset such that $\pi$ and $W$ satisfy (P). 
Let $(X,\Delta)$ be a dlt pair. 
Suppose that $X$ is $\mathbb{Q}$-factorial over $W$. 
Let 
$$(X_{0},\Delta_{0})\dashrightarrow (X_{1},\Delta_{1})\dashrightarrow \cdots \dashrightarrow (X_{n},\Delta_{n}) \dashrightarrow \cdots$$
be a sequence of steps of a $(K_{X}+\Delta)$-MMP over $Y$ around $W$ with scaling of an effective $\mathbb{R}$-divisor $A$. 
We quickly recall the argument of special termination by Fujino \cite{fujino-sp-ter} (see also \cite[Remark 2.21]{has-finite}). 

By shrinking $Y$ around $W$, we may assume that $K_{X}+\Delta$ is globally $\mathbb{R}$-Cartier and the number of components of $K_{X}$ and $\Delta$ are finite. 
Let $\Delta=\sum_{j=1}^{n}r_{j}D_{j}$ be the prime decomposition. 
Let $S$ be an lc center of $(X,\Delta)$, and let $(S,\Delta_{S})$ be the dlt pair defined by adjunction $K_{S}+\Delta_{S}=(K_{X}+\Delta)|_{S}$. 
By \cite[Corollary 3.10]{shokurov-flip}, the coefficients of $\Delta_{S}$ belong to the set
$$\mathcal{I}:=\left\{\frac{m-1}{m}+\sum_{j=1}^{n}\frac{r_{j}k_{j}}{m}\;\middle|\;m \in \mathbb{Z}_{>0},\,k_{j}\in \mathbb{Z}_{\geq0} \right\}.$$ 
Define $d_{\mathcal{I}}(S,\Delta_{S})$ to be
$$d_{\mathcal{I}}(S,\Delta_{S}):=\sum_{\alpha \in \mathcal{I}}\#\left\{E\,\middle|\, a(E,S,\Delta_{S})<-\alpha\;{\rm and}\; c_{S}(E)\not\subset \lfloor \Delta_{S}\rfloor\right\}$$
Then $d_{\mathcal{I}}(S,\Delta_{S})<\infty$. 

Let $\{Y_{i}\}_{i \geq 0}$ and $\{\phi_{i}\colon X_{i} \dashrightarrow X'_{i}\}_{i\geq 0}$ be the data of the log MMP. 
Suppose that for every $i \geq 0$, the biholomorphic locus of the bimeromorphic map $X_{0}\times_{Y_{0}}Y_{i}\dashrightarrow X_{i}$ intersects general points of $S\times_{Y_{0}}Y_{i}$. 
Let $S_{i}$ be the corresponding lc center of $(X_{i},\Delta_{i})$, and let $(S_{i},\Delta_{S_{i}})$ be the dlt pair defined by adjunction. 
Suppose that for any $i$ and lc center $C_{i}$ of $(S_{i},\Delta_{S_{i}})$, the restriction of $\phi_{i}$ to $C_{i}$ is biholomorphic to its image. 
Note that any lc center of $(S_{i},\Delta_{S_{i}})$ is an lc center of $(X_{i},\Delta_{i})$ of dimension less than ${\rm dim}\,S$. 
By Lemma \ref{lem--basic-2}, $\phi_{i}$ is an biholomorphism on a Zariski open subset of $X_{i}$ containing all lc centers of $(S_{i},\Delta_{S_{i}})$. 
By the argument in \cite[Proof of Proposition 4.2.14]{fujino-sp-ter}, we have
$$d_{\mathcal{I}}(S_{i},\Delta_{S_{i}}) \geq d_{\mathcal{I}}(S_{i+1},\Delta_{S_{i+1}})$$
and the strict inequality holds if the induced bimeromorphic map $S_{i}\times_{Y_{i}}Y_{i+1}\dashrightarrow S_{i+1}$ extracts a divisor. 
Therefore, there exists an index $i_{0}$ such that $S_{i}\times_{Y_{i}}Y_{i+1}\dashrightarrow S_{i+1}$ is a bimeromorphic contraction for all $i \geq i_{0}$. 
Then the strict transforms of $A_{i}|_{S_{i}\times_{Y_{i}}Y_{i+1}}$ and $\Delta_{S_{i}}|_{S_{i}\times_{Y_{i}}Y_{i+1}}$ to $S_{i+1}$ are $A_{i+1}$ and $\Delta_{S_{i+1}}$, respectively (cf.~\cite[Lemma 4.2.15]{fujino-sp-ter}). 
Let $S'_{i}$ be a small $\mathbb{Q}$-factorialization of $S_{i}$ over a neighborhood of $W$ (\cite[Theorem 1.24]{fujino-analytic-bchm}). 
By considering the non-increasing sequence $\{\rho(S'_{i}/Y;W_{i})\}_{i \geq 0}$ and replacing $i_{0}$, we may assume that $S_{i}\times_{Y_{i}}Y_{i+1}\dashrightarrow S_{i+1}$ is small for all $i \geq i_{0}$. 

As the first conclusion, we get the following statement. 

\begin{thm}\label{thm--sp-ter-1}
Let $\pi \colon X \to Y$ be a contraction from a normal analytic variety $X$ to a Stein space $Y$, and let $W \subset Y$ be a connected compact subset such that $\pi$ and $W$ satisfy (P). 
Let $(X,\Delta)$ be a dlt pair. 
Suppose that $X$ is $\mathbb{Q}$-factorial over $W$. 
Let 
$$(X_{0},\Delta_{0})\dashrightarrow (X_{1},\Delta_{1})\dashrightarrow \cdots \dashrightarrow (X_{n},\Delta_{n}) \dashrightarrow \cdots$$
be a sequence of steps of a $(K_{X}+\Delta)$-MMP over $Y$ around $W$ with scaling of an effective $\mathbb{R}$-divisor $A$. 
Let $\{Y_{i}\}_{i \geq 0}$ and $\{\phi_{i}\colon X_{i} \dashrightarrow X'_{i}\}_{i\geq 0}$ be the data of the log MMP. 
Let $S$ be an lc center of $(X,\Delta)$ and $(S,\Delta_{S})$ the dlt pair defined by adjunction $K_{S}+\Delta_{S}=(K_{X}+\Delta)|_{S}$. 
Suppose that for every $i \geq 0$, after shrinking $Y$ to $Y_{i}$ so that the $i$ steps of the log MMP from the beginning are represented by a bimeromorphic contraction $X_{0}\dashrightarrow X_{i}$ over $Y$, then the isomorphic locus of $X_{0}\dashrightarrow X_{i}$ intersects general points of $S$. 
Let $S_{i}$ be the corresponding lc center of $(X_{i},\Delta_{i})$, and let $(S_{i},\Delta_{S_{i}})$ be the dlt pair defined by adjunction. 
We put $A_{S_{i}}:=A_{i}|_{S_{i}}$ and $A_{S_{i+1}}:=A_{i+1}|_{S_{i+1}}$. 
Suppose the following condition:
\begin{itemize}
\item
For any $i$ and lc center $C_{i}$ of $(X_{i},\Delta_{i})$ whose dimension is less than ${\rm dim}\,S$, the restriction of $\phi_{i}$ to $C_{i}$ is biholomorphic to its image. 
\end{itemize}
Then there exists $i_{0}$ satisfying the following; For every $i \geq i_{0}$, after shrinking $Y$ to $Y_{i+1}$ so that the $(i+1)$-th step of the log MMP is represented by a bimeromorphic contraction $X_{i}\dashrightarrow X_{i+1}$ over $Y$, then the induced meromorphic map $S_{i} \dashrightarrow S_{i+1}$ is a small bimeromorphic contraction and the strict transform of $\Delta_{S_{i}}$ {\rm (}resp.~$A_{S_{i}}${\rm )} to $S_{i+1}$ is $\Delta_{S_{i+1}}$ {\rm (}resp.~$A_{S_{i+1}}${\rm )}. 
\end{thm}

By shifting the log MMP, we assume the conclusion of Theorem \ref{thm--sp-ter-1}. 
By shrinking $Y$ around $W$, we may assume that the first step of the MMP is represented by a bimeromorphic contraction over $Y$. 
Over $Y$, we have diagrams
$$
\xymatrix{
(X_{0},\Delta_{0})\ar@{-->}[rr]^{\phi_{0}}\ar[dr]_-{f_{0}}&&(X_{1},\Delta_{1}),\ar[dl]^-{f'_{0}}\\
&Z_{0}
}\qquad 
\xymatrix{
(S_{0},\Delta_{S_{0}})\ar@{-->}[rr]^{\phi_{S_{0}}}\ar[dr]_-{f_{S_{0}}}&&(S_{1},\Delta_{S_{1}}),\ar[dl]^-{f'_{S_{0}}}\\
&Z_{0}
}
$$
where $f_{0}$ is the $(K_{X_{0}}+\Delta_{0})$-negative extremal contraction, $f_{S_{0}}:=f_{0}|_{S_{0}}$, $f'_{S_{0}}:=f'_{0}|_{S_{1}}$, and $\phi_{S_{0}}:=\phi_{0}|_{S_{0}}$.  
By construction, $-(K_{S_{0}}+\Delta_{S_{0}})$ and $K_{S_{1}}+\Delta_{S_{1}}$ are ample over $Z_{0}$ and $\phi_{S_{0}*}\Delta_{S_{0}}=\Delta_{S_{1}}$. 
By \cite[Theorem 1.21]{fujino-analytic-bchm} and shrinking $Y$ around $W$, we get a dlt blow-up 
$g_{0} \colon (T_{0},\Gamma_{0}) \to (S_{0},\Delta_{S_{0}})$
 of $(S_{0},\Delta_{S_{0}})$ such that $T_{0}$ is $\mathbb{Q}$-factorial over $W$. 
Let $H_{Z_{0}}$ be a Cartier divisor on $Z_{0}$ which is ample over $Y$, and let $H_{S_{0}} \geq 0$ be a $\mathbb{Q}$-Cartier divisor on $S_{0}$ such that $H_{S_{0}} \sim_{\mathbb{R}} (3\cdot {\rm dim}\,X) f_{S_{0}}^{*}H_{Z_{0}}$. 
We put $H_{S_{1}}=\phi_{S_{0}*}H_{S_{0}}$. 
By choosing $H_{S_{0}}$ appropriately, we may assume that $(S_{0},\Delta_{S_{0}}+H_{S_{0}})$ and $(S_{1},\Delta_{S_{1}}+H_{S_{1}})$ are dlt pairs. 
Shrinking $Y$ around $W$, we may assume that $K_{S_{1}}+\Delta_{S_{1}}+H_{S_{1}}$ is ample over $Y$, and $(S_{1},\Delta_{S_{1}}+H_{S_{1}})$ is a weak lc model of both $(S_{0},\Delta_{S_{0}}+H_{S_{0}})$ and $(T_{0},\Gamma_{0}+g_{0}^{*}H_{S_{0}})$. 

We put $H_{T_{0}}:=g_{0}^{*}H_{S_{0}}$. 
We run a $(K_{T_{0}}+\Gamma_{0}+H_{T_{0}})$-MMP over $Y$ around $W$ with scaling of an ample divisor. 
By Lemma \ref{lem--mmp-ample-scaling-nefthreshold} and Theorem \ref{thm--termination-mmp} and shrinking $Y$ around $W$, the $(K_{T_{0}}+\Gamma_{0}+H_{T_{0}})$-MMP terminates with a log minimal model $(T_{k_{1}},\Gamma_{k_{1}}+H_{T_{k_{1}}})$ over $Y$ around $W$. 
By Lemma \ref{lem--two-logminmodel}, $K_{T_{k_{1}}}+\Gamma_{k_{1}}+H_{T_{k_{1}}}$ is semi-ample over $Y$ around $W$. 
Then $K_{T_{k_{1}}}+\Gamma_{k_{1}}+H_{T_{k_{1}}}$ defines a contraction over $Y$, and the target of the contraction is biholomorphic to $S_{1}$ because $K_{S_{1}}+\Delta_{S_{1}}+H_{S_{1}}$ is ample over $Y$. 
Let $g_{1} \colon T_{k_{1}} \to S_{1}$ be the contraction. 
Then $K_{T_{k_{1}}}+\Gamma_{k_{1}}+H_{T_{k_{1}}}=g_{1}^{*}(K_{S_{1}}+\Delta_{S_{1}}+H_{S_{1}})$ and $H_{T_{k_{1}}}=g_{1}^{*}H_{S_{1}}$. 
Thus we have $K_{T_{k_{1}}}+\Gamma_{k_{1}}=g_{1}^{*}(K_{S_{1}}+\Delta_{S_{1}})$. 
Moreover, $(T_{k_{1}},\Gamma_{k_{1}})$ is dlt and $T_{k_{1}}$ is $\mathbb{Q}$-factorial over $W$. 
In this way, we get a diagram
$$
\xymatrix{
 (T_{0},\Gamma_{0})\ar[d]_{g_{0}}\ar@{-->}[rr] &&(T_{k_{1}},\Gamma_{k_{1}})\ar[d]^{g_{1}}\\
(S_{0},\Delta_{S_{0}})\ar@{-->}[rr]^{\phi_{S_{0}}}\ar[dr]_-{f_{S_{0}}}&&(S_{1},\Delta_{S_{1}})\ar[dl]^-{f'_{S_{0}}}\\
&Z_{0},
}
$$
such that $T_{0}\dashrightarrow T_{k_{1}}$ is a sequence of steps of a $(K_{T_{0}}+\Gamma_{0}+H_{T_{0}})$-MMP over $Y$ around $W$ with scaling of an ample divisor. 

We set $A_{T_{0}}:=g_{0}^{*}A_{S_{0}}$. 
We check that $(T_{0},\Gamma_{0}) \dashrightarrow (T_{k_{1}},\Gamma_{k_{1}})$ is a sequence of steps of a $(K_{T_{0}}+\Gamma_{0})$-MMP over $Y$ around $W$ with scaling of $A_{T_{0}}$. 
By applying Lemma \ref{lem--relative-mmp} to $T_{0} \to Z_{0} \to Y$, we see that the $(K_{T_{0}}+\Gamma_{0}+H_{T_{0}})$-MMP is also a $(K_{T_{0}}+\Gamma_{0})$-MMP over $Y$ around $W$. 
Put
$$\lambda_{0}:={\rm inf}\{t \in \mathbb{R}_{\geq 0}\,|\, \text{$K_{X_{0}}+\Delta_{0}+tA_{0}$ {\rm is nef over} $W$}\}.$$
Then $K_{X_{0}}+\Delta_{0}+\lambda_{0}A_{0}$ is nef over $W$, and we get $K_{X_{0}}+\Delta_{0}+\lambda_{0}A_{0} \sim_{\mathbb{R},Z_{0}}0$ after shrinking $Y$ around $W$. 
Then $K_{T_{0}}+\Gamma_{0}+\lambda_{0}A_{T_{0}}$ is nef over $W$ and $K_{T_{0}}+\Gamma_{0}+\lambda_{0}A_{T_{0}} \sim_{\mathbb{R},Z_{0}}0$. 
From these facts, we see that $(T_{0},\Gamma_{0}) \dashrightarrow (T_{k_{1}},\Gamma_{k_{1}})$ is a sequence of steps of a $(K_{T_{0}}+\Gamma_{0})$-MMP over $Y$ around $W$ with scaling of $A_{T_{0}}$. 

Repeating the above discussion, we get 
$$
\xymatrix{
(T_{0},\Gamma_{0})\ar[d]_{g_{0}}\ar@{-->}[r] &(T_{k_{1}},\Gamma_{k_{1}})\ar[d]_{g_{1}}\ar@{-->}[r]& \cdots\ar@{-->}[r]&  (T_{k_{i}},\Gamma_{k_{i}})\ar[d]^{g_{i}}\ar@{-->}[r]& \cdots\\
(S_{0},\Delta_{S_{0}})&(S_{1},\Delta_{S_{1}})& \cdots&  (S_{i},\Delta_{S_{i}})& \cdots
}
$$
such that 
\begin{itemize}
\item
the sequence of horizontal arrows is a sequence of steps of a $(K_{T_{0}}+\Gamma_{0})$-MMP over $Y$ around $W$ with scaling of $A_{T_{0}}$, and 
\item
each $g_{i} \colon T_{k_{i}} \to S_{i}$ is a dlt blow-up of $(S_{i},\Delta_{S_{i}})$ and $T_{k_{i}}$ are $\mathbb{Q}$-factorial over $W$. 
\end{itemize}
Note that $(T_{k_{i}},\Gamma_{k_{i}}) \dashrightarrow (T_{k_{i+1}},\Gamma_{k_{i+1}})$ is not necessarily one step of a $(K_{T_{k_{i}}}+\Gamma_{k_{i}})$-MMP. 

For each $i \geq 0$, we define
\begin{equation*}
\begin{split}
\lambda_{i}:=&{\rm inf}\{t \in \mathbb{R}_{\geq 0}\,|\, \text{$K_{X_{i}}+\Delta_{i}+tA_{i}$ {\rm is nef over} $W$}\},\qquad{\rm and}\\
\mu_{i}:=&{\rm inf}\{t \in \mathbb{R}_{\geq 0}\,|\, \text{$K_{T_{k_{i}}}+\Gamma_{k_{i}}+tA_{T_{k_{i}}}$ {\rm is nef over} $W$}\}.
\end{split}
\end{equation*}
Then we have $\lambda_{i} \geq \mu_{i}$ for all $i$. 
Therefore, if ${\rm lim}_{i \to \infty}\lambda_{i}=0$, then ${\rm lim}_{i \to \infty}\mu_{i}=0$. 
If we assume in addition that $(S_{0},\Delta_{S_{0}})$ has a log minimal model over $Y$ around $W$, then so does $(T_{0},\Gamma_{0})$, and therefore Theorem \ref{thm--termination-mmp} implies that the $(K_{T_{0}}+\Gamma_{0})$-MMP terminates. 
This implies that $K_{S_{i}}+\Delta_{S_{i}}$ are nef over $W$ for all $i \gg 0$. 
Then, for every $i \gg0$, after shrinking $Y$ around $W$ with respect to $i$, the induced bimeromorphic contraction 
$$(S_{i},\Delta_{S_{i}}) \dashrightarrow (S_{i+1},\Delta_{S_{i+1}})$$
is a biholomorphism over $Y$. 
By Theorem \ref{lem--basic-2}, the $(i+1)$-th step of the $(K_{X}+\Delta)$-MMP over $Y$ around $W$ does not modify a Zariski open neighborhood of $S_{i}$ for all $i \gg i_{0}$.  

By the above argument, we obtain the following statement. 

\begin{thm}\label{thm--sp-ter-2}
Let $\pi \colon X \to Y$ be a contraction from a normal analytic variety $X$ to a Stein space $Y$, and let $W \subset Y$ be a connected compact subset such that $\pi$ and $W$ satisfy (P). 
Let $(X,\Delta)$ be a dlt pair. 
Suppose that $X$ is $\mathbb{Q}$-factorial over $W$. 
Let 
$$(X_{0},\Delta_{0})\dashrightarrow (X_{1},\Delta_{1})\dashrightarrow \cdots \dashrightarrow (X_{n},\Delta_{n}) \dashrightarrow \cdots$$
be a sequence of steps of a $(K_{X}+\Delta)$-MMP over $Y$ around $W$ with scaling of an effective $\mathbb{R}$-divisor $A$. 
Let $\{Y_{i}\}_{i \geq 0}$ and $\{\phi_{i}\colon X_{i} \dashrightarrow X'_{i}\}_{i\geq 0}$ be the data of the log MMP. 
Let $S$ be an lc center of $(X,\Delta)$ and $(S,\Delta_{S})$ the dlt pair defined by adjunction $K_{S}+\Delta_{S}=(K_{X}+\Delta)|_{S}$. 
Suppose that for every $i \geq 0$, after we shrink $Y$ to $Y_{i}$ so that the $(K_{X_{0}}+\Delta_{0})$-MMP $(X_{0},\Delta_{0}) \dashrightarrow (X_{i},\Delta_{i})$ is represented by a bimeromorphic contraction $X_{0}\dashrightarrow X_{i}$ over $Y$, the isomorphic locus of $X_{0}\dashrightarrow X_{i}$ intersects general points of $S$. 
Let $S_{i}$ be the corresponding lc center of $(X_{i},\Delta_{i})$, and let $(S_{i},\Delta_{S_{i}})$ be the dlt pair defined by adjunction. 
Suppose the following conditions: 
\begin{itemize}
\item
For any $i$ and lc center $C_{i}$ of $(X_{i},\Delta_{i})$ whose dimension is less than ${\rm dim}\,S$, the restriction of $\phi_{i}$ to $C_{i}$ is biholomorphic to its image, and
\item 
if we put $\lambda_{i}:={\rm inf}\{t \in \mathbb{R}_{\geq 0}\,|\, \text{$K_{X_{i}}+\Delta_{i}+tA_{i}$ {\rm is nef over} $W$}\}$, then ${\rm lim}_{i \to \infty}\lambda_{i}=0$.
\end{itemize}
Then there exists an $i_{0}$ satisfying the following; If $(S_{i_{0}},\Delta_{i_{0}})$ has a log minimal model over $Y$ around $W$ after shrinking $Y$, then the $(i+1)$-th step of the $(K_{X}+\Delta)$-MMP does not modify a Zariski open neighborhood of $S_{i}$ for all $i \gg i_{0}$.  
\end{thm}

\subsection{Lift of MMP}\label{subsection--lift-mmp}

In this subsection, we discuss the lift of MMP (\cite[Remark 2.9]{birkar-flip}). 

Let $\pi \colon X \to Y$ be a contraction from a normal analytic variety $X$ to a Stein space $Y$, and let $W \subset Y$ be a connected compact subset such that $\pi$ and $W$ satisfy (P). 
Let $(X,\Delta)$ be an lc pair, and let 
$$(X_{0},\Delta_{0})\dashrightarrow (X_{1},\Delta_{1})\dashrightarrow \cdots \dashrightarrow (X_{n},\Delta_{n}) \dashrightarrow \cdots$$
be a sequence of steps of a $(K_{X}+\Delta)$-MMP over $Y$ around $W$ with scaling of an effective $\mathbb{R}$-divisor $A$. 

By shrinking $Y$ around $W$, we may assume that the first step of the MMP is represented by a bimeromorphic contraction over $Y$. 
Over $Y$, we have a diagram
$$
\xymatrix{
(X_{0},\Delta_{0})\ar@{-->}[rr]^{\phi_{0}}\ar[dr]_-{f_{0}}&&(X_{1},\Delta_{1})\ar[dl]^-{f'_{0}}\\
&Z_{0}
}
$$
where $f_{0}$ is the $(K_{X_{0}}+\Delta_{0})$-negative extremal contraction. 
Note that $K_{X_{1}}+\Delta_{1}$ is ample over $Z_{0}$. 
By \cite[Theorem 1.21]{fujino-analytic-bchm} and shrinking $Y$ around $W$, we get a dlt blow-up 
$g_{0} \colon (\widetilde{X}_{0},\widetilde{\Delta}_{0}) \to (X_{0},\Delta_{0})$
 of $(X_{0},\Delta_{0})$ such that $\widetilde{X}_{0}$ is $\mathbb{Q}$-factorial over $W$. 
Let $H_{Z_{0}}$ be a Cartier divisor on $Z_{0}$ which is ample over $Y$, and let $H_{0} \geq 0$ be a $\mathbb{Q}$-Cartier divisor on $X_{0}$ such that $H_{0} \sim_{\mathbb{R}} (3\cdot {\rm dim}\,X) f_{0}^{*}H_{Z_{0}}$. 
We put $H_{{1}}=\phi_{0*}H_{0}$. 
By choosing $H_{0}$ appropriately, we may assume that $(X_{0},\Delta_{0}+H_{0})$ and $(X_{1},\Delta_{1}+H_{1})$ are lc pairs. 
Shrinking $Y$ around $W$, we may assume that $K_{X_{1}}+\Delta_{1}+H_{1}$ is ample over $Y$, and $(X_{1},\Delta_{1}+H_{1})$ is a weak lc model of both $(X_{0},\Delta_{0}+H_{0})$ and $(\widetilde{X}_{0},\widetilde{\Delta}_{0}+g_{0}^{*}H_{0})$. 

Put $\widetilde{H}_{0}:=g_{0}^{*}H_{0}$. 
We run a $(K_{\widetilde{X}_{0}}+\widetilde{\Delta}_{0}+\widetilde{H}_{0})$-MMP over $Y$ around $W$ with scaling of an ample divisor. 
By Theorem \ref{lem--mmp-ample-scaling-nefthreshold} and Theorem \ref{thm--termination-mmp} and shrinking $Y$ around $W$, the $(K_{\widetilde{X}_{0}}+\widetilde{\Delta}_{0}+\widetilde{H}_{0})$-MMP terminates with a log minimal model $(\widetilde{X}_{k_{1}},\widetilde{\Delta}_{k_{1}}+\widetilde{H}_{k_{1}})$ over $Y$ around $W$. 
By Lemma \ref{lem--two-logminmodel}, $K_{\widetilde{X}_{k_{1}}}+\widetilde{\Delta}_{k_{1}}+\widetilde{H}_{k_{1}}$ is semi-ample over $Y$ around $W$. 
Therefore, the divisor defines a contraction over $Y$, and the target of the contraction is biholomorphic to $X_{1}$ because $K_{X_{1}}+\Delta_{1}+H_{1}$ is ample over $Y$. 
Let $g_{1} \colon \widetilde{X}_{k_{1}} \to X_{1}$ be the contraction. 
Then $K_{\widetilde{X}_{k_{1}}}+\widetilde{\Delta}_{k_{1}}+\widetilde{H}_{k_{1}}=g_{1}^{*}(K_{X_{1}}+\Delta_{1}+H_{1})$ and $\widetilde{H}_{k_{1}}=g_{1}^{*}H_{1}$. 
Thus we have $K_{\widetilde{X}_{k_{1}}}+\widetilde{\Delta}_{k_{1}}=g_{1}^{*}(K_{X_{1}}+\Delta_{1})$. 
Moreover, $(\widetilde{X}_{k_{1}},\widetilde{\Delta}_{k_{1}})$ is dlt and $\widetilde{X}_{k_{1}}$ is $\mathbb{Q}$-factorial over $W$. 
In this way, we get a diagram
$$
\xymatrix{
 (\widetilde{X}_{0},\widetilde{\Delta}_{0})\ar[d]_{g_{0}}\ar@{-->}[rr] &&(\widetilde{X}_{k_{1}},\widetilde{\Delta}_{k_{1}})\ar[d]^{g_{1}}\\
(X_{0},\Delta_{0})\ar@{-->}[rr]^{\phi_{0}}\ar[dr]_-{f_{0}}&&(X_{1},\Delta_{1})\ar[dl]^-{f'_{0}}\\
&Z_{0},
}
$$
such that $\widetilde{X}_{0}\dashrightarrow \widetilde{X}_{k_{1}}$ is a sequence of steps of a $(K_{\widetilde{X}_{0}}+\widetilde{\Delta}_{0}+\widetilde{H}_{0})$-MMP over $Y$ around $W$ with scaling of an ample divisor. 

We set $\widetilde{A}_{0}:=g_{0}^{*}A_{0}$. 
We check that $(\widetilde{X}_{0},\widetilde{\Delta}_{0}) \dashrightarrow (\widetilde{X}_{k_{1}},\widetilde{\Delta}_{k_{1}})$ is a sequence of steps of a $(K_{\widetilde{X}_{0}}+\widetilde{\Delta}_{0})$-MMP over $Y$ around $W$ with scaling of $\widetilde{A}_{0}$. 
By applying Lemma \ref{lem--relative-mmp} to $\widetilde{X}_{0} \to Z_{0} \to Y$, we see that the $(K_{\widetilde{X}_{0}}+\widetilde{\Delta}_{0}+\widetilde{H}_{0})$-MMP is also a $(K_{\widetilde{X}_{0}}+\widetilde{\Delta}_{0})$-MMP over $Y$ around $W$. 
Put
$$\lambda_{0}:={\rm inf}\{t \in \mathbb{R}_{\geq 0}\,|\, \text{$K_{X_{0}}+\Delta_{0}+tA_{0}$ {\rm is nef over} $W$}\}.$$
Then $K_{X_{0}}+\Delta_{0}+\lambda_{0}A_{0}$ is nef over $W$, and we get $K_{X_{0}}+\Delta_{0}+\lambda_{0}A_{0} \sim_{\mathbb{R},Z_{0}}0$ after shrinking $Y$ around $W$. 
This fact implies that $(\widetilde{X}_{0},\widetilde{\Delta}_{0}) \dashrightarrow (\widetilde{X}_{k_{1}},\widetilde{\Delta}_{k_{1}})$ is a sequence of steps of a $(K_{\widetilde{X}_{0}}+\widetilde{\Delta}_{0})$-MMP over $Y$ around $W$ with scaling of $\widetilde{A}_{0}$. 

Repeating the above discussion, we get 
$$
\xymatrix{
(\widetilde{X}_{0},\widetilde{\Delta}_{0})\ar[d]_{g_{0}}\ar@{-->}[r] &(\widetilde{X}_{k_{1}},\widetilde{\Delta}_{k_{1}})\ar[d]_{g_{1}}\ar@{-->}[r]& \cdots\ar@{-->}[r]&  (\widetilde{X}_{k_{i}},\widetilde{\Delta}_{k_{i}})\ar[d]^{g_{i}}\ar@{-->}[r]& \cdots\\
(X_{0},\Delta_{0})\ar@{-->}[r]&(X_{1},\Delta_{1})\ar@{-->}[r]& \cdots\ar@{-->}[r]&  (X_{i},\Delta_{i})\ar@{-->}[r]& \cdots
}
$$
such that 
\begin{itemize}
\item
the sequence of upper horizontal arrows is a sequence of steps of a $(K_{\widetilde{X}_{0}}+\widetilde{\Delta}_{0})$-MMP over $Y$ around $W$ with scaling of $\widetilde{A}_{0}$, and 
\item
each $g_{i} \colon \widetilde{X}_{k_{i}} \to X_{i}$ is a dlt blow-up of $(X_{i},\Delta_{i})$ and $\widetilde{X}_{k_{i}}$ are $\mathbb{Q}$-factorial over $W$. 
\end{itemize}
Note that the arrows are not necessarily a bimeromorphic maps (see Definition \ref{defn--mmp}) and $(\widetilde{X}_{k_{i}},\Delta_{k_{i}}) \dashrightarrow (\widetilde{X}_{k_{i+1}},\Delta_{k_{i+1}})$ is not necessarily one step of a $(K_{\widetilde{X}_{k_{i}}}+\Delta_{k_{i}})$-MMP. 
We call it the {\em lift of the $(K_{X}+\Delta)$-MMP}.

\section{Asymptotic vanishing order and Nakayama--Zariski decomposition}\label{sec4}

In this section, we study relations between the Nakayama--Zariski decomposition and existence of log minimal models for lc pairs. 
For the statements in the algebraic case, see \cite{bhzariski} and \cite[Subsection 2.4]{has-finite}.

\subsection{Definitions and basic properties}

In this subsection, we collect the definition and basic properties of Nakayama--Zariski decomposition.  
We freely use the results in \cite[III, \S 4]{nakayama}. 
For the algebraic case, see \cite{liuxie-relative-nakayama}. 

\begin{defn}\label{defn--asymvanorder}
Let $\pi \colon X \to Y$ be a projective morphism from a non-singular analytic variety $X$ to an analytic space $Y$. 
Let $D$ be a $\pi$-big $\mathbb{R}$-Cartier divisor on $X$. 
Let $P$ be a prime divisor on $X$. 
We define $m_{P}$ to be $\infty$ if $\pi_{*}\mathcal{O}_{X}(\lfloor D\rfloor)=0$, and otherwise
$$m_{P}(D):={\rm max}\set{n \in \mathbb{Z}_{\geq 0}|\text{$\pi_{*}\mathcal{O}_{X}(\lfloor D\rfloor-nP)\hookrightarrow \pi_{*}\mathcal{O}_{X}(\lfloor D\rfloor)$ is an isomorphism}}.$$
We set $\sigma_{P}(D;X/Y)_{\mathbb{Z}}:=m_{P}(D)+{\rm coeff}_{P}(\{D\})$, and we define 
$$\sigma_{P}(D;X/Y):=\underset{m \to \infty}{\rm liminf}\frac{1}{m}\sigma_{P}(mD;X/Y)_{\mathbb{Z}}.$$
\end{defn}

\begin{rem}
By the definition of $\pi$-bigness (\cite[Definition 2.46]{fujino-analytic-bchm}) and the standard argument of divisors (cf.~\cite[II, 3.17. Corollary]{nakayama}), there exists $m_{0} \in \mathbb{Z}_{>0}$ such that $\pi_{*}\mathcal{O}_{X}(\lfloor mD\rfloor) \neq 0$ for all $m \geq m_{0}$. 
Thus the equality
$$\sigma_{P}(D;X/Y):=\underset{m \to \infty}{\rm lim}\frac{1}{m}\sigma_{P}(mD;X/Y)_{\mathbb{Z}}$$
holds. 
\end{rem}

\begin{lem}\label{lem--asymvanorder-basic}
Let $\pi \colon X \to Y$ be a projective morphism from a non-singular analytic variety $X$ to an analytic space $Y$. 
Let $D$ be a $\pi$-big $\mathbb{R}$-Cartier divisor on $X$ and let $P$ be a prime divisor on $X$. 
If $Y$ is Stein, then the equalities
$$\sigma_{P}(D;X/Y)={\rm inf}\set{{\rm coeff}_{P}(E)| E \in |D/Y|_{\mathbb{R}}}={\rm inf}\set{{\rm coeff}_{P}(E)| E \in |D|_{\mathbb{R}}}$$
hold.
\end{lem}

\begin{proof}
By Cartan's Theorem A (\cite[Theorem 2.6 (1)]{fujino-analytic-bchm}), for any $m \in \mathbb{Z}_{>0}$ such that $\pi_{*}\mathcal{O}_{X}(\lfloor mD\rfloor) \neq 0$, the equalities
\begin{equation*}
\begin{split}
\sigma_{P}(mD;X/Y)_{\mathbb{Z}}=&{\rm inf}\{{\rm coeff}_{P}(E)\,|\, E\geq0,\, E  \sim \lfloor mD\rfloor\}+{\rm coeff}_{P}(\{mD\})\\
=&{\rm inf}\{{\rm coeff}_{P}(E')\,|\, E'\geq0,\, E' \sim mD\}
\end{split}
\end{equation*}
hold (see also the discussion in \cite[III,  \S 4.a]{nakayama}). 
From this, we have
\begin{equation*}
\begin{split}
\sigma_{P}(D;X/Y) =&{\rm inf}\{{\rm coeff}_{P}(E)\,|\, E\geq0,\, E \sim_{\mathbb{Q}} D\}.
\end{split}
\end{equation*}
This implies that the following properties hold:
\begin{itemize}
\item
For any $t \in \mathbb{Q}_{>0}$, we have $\sigma_{P}(tD;X/Y)=t\sigma_{P}(D;X/Y)$. 
\item
If $D=D'+B$, where $D'$ is a $\pi$-big $\mathbb{R}$-Cartier divisor and $B$ is effective, then 
$$\sigma_{P}(D;X/Y) \leq \sigma_{P}(D';X/Y)+{\rm coeff}_{P}(B).$$
\item
If $D=D_{1}+D_{2}$, where $D_{1}$ and $D_{2}$ are $\pi$-big $\mathbb{R}$-Cartier divisors, then
$$\sigma_{P}(D;X/Y) \leq \sigma_{P}(D_{1};X/Y)+\sigma_{P}(D_{2};X/Y).$$
\end{itemize}

We will check that $\sigma_{P}(D;X/Y)=0$ when $D$ is $\pi$-ample. 
When $D$ is $\pi$-ample, we can write $D=\sum_{i=1}^{l}r_{i}A_{i}$ for some $r_{i}\in \mathbb{R}_{>0}$ and $\pi$-ample Cartier divisors $A_{i}$ on $X$. 
By the third property, it is sufficient to check that $\sigma_{P}(r_{i}A_{i};X/Y)=0$ holds for every $1\leq i \leq l$. 
By \cite[III, 4.1. Lemma]{nakayama} and shrinking $Y$ around a point of $\pi(P)$, we may assume that every $A_{i}$ has only finitely many components. 
Then the equality $\sigma_{P}(r_{i}A_{i};X/Y)=0$ follows from the standard argument. 
In this way, we have $\sigma_{P}(D;X/Y)=0$ when $D$ is $\pi$-ample. 

Since $\sigma_{P}(D;X/Y) ={\rm inf}\{{\rm coeff}_{P}(E)\,|\, E\geq0,\, E \sim_{\mathbb{Q}} D\}$, it is easy to see that
$$\sigma_{P}(D;X/Y) \geq {\rm inf}\{{\rm coeff}_{P}(E)\,|\, E \in |D|_{\mathbb{R}}\}.
$$
We also have
$${\rm inf}\{{\rm coeff}_{P}(E)\,|\, E \in |D|_{\mathbb{R}}\} \geq{\rm inf}\{{\rm coeff}_{P}(E)\,|\, E \in |D/Y|_{\mathbb{R}}\}.$$
Pick $G \geq 0$ such that $D-G$ is $\pi$-ample. 
For any $E \in |D/Y|_{\mathbb{R}}$ and $\epsilon \in \mathbb{Q}_{>0}$, the divisor $D-E+\epsilon(D-G)$ is $\pi$-ample, and therefore
we have
\begin{equation*}
\begin{split}
(1+\epsilon)\sigma_{P}(D;X/Y) =& \sigma_{P}((1+\epsilon)D;X/Y)\\
=&\sigma_{P}(D-E+\epsilon(D-G)+(E+\epsilon G);X/Y)\\
\leq & \sigma_{P}((D-E+\epsilon(D-G));X/Y)+ {\rm coeff}_{P}(E)+\epsilon\cdot {\rm coeff}_{P}(G)\\
=&{\rm coeff}_{P}(E)+\epsilon\cdot {\rm coeff}_{P}(G).
\end{split}
\end{equation*}
Taking the limit $\epsilon \to 0$, we have
$\sigma_{P}(D;X/Y) \leq  {\rm coeff}_{P}(E)$ for any $E \in |D/Y|_{\mathbb{R}}$. 
Hence
$$ {\rm inf}\{{\rm coeff}_{P}(E)\,|\, E \in |D/Y|_{\mathbb{R}}\} \geq \sigma_{P}(D;X/Y).$$
Therefore Lemma \ref{lem--asymvanorder-basic} holds. 
\end{proof}

\begin{rem}\label{rem--quasiproj}
In Lemma \ref{lem--asymvanorder-basic}, the Steinness of base space is crucial. 
Even in the algebraic case \cite[Section 3]{liuxie-relative-nakayama}, the quasi-projectivity of the base variety is crucial. 
For example, consider a complete but not projective variety $Y$ with the trivial Picard group (see, for example, \cite[Example 4.2.13]{toric-varieties}). 
Let $X \to Y$ be a blow-up at a smooth point, and let $E$ be the exceptional divisor. 
Then $-E$ is clearly big over $Y$, but there is no effective $\mathbb{R}$-divisor $E'$ such that $E' \sim_{\mathbb{R},\,Y}-E$ because $Y$ has the trivial Picard group. 
In this case, the invariant 
$${\rm inf}\set{{\rm coeff}_{P}(E')| E' \in |-E/Y|_{\mathbb{R}}}$$
cannot be defined since $|-E/Y|_{\mathbb{R}}$ is empty. 
\end{rem}

\begin{lem}\label{lem--asymvanorder-basic2}
Let $\pi \colon X \to Y$ be a projective morphism from a non-singular analytic variety $X$ to an analytic space $Y$. 
Let $D$ be a $\pi$-big $\mathbb{R}$-Cartier divisor on $X$, $P$ a prime divisor on $X$,  and $A$ a $\pi$-big $\mathbb{R}$-divisor on $X$. 
If $Y$ is Stein, then the equality
$$\sigma_{P}(D;X/Y)=\underset{\epsilon \to 0+}{\rm lim}\sigma_{P}(D+\epsilon A;X/Y)$$
holds.
\end{lem}

\begin{proof}
By Lemma \ref{lem--asymvanorder-basic}, for any $\epsilon \in \mathbb{R}_{>0}$ we have
$$\sigma_{P}(D+\epsilon A;X/Y) \leq \sigma_{P}(D;X/Y)+\epsilon \sigma_{P}(A;X/Y).$$
By taking the limit, we have
$$\underset{\epsilon \to 0+}{\rm lim}\sigma_{P}(D+\epsilon A;X/Y) \leq \sigma_{P}(D;X/Y).$$
On the other hand, take $r \in \mathbb{R}_{>0}$ such that $rD-A \sim_{\mathbb{R}}E$ for some effective $\mathbb{R}$-divisor $E$ on $X$ (\cite[Lemma 2.53]{fujino-analytic-bchm}). 
Then, for any $\epsilon \in \mathbb{R}_{>0}$ we have
\begin{equation*}
\begin{split}
\sigma_{P}(D;X/Y) \leq(1+r \epsilon) \sigma_{P}(D;X/Y)=&\sigma_{P}(D+\epsilon A+\epsilon(rD-A);X/Y)\\
\leq & \sigma_{P}(D+\epsilon A;X/Y)+ \epsilon \cdot {\rm coeff}_{P}(E). 
\end{split}
\end{equation*}
From this, we have
$$\sigma_{P}(D;X/Y) \leq \underset{\epsilon \to 0+}{\rm lim}\sigma_{P}(D+\epsilon A;X/Y).$$
Therefore Lemma \ref{lem--asymvanorder-basic2} holds true. 
\end{proof}

\begin{defn}[Asymptotic vanishing order, {\cite[III, \S 4]{nakayama}}]\label{defn--asymvanorder-pseudoeff}
Let $\pi \colon X \to Y$ be a projective morphism from a normal analytic variety $X$ to a Stein space $Y$. 
Let $D$ be a $\pi$-pseudo-effective $\mathbb{R}$-Cartier divisor on $X$, and let $P$ be a prime divisor over $X$. 
Then the {\em asymptotic vanishing order of $P$ over $Y$}, denoted by $\sigma_{P}(D;X/Y)$, is defined as follows: 
We take a resolution $f \colon X' \to X$ of $X$ such that $P$ appears as a prime divisor on $X'$, and we take a $(\pi \circ f)$-ample $\mathbb{R}$-divisor $A'$ on $X'$. 
Then we define
$$\sigma_{P}(D;X/Y):=\underset{\epsilon \to 0+}{\rm lim}\sigma_{P}(f^{*}D+\epsilon A';X'/Y).$$
As the algebraic case, we can check that $\sigma_{P}(D;X/Y)$ does not depend on $f \colon X' \to X$ and $A'$. 
However, we may have $\sigma_{P}(D;X/Y) = \infty$. 
\end{defn}

\begin{thm}\label{thm--asymvanorder-pseudoeff-basic}
Let $\pi \colon X \to Y$ be a projective morphism from a normal analytic variety $X$ to a Stein space $Y$,  
and let $P$ be a prime divisor over $X$. 
Then the following properties hold.
\begin{enumerate}[(1)]
\item\label{thm--asymvanorder-pseudoeff-basic-(1)}
For a $\pi$-pseudo-effective $\mathbb{R}$-Cartier divisor $D$ on $X$ and a positive real number $r$, we have 
$$\sigma_{P}(rD;X/Y)=r\sigma_{P}(D;X/Y).$$  
\item\label{thm--asymvanorder-pseudoeff-basic-(2)}
Let $W \subset Y$ be a subset. 
Suppose that $P$ is a prime divisor on $X$ and $\pi(P) \cap W \neq \emptyset$. 
For any $\mathbb{R}$-Cartier divisor $D$ on $X$, if $D$ is a limit of movable divisors over $W$ then 
$$\sigma_{P}(D;X/Y)=0.$$ 
\item\label{thm--asymvanorder-pseudoeff-basic-(3)}
For two $\pi$-pseudo-effective $\mathbb{R}$-Cartier divisors $D_{1}$ and $D_{2}$ on $X$, we have 
$$\sigma_{P}(D_{1}+D_{2};X/Y) \leq \sigma_{P}(D_{1};X/Y)+\sigma_{P}(D_{2};X/Y).$$
\item\label{thm--asymvanorder-pseudoeff-basic-(4)}
For two $\pi$-pseudo-effective $\mathbb{R}$-Cartier divisors $D_{1}$ and $D_{2}$ on $X$, we have 
$$\sigma_{P}(D_{1};X/Y) = \underset{\epsilon \to 0+}{\rm lim}\sigma_{P}(D_{1}+\epsilon D_{2};X/Y).$$
\item\label{thm--asymvanorder-pseudoeff-basic-(5)}
Let $f \colon X' \to X$ be a projective bimeromorphism from a normal analytic variety $X'$. 
For a $\pi$-pseudo-effective $\mathbb{R}$-Cartier divisor $D$ on $X$ and an effective $f$-exceptional $\mathbb{R}$-Cartier divisor $E$ on $X'$, we have
$$\sigma_{P}(f^{*}D+E;X'/Y)=\sigma_{P}(D;X/Y)+{\rm ord}_{P}(E).$$
\item\label{thm--asymvanorder-pseudoeff-basic-(6)}
Suppose that $P$ is a prime divisor on $X$. 
Let $U \subset Y$ be a Stein open subset such that $U \cap \pi(P) \neq \emptyset$. 
We put $X_{U}:=\pi^{-1}(U)$ and let $P|_{X_{U}}=\bigcup_{j}P_{j}$ be the irreducible decomposition. 
For any $\pi$-pseudo-effective $\mathbb{R}$-Cartier divisor $D$ on $X$, all $P_{j}$ satisfy 
$$\sigma_{P}(D;X/Y)=\sigma_{P_{j}}(D|_{X_{U}};X_{U}/U).$$
\end{enumerate}
\end{thm}

\begin{proof}
The properties (\ref{thm--asymvanorder-pseudoeff-basic-(1)})--(\ref{thm--asymvanorder-pseudoeff-basic-(5)}) easily follow as in the algebraic case. 
Thus we only need to check (\ref{thm--asymvanorder-pseudoeff-basic-(6)}). 
By Definition \ref{defn--asymvanorder-pseudoeff}, we may assume that $D$ is $\pi$-big. 
By replacing $X$ with a resolution of $X$, we may assume that $X$ is non-singular. 
By Definition \ref{defn--asymvanorder}, it is sufficient to prove that the equality $m_{P}(lD)=m_{P_{j}}(lD|_{X_{U}})$ holds for all $j$ and $l\gg0$. 
By Cartan's Theorem A (\cite[Theorem 2.6 (1)]{fujino-analytic-bchm}), for any $l \in \mathbb{Z}_{>0}$ such that $\pi_{*}\mathcal{O}_{X}(\lfloor lD\rfloor) \neq 0$, the equalities
\begin{equation*}
\begin{split}
m_{P}(lD)=&{\rm inf}\{{\rm coeff}_{P}(E)\,|\, E\geq0,\, E  \sim \lfloor lD\rfloor\}, \qquad {\rm and}\\
m_{P_{j}}(lD|_{X_{U}})=&{\rm inf}\{{\rm coeff}_{P_{j}}(E')\,|\, E'\geq0,\, E'  \sim \lfloor lD|_{X_{U}}\rfloor\}. 
\end{split}
\end{equation*}
We fix an index $j$. 
It is easy to check that
$$m_{P}(lD) \geq m_{P_{j}}(lD|_{X_{U}}).$$
We take an effective divisor $E' \sim \lfloor lD|_{X_{U}}\rfloor$ on $X_{U}$ such that ${\rm coeff}_{P_{j}}(E')=m_{P_{j}}(lD|_{X_{U}})$. 
By Cartan's Theorem A (\cite[Theorem 2.6 (1)]{fujino-analytic-bchm}), there is an open subset $V \subset U$, an effective divisor $E \sim \lfloor lD \rfloor$, and $\varphi \in \mathcal{O}_{Y}(V)$ such that $V \cap (\pi|_{X_{U}})(P_{j}) \neq \emptyset$ and if we put $X_{V}:=\pi^{-1}(V)$ then $E|_{X_{V}}+{\rm div}(\pi^{*}\varphi)=E'|_{X_{V}}$. 
Then ${\rm coeff}_{P}(E) \leq {\rm coeff}_{P_{j}}(E')$, and therefore we have
$$m_{P}(lD) \leq m_{P_{j}}(lD|_{X_{j}}).$$
Thus we have $m_{P}(lD) = m_{P_{j}}(lD|_{X_{j}})$. 
Therefore, (\ref{thm--asymvanorder-pseudoeff-basic-(6)}) holds.    
\end{proof}

\begin{lem}[{\cite[III, 4.2. Lemma]{nakayama}}]\label{lem--nakayama-III4.2}
Let $\pi \colon X \to Y$ be a projective morphism from a normal analytic variety $X$ to a Stein space $Y$. 
Let $D$ be a $\pi$-pseudo-effective $\mathbb{R}$-Cartier divisor on $X$ and let $P_{1},\,\cdots,\,P_{l}$ be mutually distinct prime divisors on $X$. 
Then the following properties hold. 
\begin{itemize}
\item
For any $1 \leq i \leq l$ and any real number $s_{i}$ such that $0 \leq s_{i} \leq \sigma_{P_{i}}(D; X/Y)$, we have
$$\sigma_{P_{i}}\Bigl( D-\sum_{j=1}^{l}s_{j}P_{j};X/Y\Bigr) = \sigma_{P_{i}}(D; X/Y) - s_{i}.$$ 
\item
If $\sigma_{P_{i}}(D; X/Y)>0$ for all $1 \leq i \leq l$, then
$$\sigma_{P_{i}}\Bigl(\sum_{j=1}^{l}r_{j}P_{j};X/Y\Bigr) = r_{i}$$
for any nonnegative real numbers $r_{1},\,\cdots,\, r_{l}$. 
\end{itemize}
\end{lem}

\begin{thm}\label{thm--finite-negativepart}
Let $\pi\colon X \to Y$ be a projective morphism from a normal analytic variety $X$ to a Stein space $Y$.  
Let $W \subset Y$ be a compact subset. 
Fix a resolution $\widetilde{X} \to X$ of $X$. 
Let $D$ be a $\pi$-pseudo-effective $\mathbb{R}$-Cartier divisor on $X$ and let $l$ be (possibly infinite) the number of prime divisors $P$ on $X$ satisfying $\pi(P) \cap W \neq \emptyset$ and $\sigma_{P}(D; X/Y) >0$. 
Then the inequality $l \leq \rho (\widetilde{X}/Y;W)$ holds. 
In particular, if $W \cap Z$ has only finitely many connected components for any analytic subset $Z \subset Y$ which is defined over an open neighborhood of $W$ (see also the condition $({\rm P}4)$ in \cite{fujino-analytic-bchm}), then $l<\infty$. 
\end{thm}

\begin{proof}
By replacing $X$ with $\widetilde{X}$, we may assume that $X$ is non-singular. 
We may assume $\rho (\widetilde{X}/Y;W) < \infty$ because otherwise there is nothing to prove. 
If $l >\rho (\widetilde{X}/Y;W)$, then there are prime divisors $P_{1},\,\cdots,\, P_{l}$ on $X$ and non-zero elements $a_{1},\,\cdots,\,a_{l} \in \mathbb{R}$ such that $\pi(P_{i})\cap W \neq \emptyset$, $\sigma_{P_{i}}(D; X/Y) >0$ for all $1\leq i \leq l$, and $\sum_{i=1}^{l}a_{i}P_{i}$ is numerically trivial over $W$. 
We put 
$$Q:=\sum_{a_{i}>0}a_{i}P_{i} \qquad {\rm and} \qquad Q':=-\sum_{a_{i}<0}a_{i}P_{i}.$$
Then neither $Q$ nor $Q'$ is a zero divisor and $Q \equiv_{W}Q'$. 
By Theorem \ref{thm--asymvanorder-pseudoeff-basic} (\ref{thm--asymvanorder-pseudoeff-basic-(2)}), we have $\sigma_{P_{1}}(Q-Q';X/Y)=0$. 
By Theorem \ref{thm--asymvanorder-pseudoeff-basic} (\ref{thm--asymvanorder-pseudoeff-basic-(3)}) and Lemma \ref{lem--nakayama-III4.2}, we have
\begin{equation*}
\begin{split}
0<&\sigma_{P_{1}}(Q;X/Y)=\sigma_{P_{1}}(Q'+(Q-Q');X/Y)\\
\leq &\sigma_{P_{1}}(Q';X/Y)+\sigma_{P_{1}}(Q-Q';X/Y)=0+0=0.
\end{split}
\end{equation*}
Thus we get a contradiction. 
From this, Theorem \ref{thm--finite-negativepart} follows. 
\end{proof}

\begin{defn}[cf.~{\cite[III,  \S 4.a]{nakayama}}]\label{defn--nakayama-zariski-decomp-negative}
Let $\pi \colon X \to Y$ be a projective morphism from a normal analytic variety $X$ to a Stein space $Y$, and
let $W \subset Y$ be a subset. 
Let $D$ be a $\pi$-pseudo-effective $\mathbb{R}$-Cartier divisor on $X$. 
Then $N_{\sigma}(D; X/Y, W)$ is defined by the formal sum 
$$N_{\sigma}(D; X/Y, W):=\sum_{\substack{P\,:\, {\rm prime\,divisor\, on\,}X\\W \cap \pi(P) \neq \emptyset}}\sigma_{P}(D;X/Y)P.$$
\end{defn}

\begin{lem}\label{lem--nagativepart-basic}
Let $\pi \colon X \to Y$ be a projective morphism from a normal analytic variety $X$ to a Stein space $Y$, and
let $W \subset Y$ be a subset. 
Let $D$ be a $\pi$-pseudo-effective $\mathbb{R}$-Cartier divisor on $X$. 
Then the following properties hold.
\begin{itemize}
\item
If $D \sim_{\mathbb{R},Y}E$ for some effective $\mathbb{R}$-divisor on $X$, then we have 
$$\sigma_{P}(D;X/Y) \leq {\rm coeff}_{P}(E)$$
for all prime divisors $P$ on $X$. 
In particular, $N_{\sigma}(D; X/Y, W)$ is an $\mathbb{R}$-divisor on $X$ and $N_{\sigma}(D; X/Y, W) \leq E$. 
\item
If $W$ is compact and $W \cap Z$ has only finitely many connected components for any analytic subset $Z$ which is defined over an open neighborhood of $W$ (see also the condition $({\rm P}4)$ in \cite{fujino-analytic-bchm}), then $N_{\sigma}(D; X/Y, W)$ has only finitely many components. 
\item
Let $U \subset Y$ be a Stein open subset containing $W$, and we put $X_{U}=\pi^{-1}(U)$. 
If $N_{\sigma}(D; X/Y, W)$ is well defined as an $\mathbb{R}$-divisor, then there exists a Zariski open subset $U'$ of $U$ such that $U' \supset W$ and putting $X_{U'}=\pi^{-1}(U')$ then
$$N_{\sigma}(D; X/Y, W)|_{X_{U'}}=N_{\sigma}(D|_{X_{U}}; X_{U}/U, W)|_{X_{U'}}$$
as $\mathbb{R}$-divisor on $X_{U}$. 
\end{itemize}
\end{lem}

\begin{proof}
The first property follows from Definition \ref{defn--asymvanorder-pseudoeff}. 
The second property follows from Theorem \ref{thm--finite-negativepart}. 
The third property follows from Theorem \ref{thm--asymvanorder-pseudoeff-basic} (\ref{thm--asymvanorder-pseudoeff-basic-(6)}) and the fact that $\sigma_{Q}(D|_{X_{U}}; X_{U}/U)=0$ for any prime divisor $Q$ on $X_{U}$ that is not a component of the restriction $P|_{U}$ of any prime divisor $P$ on $X$. 
\end{proof}

\begin{defn}[Nakayama--Zariski decomposition]\label{defn--nakayama-zariski-positive-negative}
Let $\pi \colon X \to Y$ be a projective morphism from a normal analytic variety $X$ to a Stein space $Y$, and
let $W \subset Y$ be a subset. 
Let $D$ be a $\pi$-pseudo-effective $\mathbb{R}$-Cartier divisor on $X$. 
When $N_{\sigma}(D; X/Y, W)$ is well defined as an $\mathbb{R}$-divisor on $X$, we call it the {\em negative part of Nakayama--Zariski decomposition over $W$}.
We set
$$P_{\sigma}(D; X/Y, W):=D-N_{\sigma}(D; X/Y, W),$$
and we call it the {\em positive part of Nakayama--Zariski decomposition over $W$}.
\end{defn}

\begin{thm}\label{thm--nefpositivepart}
Let $\pi \colon X \to Y$ be a projective morphism from a normal analytic variety $X$ to a Stein space $Y$, and
let $W \subset Y$ be a subset. 
Let $D$ be a $\pi$-pseudo-effective $\mathbb{R}$-Cartier divisor on $X$. 
Let $f \colon X' \to X$ be a projective bimeromorphism from a normal variety $X'$. 
If $N_{\sigma}(D; X/Y, W)$ is well defined as an $\mathbb{R}$-divisor on $X$ and $P_{\sigma}(D; X/Y, W)$ is nef over $W$, then $N_{\sigma}(D; X/Y, W)$ is $\mathbb{R}$-Carter and the relation
$$N_{\sigma}(f^{*}D; X'/Y, W)=f^{*}N_{\sigma}(D; X/Y, W)$$
holds. 
In particular, $N_{\sigma}(f^{*}D; X'/Y, W)$ is well defined as an $\mathbb{R}$-Cartier divisor on $X'$, the relation $f^{*}P_{\sigma}(D; X/Y, W)=P_{\sigma}(f^{*}D; X'/Y, W)$ holds, and $P_{\sigma}(f^{*}D; X'/Y, W)$ is nef over $W$.  
\end{thm}

\begin{proof}
Since $D$ and $P_{\sigma}(D; X/Y, W)$ are $\mathbb{R}$-Cartier, $N_{\sigma}(D; X/Y, W)$ is $\mathbb{R}$-Cartier. 
For any prime divisor $Q'$ on $X'$, we have
\begin{equation*}
\begin{split}
\sigma_{Q'}(f^{*}D;X'/Y)=&\sigma_{Q'}(f^{*}P_{\sigma}(D; X/Y, W)+f^{*}N_{\sigma}(D; X/Y, W);X'/Y)\\
\leq &\sigma_{Q'}(f^{*}P_{\sigma}(D; X/Y, W);X'/Y)+\sigma_{Q'}(f^{*}N_{\sigma}(D; X/Y, W);X'/Y)\\
\leq &0+{\rm coeff}_{Q'}(f^{*}N_{\sigma}(D; X/Y, W)),
\end{split}
\end{equation*}
where the final inequality follows from Theorem \ref {thm--asymvanorder-pseudoeff-basic} (\ref{thm--asymvanorder-pseudoeff-basic-(2)}) and Lemma \ref{lem--nagativepart-basic}. 
This shows that $N_{\sigma}(f^{*}D; X'/Y, W)$ is well defined as an $\mathbb{R}$-divisor on $X'$ and we have
$$N_{\sigma}(f^{*}D; X'/Y, W) \leq f^{*}N_{\sigma}(D; X/Y, W).$$
To prove the converse relation, we may replace $f \colon X' \to X$ by $f \circ g \colon X'' \to X' \to X$ for some projective bimeromorphism $g \colon X'' \to X'$ because 
$$g_{*}N_{\sigma}(g^{*}f^{*}D; X''/Y, W)=N_{\sigma}(f^{*}D; X'/Y, W)$$
holds by definition. 
Therefore, by replacing $f$ with a resolution of $X$, we may assume that $X'$ is non-singular. 
We put
$$E'=N_{\sigma}(f^{*}D; X'/Y, W) - f^{*}N_{\sigma}(D; X/Y, W).$$
Then $E'$ is $f$-exceptional.
We also have
\begin{equation*}
\begin{split}
E' \sim_{\mathbb{R},\,X}&N_{\sigma}(f^{*}D; X'/Y, W) - f^{*}D
=-P_{\sigma}(f^{*}D; X'/Y, W),
\end{split}
\end{equation*}
and $P_{\sigma}(f^{*}D; X'/Y, W)$ is a limit of movable divisors over $W$. 
By applying the negativity lemma (Corollary \ref{cor--negativity-veryexc-2}) to $f \colon X' \to X$, $\pi^{-1}(W)$, and $E$, there is $U \supset \pi^{-1}(W)$ an open subset of $X$ such that $E'|_{f^{-1}(U)} \geq 0$. 
By construction, any component $E'_{i}$ of $E'$ satisfies $(\pi \circ f)(E'_{i}) \cap W \neq \emptyset$. 
Therefore $E'\geq 0$, from which we have 
$$N_{\sigma}(f^{*}D; X'/Y, W) \geq f^{*}N_{\sigma}(D; X/Y, W).$$
In this way, we see that Theorem \ref{thm--nefpositivepart} holds. 
\end{proof}

\begin{lem}[cf.~{\cite[Lemma 2.4]{has-finite}}]\label{lem--decomp-supp-inv}
Let $\pi \colon X \to Y$ be a projective morphism from a normal analytic variety $X$ to a Stein space $Y$, and let $W \subset Y$ be a subset. 
Let $D$ and $D'$ be $\pi$-pseudo-effective $\mathbb{R}$-Cartier divisors on $X$ such that $N_{\sigma}(D; X/Y, W)$ and $N_{\sigma}(D'; X/Y, W)$ are well defined as $\mathbb{R}$-divisors on $X$. 
Then there exists a positive  real number $t_{0}$ such that $N_{\sigma}(D+tD'; X/Y, W)$ is well defined as $\mathbb{R}$-divisors on $X$ and the support of $N_{\sigma}(D+tD'; X/Y, W)$ is independent of $t \in (0,t_{0}]$. 
\end{lem}

\begin{proof}
The argument in \cite[Proof of Lemma 2.4]{has-finite} works with no changes. 
\end{proof}

\begin{lem}[cf.~{\cite[Lemma 2.6]{has-finite}}]\label{lem--dlt-lccenter-discrepancy}
Let $Y$ be a Stein space and $W \subset Y$ a compact subset. 
Let $(X,\Delta)$ and $(X',\Delta')$ be dlt pairs with projective morphisms $\pi \colon X \to Y$ and $\pi' \colon X' \to Y$, and let $S$ and $S'$ be lc centers of $(X,\Delta)$ and $(X',\Delta')$ respectively. 
Let $f \colon X \dashrightarrow X'$ be a bimeromorphic map over $Y$ such that $f$ is a biholomorphism on a Zariski open subset intersecting $S$ and $f|_{S}$ induces a bimeromorphic map $f_{S} \colon S \dashrightarrow S'$ over $Y$. 
Suppose that $K_{X}+\Delta$ is $\pi$-pseudo-effective.  
Suppose in addition that
\begin{itemize}
\item
the inequality $a(D',X',\Delta') \leq a(D',X,\Delta)$ holds for all prime divisors $D'$ on $X'$ such that $\pi'(D')\cap W \neq \emptyset$, and 
\item
$\sigma_{P}(K_{X}+\Delta;X/Y)=0$ for all prime divisor $P$ over $X$ such that $a(P,X,\Delta)<0$, the center of $P$ on $X$ intersects $S$, and the image of $P$ on $Y$ intersects $W$. 
\end{itemize}
Let $(S,\Delta_{S})$ and $(S',\Delta_{S'})$ be the dlt pairs defined by adjunctions $K_{S}+\Delta_{S}=(K_{X}+\Delta)|_{S}$ and $K_{S'}+\Delta_{S'}=(K_{X'}+\Delta')|_{S'}$, respectively. 
Then the inequality
$$a(Q',S',\Delta_{S'}) \leq a(Q',S,\Delta_{S})$$
holds for any prime divisor $Q'$ on $S'$ such that $\pi'(Q')\cap W \neq \emptyset$. 
\end{lem}

\begin{proof}
We may apply \cite[Proof of Lemma 2.6]{has-finite} to our situation. 
We only outline the proof. 

Since the problem is local (see also Theorem \ref{thm--asymvanorder-pseudoeff-basic} (\ref{thm--asymvanorder-pseudoeff-basic-(6)})), we may freely shrink $Y$ around $W$ without loss of generality. 
By shrinking $Y$ around $W$, we may assume that there is a common log resolution $g \colon \overline{X} \to X$ and $g' \colon \overline{X} \to X'$ of $f \colon X \dashrightarrow X'$ and a subvariety $T \subset \overline{X}$ such that $g$ and $g'$ induce bimeromorphisms $g_{T}\colon T \to S$ and $g'_{T} \colon T \to S'$, respectively. 
We may write
$$g^{*}(K_{X}+\Delta)=g'^{*}(K_{X'}+\Delta')+M-N$$
with $M \geq 0$ and $N \geq 0$ having no common components. 
By the first condition of Lemma \ref{lem--dlt-lccenter-discrepancy} and the argument in \cite[Proof of Lemma 2.6]{has-finite}, $M$ is $g'$-exceptional over the inverse image of an open neighborhood of $W$, and furthermore we have ${\rm Supp}\,M \not\supset T$ and ${\rm Supp}\,N \not\supset T$ over the inverse image of an open neighborhood of $W$. 
From these facts, we see that $K_{X'}+\Delta'$ is $\pi$-pseudo-effective. 
By shrinking $Y$ around $W$ we may write
$$g^{*}_{T}(K_{S}+\Delta_{S})=g'^{*}_{T}(K_{S'}+\Delta_{S'})+M|_{T}-N|_{T},$$
where $M_{T} \geq0$ and $N|_{T} \geq 0$ have no common components. 
By the same argument as in \cite[Proof of Lemma 2.6]{has-finite} using Theorem \ref{thm--asymvanorder-pseudoeff-basic}, we have
$$\sigma_{P}(K_{X}+\Delta;X/Y)>0$$
for any component $P$ of $M$ that intersects $T$. 
By the second condition of Lemma \ref{lem--dlt-lccenter-discrepancy}, we have 
$$a(P,X',\Delta')>a(P,X,\Delta) \geq 0$$
for any component $P$ of $M$ such that $P \cap T \neq \emptyset$ and $(\pi \circ g)(P) \cap W \neq \emptyset$. 
By shrinking $Y$ around $W$ and the standard argument as in the proof of \cite[Lemma 2.5]{has-finite}, it follows that $M|_{T}$ is exceptional over $S'$. 
Then, for any prime divisor $Q'$ on $S'$, we have ${\rm coeff}_{g'^{-1}_{T}Q'}(M|_{T})=0$. 
In particular, 
$$a(Q',S',\Delta_{S'}) \leq a(Q',S,\Delta_{S})$$
holds for any prime divisor $Q'$ on $S'$ such that $\pi'(Q')\cap W \neq \emptyset$. 
\end{proof}

\subsection{Criterion for existence of log minimal model} 

In this subsection, we study the existence of log minimal models for lc pairs using the asymptotic vanishing order and the negative part of Nakayama--Zariski decomposition.

\begin{thm}[cf.~{\cite[Theorem 1.1]{bhzariski}}, {\cite[Theorem 4.18]{tsakanikas-phdthesis}}]\label{thm--birkarhu}
Let $\pi \colon X \to Y$ be a contraction from a normal analytic variety $X$ to a Stein space $Y$, and let $W \subset Y$ be a connected compact subset such that $\pi$ and $W$ satisfy (P). 
Let $(X,\Delta)$ be an lc pair. 
Then the following conditions are equivalent. 
\begin{itemize}
\item
After shrinking $Y$ around $W$, the lc pair $(X,\Delta)$ has a log minimal model over $Y$ around $W$.
\item
After shrinking $Y$ around $W$, there exists a resolution $f \colon \widetilde{X} \to X$ of $X$ such that $N_{\sigma}(f^{*}(K_{X}+\Delta); \widetilde{X}/Y, W)$ is well defined as an $\mathbb{R}$-divisor on $\widetilde{X}$ and the divisor $P_{\sigma}(f^{*}(K_{X}+\Delta); \widetilde{X}/Y, W)$ is a finite $\mathbb{R}_{>0}$-linear combination of $\mathbb{Q}$-Cartier divisors on $\widetilde{X}$ that are nef over $W$. 
\end{itemize}
\end{thm}

\begin{proof}
Assume the first condition of Theorem \ref{thm--birkarhu}. 
After shrinking $Y$ around $W$, there exists a log minimal model $(X',\Delta')$ of $(X,\Delta)$ over $Y$ around $W$. 
By the argument using Shokurov's polytope \cite[Theorem 14.3 (3)]{fujino-analytic-conethm}, we can write
$$K_{X'}+\Delta'=\sum_{i=1}^{l}r_{i}D'_{i},$$ 
where $r_{i}$ are positive real number and $D'_{i}$ are $\mathbb{Q}$-Cartier divisors on $X'$ that are nef over $W$. 
Let $g \colon X'' \to X$ be a resolution of $X$ which resolves the indeterminacy of $X \dashrightarrow X'$, and let $g' \colon X'' \to X'$ be the induced bimeromorphism. 
After shrinking $Y$ around $W$, we can write
$$g^{*}(K_{X}+\Delta)=g'^{*}(K_{X'}+\Delta')+E''$$
such that $E''$ is an effective $g'$-exceptional $\mathbb{R}$-divisor on $X''$. 
By Theorem \ref{thm--asymvanorder-pseudoeff-basic} (\ref{thm--asymvanorder-pseudoeff-basic-(2)}) (\ref{thm--asymvanorder-pseudoeff-basic-(6)}), $N_{\sigma}(g^{*}(K_{X}+\Delta); X''/Y, W)$ is well defined and we can find an open subset $U \subset Y$ containing $W$ such that
\begin{equation*}
N_{\sigma}(g^{*}(K_{X}+\Delta); X''/Y, W)|_{(\pi \circ g)^{-1}(U)}=E''|_{(\pi \circ g)^{-1}(U)}.
\end{equation*}
Then 
$$P_{\sigma}(g^{*}(K_{X}+\Delta); X''/Y, W)|_{(\pi \circ g)^{-1}(U)}=g'^{*}(K_{X'}+\Delta')|_{(\pi \circ g)^{-1}(U)}=\sum_{i=1}^{l}r_{i}g'^{*}D'_{i}|_{(\pi \circ g)^{-1}(U)}.$$
Thus $N_{\sigma}(g^{*}(K_{X}+\Delta); X''/Y, W)$ is well defined as an $\mathbb{R}$-divisor on $X''$ and the divisor $P_{\sigma}(g^{*}(K_{X}+\Delta); X''/Y, W)$ is a finite $\mathbb{R}_{>0}$-linear combination of $\mathbb{Q}$-Cartier divisors on $X''$ that are nef over $W$. 
From this, we see that the first condition of Theorem \ref{thm--birkarhu} implies the second condition of Theorem \ref{thm--birkarhu}. 

Conversely, suppose that the second condition of Theorem \ref{thm--birkarhu} holds. 
After shrinking $Y$ around $W$, we get a resolution $f \colon \widetilde{X} \to X$ of $X$ such that $N_{\sigma}(f^{*}(K_{X}+\Delta); \widetilde{X}/Y, W)$ is well defined as an $\mathbb{R}$-divisor on $\widetilde{X}$ and we can write
$$P_{\sigma}(f^{*}(K_{X}+\Delta); \widetilde{X}/Y, W)= \sum_{i=1}^{k}\gamma_{i}M_{i},$$
where $\gamma_{i}$ are positive real numbers and $M_{i}$ are $\mathbb{Q}$-Cartier divisors on $\widetilde{X}$ that are nef over $W$. 
By shrinking $Y$ around $W$, we get a bimeromorphism $f' \colon \widetilde{X}' \to \widetilde{X}$ such that $f \circ f' \colon \widetilde{X}' \to X$ is a log resolution of $(X,\Delta)$. 
By Theorem \ref{thm--nefpositivepart} and replacing $f$ with $f \circ f'$, we may assume that $f \colon X' \to X$ is a log resolution of $(X,\Delta)$. 
By shrinking $Y$ around $W$ again, we may assume that there are only finitely many $f$-exceptional prime divisors and $M_{i}$ are globally $\mathbb{Q}$-Cartier. 
By replacing $\gamma_{i}$ with some multiple, we may assume that $M_{i}$ are Cartier. 

For the simplicity of notations, from now on we drop $\widetilde{X}/Y, W$ and $\widetilde{X}'/Y, W$ in the negative part and the positive part of the Nakayama--Zariski decomposition. 

Let $\widetilde{\Delta}$ be the sum of $f^{-1}_{*}\Delta$ and the reduced $f$-exceptional divisor. 
Then we can write 
$$K_{\widetilde{X}}+\widetilde{\Delta}=f^{*}(K_{X}+\Delta)+E,$$
where $E$ is effective and $f$-exceptional. 
By Theorem \ref{thm--asymvanorder-pseudoeff-basic} (\ref{thm--asymvanorder-pseudoeff-basic-(5)}), we have
$$N_{\sigma}(K_{\widetilde{X}}+\widetilde{\Delta})=N_{\sigma}(f^{*}(K_{X}+\Delta))+E.$$
Therefore, $N_{\sigma}(K_{\widetilde{X}}+\widetilde{\Delta})$ is well defined as an $\mathbb{R}$-divisor on $\widetilde{X}$, and this shows
$$P_{\sigma}(K_{\widetilde{X}}+\widetilde{\Delta})=P_{\sigma}(f^{*}(K_{X}+\Delta))=\sum_{i=1}^{k}\gamma_{i} M_{i}.$$
Thus $P_{\sigma}(K_{\widetilde{X}}+\widetilde{\Delta})$  is a finite $\mathbb{R}_{>0}$-linear combination of $\mathbb{Q}$-Cartier divisors on $\widetilde{X}$ that are nef over $W$. 
Moreover, by Lemma \ref{lem--exist-model-birat-I}, we may replace $(X,\Delta)$ with $(\widetilde{X},\widetilde{\Delta})$. 
From these facts, by replacing $(X,\Delta)$ with $(\widetilde{X},\widetilde{\Delta})$, we may assume that 
\begin{itemize}
\item
$(X,\Delta)$ is log smooth, 
\item
$N_{\sigma}(K_{X}+\Delta)$ is well defined as an $\mathbb{R}$-divisor on $X$, and
\item
$P_{\sigma}(K_{X}+\Delta)=\sum_{i=1}^{k}\gamma_{i}M_{i}$ for some Cartier divisors $M_{i}$ on $X$ that are nef over $W$. 
\end{itemize}

We follow the argument in \cite[Section 4.2]{tsakanikas-phdthesis}.  
We put $M:=P_{\sigma}(K_{X}+\Delta)$. 
Let $\alpha$ be a positive real number such that $\alpha \gamma_{i}>2 \cdot {\rm dim}\,X$ for all $i$. 
We run a $(K_{X}+\Delta+\alpha M)$-MMP over $Y$ around $W$ with scaling of an ample divisor. 
Note that we may run the MMP since $M$ is nef over $W$. 
By the same argument as in \cite[Proof of Lemma 13.7]{fujino-analytic-bchm}, after shrinking $Y$ around $W$ we get a bimeromorphic contraction
$$(X,\Delta+\alpha M)\overset{\phi}{\dashrightarrow} (\overline{X},\overline{\Delta}+\alpha \overline{M})$$
over $Y$ such that $K_{\overline{X}}+\overline{\Delta}+\alpha \overline{M}$ is the limit of movable divisors over a neighborhood of $W$. 
Then
$$N_{\sigma}(K_{\overline{X}}+\overline{\Delta}+\alpha \overline{M})=0.$$
By Lemma \ref{lem--nakayama-III4.2}, we have
$$N_{\sigma}(K_{X}+\Delta+\alpha M)=N_{\sigma}((\alpha+1)(K_{X}+\Delta)-\alpha N_{\sigma}(K_{X}+\Delta))=N_{\sigma}(K_{X}+\Delta).$$
This implies 
$$P_{\sigma}(K_{X}+\Delta+\alpha M)=K_{X}+\Delta+\alpha M-N_{\sigma}(K_{X}+\Delta+\alpha M)=(\alpha+1)M.$$ 
Therefore
\begin{equation*}
\begin{split}
K_{X}+\Delta+\alpha M=(\alpha+1)M+N_{\sigma}(K_{X}+\Delta+\alpha M).
\end{split}
\end{equation*}
By the definition of $\alpha$, the strict transform of $(\alpha+1)M$ is trivial with respect to the extremal contraction in each step of the $(K_{X}+\Delta+\alpha M)$-MMP. 
This shows that $(\alpha+1)\phi_{*}M$ is nef over $W$ and $\phi \colon X \dashrightarrow \overline{X}$ is a sequence of steps of a $(K_{X}+\Delta)$-MMP over $Y$ around $W$. 
By the same argument as in \cite[Proof of Lemma 4.1]{bhzariski}, we have
 $$P_{\sigma}(K_{\overline{X}}+\overline{\Delta}+\alpha \overline{M})=\phi_{*}P_{\sigma}(K_{X}+\Delta+\alpha M)=(\alpha+1)\phi_{*}M.$$
From these relations, we have
$$K_{\overline{X}}+\overline{\Delta}+\alpha \overline{M}=P_{\sigma}(K_{\overline{X}}+\overline{\Delta}+\alpha \overline{M})+N_{\sigma}(K_{\overline{X}}+\overline{\Delta}+\alpha \overline{M})=(\alpha+1)\phi_{*}M,$$
and therefore $K_{\overline{X}}+\overline{\Delta}=\phi_{*}M$ is nef over $W$. 
Then 
$$(X,\Delta)\overset{\phi}{\dashrightarrow} (\overline{X},\overline{\Delta})$$
is a sequence of steps of a $(K_{X}+\Delta)$-MMP over $Y$ around $W$ to a log minimal model. 
In particular, $(X,\Delta)$ has a log minimal model over $Y$ around $W$. 

By the above discussion, it follows that the second condition of Theorem \ref{thm--birkarhu} implies the first condition of Theorem \ref{thm--birkarhu}. 
Therefore, Theorem \ref{thm--birkarhu} holds true.  
\end{proof}

By the same argument, we can prove the following.

\begin{thm}[cf.~{\cite[Theorem 2.23]{has-finite}}]\label{thm--birkarhu-semiample}
Let $\pi \colon X \to Y$ be a contraction from a normal analytic variety $X$ to a Stein space $Y$, and let $W \subset Y$ be a connected compact subset such that $\pi$ and $W$ satisfy (P). 
Let $(X,\Delta)$ be an lc pair. 
Then the following conditions are equivalent. 
\begin{itemize}
\item
After shrinking $Y$ around $W$, the lc pair $(X,\Delta)$ has a good minimal model over $Y$ around $W$.
\item
After shrinking $Y$ around $W$, there exists a resolution $f \colon \widetilde{X} \to X$ of $X$ such that $N_{\sigma}(f^{*}(K_{X}+\Delta); \widetilde{X}/Y, W)$ is well defined as an $\mathbb{R}$-divisor on $\widetilde{X}$ and the divisor $P_{\sigma}(f^{*}(K_{X}+\Delta); \widetilde{X}/Y, W)$ is semi-ample over a neighborhood of $W$. 
\end{itemize}
\end{thm}

\begin{thm}[cf.~{\cite[Lemma 2.25]{has-finite}}]\label{thm--minmodel-asymvanorder-boundary}
Let $\pi \colon X \to Y$ be a projective morphism from a normal analytic variety $X$ to a Stein space $Y$, and let $W \subset Y$ be a compact subset such that $\pi$ and $W$ satisfy (P). 
Let $(X,\Delta)$ be an lc pair and let $(X',\Delta')$ be an lc pair together with a projective bimeromorphic morphism $f \colon X' \to X$. 
Suppose that $K_{X}+\Delta$ is $\pi$-pseudo-effective and the relation 
$$0 \leq a(D,X',\Delta')-a(D,X,\Delta) \leq \sigma_{D}(K_{X}+\Delta;X/Y)$$
holds for all prime divisors $D$ on $X'$ such that $(\pi \circ f)(D) \cap W \neq \emptyset$. 
Then the following conditions are equivalent.
\begin{itemize}
\item
After shrinking $Y$ around $W$, the lc pair $(X,\Delta)$ has a log minimal model over $Y$ around $W$.  
\item
After shrinking $Y$ around $W$, the lc pair $(X',\Delta')$ has a log minimal model over $Y$ around $W$.  
\end{itemize}
\end{thm}

\begin{proof}

As in the discussion in Remark \ref{rem--mmp-reduction-basic}, we may assume that $\pi$ is a contraction and $W$ is connected. 
By Theorem \ref{thm--asymvanorder-pseudoeff-basic} (\ref{thm--asymvanorder-pseudoeff-basic-(6)}), we may freely shrink $Y$ around $W$ if necessary without loosing the irreducibility of both $X$ and $X$'. 
By shrinking $Y$ around $W$, we can take a log resolution $g \colon \widetilde{X} \to X'$ of $(X',\Delta')$ such that $f \circ g \colon \widetilde{X} \to X$ is a log resolution of $(X,\Delta)$. 
Let $\widetilde{\Delta}'$ be the sum of $g^{-1}_{*}\Delta'$ and the reduced $g$-exceptional divisor on $\widetilde{X}$, and let $\widetilde{\Delta}$ be the sum of $(f \circ g)^{-1}_{*}\Delta$ and the reduced $(f \circ g)$-exceptional divisor on $\widetilde{X}$. 
By the same argument as in \cite[Lemma 3.6]{has-iitakafibration}, by simple computation of the asymptotic vanishing orders using Theorem \ref{thm--asymvanorder-pseudoeff-basic} (\ref{thm--asymvanorder-pseudoeff-basic-(5)}), we have
$$0 \leq a(\widetilde{D},\widetilde{X},\widetilde{\Delta}')-a(\widetilde{D},\widetilde{X},\widetilde{\Delta}) \leq \sigma_{D}(K_{\widetilde{X}}+\widetilde{\Delta};\widetilde{X}/Y)$$
for every prime divisor $\widetilde{D}$ on $\widetilde{X}$ whose image on $Y$ intersects $W$. 
By Lemma \ref{lem--exist-model-birat-I}, the lc pair $(X,\Delta)$ (resp.~$(X',\Delta')$) has a log minimal model over $Y$ around $W$ if and only if the lc pair $(\widetilde{X},\widetilde{\Delta})$ (resp.~$(\widetilde{X},\widetilde{\Delta}')$) has a log minimal model over $Y$ around $W$. 
Replacing $(X,\Delta)$ and $(X',\Delta')$ by $(\widetilde{X},\widetilde{\Delta})$ and $(\widetilde{X},\widetilde{\Delta}')$ respectively, we may assume $X=X'$ and that $(X,\Delta)$ and $(X',\Delta')$ are log smooth. 
Then \cite[Proof of Lemma 2.25]{has-finite} works with minor changes because we can use Theorem \ref{thm--birkarhu} instead of \cite[Theorem 2.23]{has-finite}. 
\end{proof}

\begin{thm}[{cf.~\cite[Lemma 2.15]{has-mmp}}]\label{thm--exist-model-birat-II}
Let $\pi \colon X \to Y$ be a projective morphism from a normal analytic variety $X$ to a Stein space $Y$, and let $W \subset Y$ be a compact subset such that $\pi$ and $W$ satisfy (P). 
Let $(X,\Delta)$ be an lc pair and let $(X',\Delta')$ be an lc pair together with a projective bimeromorphic morphism $f \colon X' \to X$. 
Suppose that we can write
$$K_{X'}+\Delta'=f^{*}(K_{X}+\Delta)+E$$
for some effective $f$-exceptional $\mathbb{R}$-divisor $E$ on $X'$. 
Then the following conditions are equivalent.
\begin{itemize}
\item
After shrinking $Y$ around $W$, the lc pair $(X,\Delta)$ has a log minimal model over $Y$ around $W$.  
\item
After shrinking $Y$ around $W$, the lc pair $(X',\Delta')$ has a log minimal model over $Y$ around $W$.  
\end{itemize}
\end{thm}

\begin{proof}
The argument in the algebraic case (see \cite[Proof of Lemma 3.7]{has-iitakafibration}) works with no changes since we may apply Theorem \ref{thm--minmodel-asymvanorder-boundary}. 
\end{proof}

\begin{lem}[cf.~{\cite[Remark 2.9]{has-finite}}]\label{lem--asymvanorder-discrepancy}
Let $\pi \colon X \to Y$ be a projective morphism from a normal analytic variety $X$ to a Stein space $Y$, and let $W \subset Y$ be a subset. 
Let $(X,\Delta)$ be an lc pair. 
Let $(X',\Delta')$ be a weak lc model of $(X,\Delta)$ over $Y$ around $W$. 
We denote the bimeromorphic map $X \dashrightarrow X'$ by $\phi$. 
Then the following properties hold. 
\begin{itemize}
\item
For any prime divisor $P$ over $X$ whose image on $Y$ intersects $W$, we have
$$\sigma_{P}(K_{X}+\Delta;X/Y)=a(P,X',\Delta')-a(P,X,\Delta).$$
\item
Suppose that $(X',\Delta')$ be a log minimal model of $(X,\Delta)$ over $Y$ around $W$ and $\phi$ is a sequence of steps of a $(K_{X}+\Delta)$-MMP over $Y$ around $W$. 
Let $D$ be a prime divisor on $X$ such that $\pi(D) \cap W \neq \emptyset$. 
Then $D$ is contracted by $\phi$ if and only if $\sigma_{D}(K_{X}+\Delta; X/Y)>0$.  
\end{itemize}
\end{lem}

\begin{proof}
Let $f \colon X'' \to X$ be a bimeromorphism which resolves the indeterminacy of $\phi \colon X \dashrightarrow X'$, and let $f' \colon X'' \to X'$ be the induced bimeromorphism. 
By shrinking $Y$ around $W$, we may assume that the divisor
$$E:=f^{*}(K_{X}+\Delta)-f'^{*}(K_{X'}+\Delta')$$
is $f'$-exceptional. 
Then 
$${\rm coeff}_{P}(E)=a(P,X',\Delta')-a(P,X,\Delta).$$
On the other hand, by using Theorem \ref{thm--asymvanorder-pseudoeff-basic}, we obtain
\begin{equation*}
\begin{split}
\sigma_{P}(K_{X}+\Delta;X/Y)=&\sigma_{P}(f^{*}(K_{X}+\Delta);X''/Y)\\
=&\sigma_{P}(f'^{*}(K_{X'}+\Delta')+E;X''/Y)\\
=&\sigma_{P}(f'^{*}(K_{X'}+\Delta');X''/Y)+{\rm coeff}_{P}(E)\\
=&{\rm coeff}_{P}(E). 
\end{split}
\end{equation*}
Therefore, the first statement holds true. 

Next, suppose that $(X',\Delta')$ be a log minimal model of $(X,\Delta)$ over $Y$ around $W$ and $\phi$ is a sequence of steps of a $(K_{X}+\Delta)$-MMP over $Y$ around $W$. 
By the first statement, $\sigma_{D}(K_{X}+\Delta; X/Y)>0$ if and only if $a(D,X,\Delta) \neq a(D,X',\Delta')$, and the latter condition is equivalent to that $D$ is contracted by $X \dashrightarrow X'$ the sequence of steps of the $(K_{X}+\Delta)$-MMP. 
Therefore, the second statement holds. 
\end{proof}

\begin{lem}[cf.~{\cite[Lemma 2.26]{has-finite}}]\label{lem--bir-relation}
Let $\pi \colon X \to Y$ be a projective morphism from a normal analytic variety $X$ to a Stein space $Y$, and let $W \subset Y$ be a compact subset such that $\pi$ and $W$ satisfy (P). 
Let $\pi' \colon X' \to Y$ be a projective morphism from a normal analytic variety $X'$. 
Let $(X,\Delta)$ and $(X',\Delta')$ be lc pairs and let $X \dashrightarrow X'$ be a bimeromorphic map over $Y$. 
Suppose that
\begin{itemize}
\item
$a(P,X,\Delta) \leq a(P,X',\Delta')$ for all prime divisors $P$ on $X$ whose image on $Y$ intersects $W$, and 
\item
$a(P',X,\Delta) \geq a(P',X',\Delta')$ for all prime divisors $P'$ on $X'$ whose image on $Y$ intersects $W$. 
\end{itemize}
Then the following conditions are equivalent. 
\begin{itemize}
\item
After shrinking $Y$ around $W$, the lc pair $(X,\Delta)$ has a log minimal model over $Y$ around $W$.  
\item
After shrinking $Y$ around $W$, the lc pair $(X',\Delta')$ has a log minimal model over $Y$ around $W$.  
\end{itemize}
\end{lem}

\begin{proof}
By shrinking $Y$ around $W$, we may assume that $a(P,X,\Delta) \leq a(P,X',\Delta')$ for all prime divisors $P$ on $X$ and
$a(P',X,\Delta) \geq a(P',X',\Delta')$ for all prime divisors $P'$ on $X'$. 
The argument of \cite[Proof of Lemma 2.26]{has-finite} works with no changes because we can use Theorem \ref{thm--exist-model-birat-II}. 
\end{proof}


\section{Log MMP for log abundant lc pairs}\label{sec5}

In this section we study log MMP with scaling that contains log abundant lc pairs.

\subsection{Auxiliary results}

In this subsection, we collect some results used to prove the main results of this paper.

\begin{lem}\label{lem--mmpscaling-effective}
Let $\pi\colon X \to Y$ be a contraction from a normal complex variety $X$ to a Stein space $Y$, and let $W \subset Y$ be a connected compact subset such that $\pi$ and $W$ satisfy (P). 
Let $S$ be a subvariety of $X$ and let $(X,\Delta)$ be an lc pair. 
Let
$$(X_{0}:=X,\Delta_{0}:=\Delta) \dashrightarrow (X_{1},\Delta_{1}) \dashrightarrow\cdots \dashrightarrow (X_{i},\Delta_{i})\dashrightarrow \cdots$$ 
be a sequence of steps of a $(K_{X}+\Delta)$-MMP over $Y$ around $W$ with scaling of an effective $\mathbb{R}$-divisor $A$. 
Let $A_{i}$ be the birational transform of $A$ on $X_{i}$. 
We define 
$$\lambda_{i}={\rm inf}\set{\mu \in \mathbb{R}_{\geq0} \!|\! K_{X_{i}}+\Delta_{i}+\mu A_{i}\text{\rm \, is nef over }W}.$$ 
Suppose that each step of the $(K_{X}+\Delta)$-MMP is an biholomorphism on a neighborhood of $S$ and ${\rm lim}_{i\to\infty}\lambda_{i}=0$.  
Then, for any prime divisor $P$ over $X$ whose images on $X$ and $Y$ intersect $S$ and $W$ respectively, we have $\sigma_{P}(K_{X}+\Delta;X/Y)=0$. 
\end{lem}

\begin{proof}
Since ${\rm lim}_{i\to\infty}\lambda_{i}=0$, it is sufficient to prove $\sigma_{P}(K_{X}+\Delta+\lambda_{i}A;X/Y)=0$ for all $i \geq 0$. 
Let $c_{X}(P)$ be the image of $P$ on $X$.
By Theorem \ref{thm--asymvanorder-pseudoeff-basic} (\ref{thm--asymvanorder-pseudoeff-basic-(6)}), we may freely shrink $Y$. 
By shrinking $Y$ around $W$, we may assume that $i$ steps of the $(K_{X}+\Delta)$-MMP over $Y$ around $W$
$$(X_{0}:=X,\Delta_{0}:=\Delta) \dashrightarrow (X_{1},\Delta_{1}) \dashrightarrow\cdots \dashrightarrow (X_{i},\Delta_{i})$$ 
is represented by a bimeromorphic contraction $X \dashrightarrow X_{i}$ over $Y$. 
Let $f \colon X' \to X$ and $f_{i} \colon X' \to X_{i}$ be a resolution of the indeterminacy of $X \dashrightarrow X_{i}$ such that $P$ appears as a prime divisor on $X'$. 
By shrinking $Y$ around $W$, we may write
$$f^{*}(K_{X}+\Delta+\lambda_{i}A)=f_{i}^{*}(K_{X_{i}}+\Delta_{i}+\lambda_{i}A_{i})+E_{i},$$
where $E_{i}$ is an effective $f_{i}$-exceptional $\mathbb{R}$-divisor on $X'$. 
By Theorem \ref{thm--asymvanorder-pseudoeff-basic} and the facts that the $(K_{X}+\Delta)$-MMP is an biholomorphism on a neighborhood of $S$ and that the images of $P$ on $X$ and $Y$ intersect $S$ and $W$ respectively, we have
\begin{equation*}
\begin{split}
\sigma_{P}(K_{X}+\Delta+\lambda_{i}A;X/Y)=&\sigma_{P}(f^{*}(K_{X}+\Delta+\lambda_{i}A);X'/Y)\\
=&\sigma_{P}(f_{i}^{*}(K_{X_{i}}+\Delta_{i}+\lambda_{i}A_{i})+E_{i};X'/Y)\\
=&\sigma_{P}(f_{i}^{*}(K_{X_{i}}+\Delta_{i}+\lambda_{i}A_{i});X'/Y)+{\rm coeff}_{P}(E_{i})=0.
\end{split}
\end{equation*}
Since ${\rm lim}_{i\to\infty}\lambda_{i}=0$, and thus $\sigma_{P}(K_{X}+\Delta;X/Y)=0$. 
\end{proof}

\begin{lem}[cf.~{\cite[Lemma 2.10]{has-trivial}}]\label{thm--speciallogresol}
Let $\pi \colon X \to Y$ be a projective morphism from a normal analytic variety to an analytic space $Y$, and let $W \subset Y$ be a compact subset. 
Let $(X,\Delta)$ be an lc pair and let $D$ be a reduced divisor on $X$. 
Then, after shrinking $Y$ around $W$, there exists a log smooth model $f \colon (\overline{X},\overline{\Delta}) \to (X,\Delta)$ of $(X,\Delta)$ satisfying the following.
\begin{enumerate}[(i)]
\item\label{thm--speciallogresol-(i)}
$\overline{\Delta}=\overline{\Delta}'+\overline{\Delta}''$, where $\overline{\Delta}' \geq 0$ and $\overline{\Delta}''$ is zero or a reduced divisor, 
\item\label{thm--speciallogresol-(ii)}
$(\pi \circ f)({\rm Supp}\,\overline{\Delta}'') \subsetneq Y$, 
\item\label{thm--speciallogresol-(iii)}
any lc center of $(\overline{X},\overline{\Delta}-t \overline{\Delta}'')$ dominates $Y$ for all $t \in (0,1]$, and
\item\label{thm--speciallogresol-(iv)}
the exceptional locus ${\rm Ex}(f)$ of $f$ is pure codimension one and ${\rm Ex}(f) \cup f_{*}^{-1}(\Delta+D)$  has simple normal crossing support. 
\end{enumerate}
\end{lem}

\begin{proof}
By Remark \ref{rem--mmp-reduction-basic}, we may assume that $W$ is connected. 
By shrinking $Y$ around $W$, we may assume that $\Delta+D$ has only finitely many components and $(X,\Delta)$ has only finitely many lc centers. 
By replacing $(X,\Delta)$ with a log smooth model and replacing $D$ accordingly, we may assume that $(X,\Delta+D)$ is log smooth. 
For each lc center $S$ of $(X,\Delta)$ such that $\pi(S) \subsetneq Y$, let $X_{S} \to X$ be the blow-up of $X$ along $S$. 
Let $f \colon \overline{X} \to X$ be a log resolution of $(X,\Delta+D)$ that factors through $X_{S} \to X$ for any lc center $S$ of $(X,\Delta)$ such that $\pi(S) \subsetneq Y$. 
We construct a log smooth model $f \colon (\overline{X},\overline{\Delta}) \to (X,\Delta)$ of $(X,\Delta)$, we define $\overline{\Delta}''$ to be the sum of all components of $\lfloor \overline{\Delta} \rfloor$ that do not dominate $Y$, and we set $\overline{\Delta}':=\overline{\Delta}-\overline{\Delta}''$. 
By the same argument as in \cite[Proof of Lemma 2.10]{has-trivial}, it is easy to see that $f \colon (\overline{X},\overline{\Delta}) \to (X,\Delta)$ is the desired log smooth model. 
\end{proof}

\begin{lem}[cf.~{\cite[Lemma 2.16]{hashizumehu}}, {\cite[Corollary 1.37 and Corollary 1.38]{kollar-mmp}}]\label{lem--extraction}
Let $\pi \colon X \to Y$ be a projective morphism from a normal analytic variety $X$ to a Stein space $Y$, and let $W \subset Y$ be a compact subset such that $\pi$ and $W$ satisfy (P). 
Let $(X,\Delta)$ be a dlt pair. 
Let $\mathcal{T}$ be a (possibly empty) finite set of prime divisors $D$ over $X$ such that the image of $D$ on $Y$ intersects $W$ and $-1< a(D,X,\Delta)<0$ for any $D \in \mathcal{T}$. 
For any open subset $U \subset Y$ containing $W$, we put $X_{U}=\pi^{-1}(U)$ and $\Delta_{U}=\Delta|_{X_{U}}$, and the restriction of any divisor $D \in \mathcal{T}$ to the inverse image of $U$ is denoted by $D_{U}$. 
Then there exist a Stein open subset $U \subset Y$ containing $W$, a projective bimeromorphism $f \colon \widetilde{X} \to X_{U}$ from a normal analytic variety $\widetilde{X}$, and $(\widetilde{X}, \widetilde{\Delta})$ a disjoint union of dlt pairs  satisfying the following.
\begin{itemize}
\item
$\widetilde{X}$ is $\mathbb{Q}$-factorial over $W$, 
\item
$K_{\widetilde{X}}+\widetilde{\Delta}=f^{*}(K_{X_{U}}+\Delta_{U})$, and
\item
the set of $f$-exceptional prime divisors on $\widetilde{X}$ is exactly the same as the set of the components of $D_{U}$ for the divisors $D \in \mathcal{T}$.
\end{itemize}
\end{lem}

\begin{proof}
The argument in the algebraic case (see \cite[Corollary 1.37 and Corollary 1.38]{kollar-mmp}) works with no changes. 
\end{proof}

\begin{lem}[cf.~{\cite[Proposition 3.3]{has-mmp}}]\label{lem--mmp-genericCYfib}
Let $\pi \colon X \to Y$ be a projective surjective morphism from a normal analytic variety $X$ to a Stein space $Y$, and let $W \subset Y$ be a compact subset such that $\pi$ and $W$ satisfy (P). 
Let $(X,\Delta)$ be an lc pair. 
Suppose that there exists a contraction $\phi \colon X \to Z$ over $Y$, where $Z$ is a normal analytic variety that is projective over $Y$, satisfying the following.
\begin{itemize}
\item
$\kappa_{\sigma}(X/Z, K_{X}+\Delta)=0$, 
\item
$\kappa_{\sigma}(X/Y, K_{X}+\Delta)={\rm dim}\,Z-{\rm dim}\,Y$, and 
\item
any lc center of $(X,\Delta)$ dominates $Z$. 
\end{itemize}
Then, after shrinking $Y$ around $W$, the lc pair $(X,\Delta)$ has a good minimal model over $Y$ around $W$. 
\end{lem}

\begin{proof}
The argument in \cite[Proof of Proposition 3.3]{has-mmp} works with no changes because we can use the canonical bundle formula in the analytic setting \cite[Theorem 21.4]{fujino-analytic-bchm} and the weak semistable reduction in the analytic setting \cite{eh-semistablereduction}. 
So we only outline the proof. 

By Remark \ref{rem--mmp-reduction-basic}, we may assume that $\pi \colon X \to Y$ and $Z \to Y$ are contractions and $W$ is connected. 
By taking a log resolution of $(X,\Delta)$ and applying weak semistable reduction (\cite{eh-semistablereduction}), we may assume that $(X,0)$ is $\mathbb{Q}$-factorial klt and all fibers of $\phi$ have the same dimensions. 
Let $H_{Z}$ be a Cartier divisor on $Z$ which is ample over $Y$, and let $H \geq 0$ be an $\mathbb{R}$-divisor on $X$ such that $(X,\Delta+H)$ is an lc pair and $H \sim_{\mathbb{R}} r\phi^{*}H_{Z}$ for some $r>2\cdot{\rm dim}\,X$. 
We run a $(K_{X}+\Delta+H)$-MMP over $Y$ around $W$ with scaling of an ample divisor
$$(X_{0},\Delta_{0}+H_{0})\dashrightarrow (X_{1},\Delta_{1}+H_{1})\dashrightarrow \cdots \dashrightarrow (X_{i},\Delta_{i}+H_{i})\dashrightarrow \cdots.$$
By Lemma \ref{lem--relative-mmp}, the induced map $X_{i} \dashrightarrow Z$ is a morphism and the $(K_{X}+\Delta+H)$-MMP is a $(K_{X}+\Delta)$-MMP over $Y$ around $W$. 
By Lemma \ref{lem--mmp-ample-scaling-nefthreshold} and replacing $(X,\Delta)$ with $(X_{i},\Delta_{i})$ for some $i \gg 0$, we may assume that $K_{X}+\Delta+H$ is the limit of movable divisors over a neighborhood of $W$ (however, we lose the property of being equi-dimensional of $X \to Z$). 
Then $K_{X}+\Delta$ is the limit of movable divisors over a neighborhood of the inverse image of $W$ to $Z$. 
By our hypothesis $\kappa_{\sigma}(X/Z, K_{X}+\Delta)=0$ and the argument as in \cite[Proof of Proposition 3.3]{has-mmp}, we have $K_{X}+\Delta \sim_{\mathbb{R},\,Z}E$ for some effective $\mathbb{R}$-divisor $E$ on $X$ which is very exceptional over $Z$. 
By the negativity lemma for very exceptional divisors (Lemma \ref{lem--negativity-veryexc}) and shrinking $Y$ around $W$, we have $K_{X}+\Delta \sim_{\mathbb{R},Z}0$. 

By the third condition of Lemma \ref{lem--mmp-genericCYfib}, the canonical bundle formula in the analytic setting \cite[Theorem 21.4]{fujino-analytic-bchm} (and the analytic argument of \cite{fg-bundle}), an argument from convex geometry, and the argument of the perturbation of coefficients with the aid of $\kappa_{\sigma}(X/Y, K_{X}+\Delta)={\rm dim}\,Z-{\rm dim}\,Y$, there is a klt pair $(Z,\Delta_{Z})$ such that 
$$K_{X}+\Delta\sim_{\mathbb{R}}\phi^{*}(K_{Z}+\Delta_{Z}).$$ 
By $\kappa_{\sigma}(X/Y, K_{X}+\Delta)={\rm dim}\,Z-{\rm dim}\,Y$, we see that $K_{Z}+\Delta_{Z}$ is big over $Z$. 
By \cite{fujino-analytic-bchm} and shrinking $Y$ around $W$, the klt pair $(Z,\Delta_{Z})$ has a good minimal model $(Z',\Delta_{Z'})$ over $Y$ around $W$. 
Since $(Z,\Delta_{Z})$ is klt, the bimeromorphic map $Z \dashrightarrow Z'$ over $Y$  is a bimeromorphic contraction. 
Hence we can find a Zariski open subset $U' \subset Z'$ such that ${\rm codim}_{Z'}(Z'\setminus U')\geq 2$ and $Z'\dashrightarrow Z$ is a biholomorphism on $U'$.

Let $f\colon X' \to X$ be a log resolution of $(X,\Delta)$ such that the induced map $\phi' \colon X' \dashrightarrow Z'$ is a morphism. 
Let $(X',\Delta')$ be a log birational model of $(X,\Delta)$ as in Definition \ref{defn--models}. 
After shrinking $Y$ around $W$, we can write 
$$f^{*}(K_{X}+\Delta) \sim_{\mathbb{R}} \phi'^{*}(K_{Z'}+\Delta_{Z'})+F$$ 
for some effective $\mathbb{R}$-divisor $F$ on $X$ such that $\phi'({\rm Supp}F) \subset Z'\setminus U'$. 
Since $K_{Z'}+\Delta_{Z'}$ is semi-ample over $Z$, we have $N_{\sigma}(f^{*}(K_{X}+\Delta); X'/Y, W) \leq F$, and the negativity lemma for very exceptional divisors (Lemma \ref{lem--negativity-veryexc}) implies
$$-P_{\sigma}(f^{*}(K_{X}+\Delta); X'/Y, W)\sim_{\mathbb{R},\,Z'}-(F-N_{\sigma}(f^{*}(K_{X}+\Delta); X'/Y, W))\geq 0.$$
Therefore, $N_{\sigma}(f^{*}(K_{X}+\Delta); X'/Y, W) = F$, and the divisor 
$$P_{\sigma}(f^{*}(K_{X}+\Delta); X'/Y, W)\sim_{\mathbb{R}}\phi'^{*}(K_{Z'}+\Delta_{Z'})$$ is semi-ample over $Y$. 
By Theorem \ref{thm--birkarhu-semiample}, after shrinking $Y$ around $W$ the pair $(X,\Delta)$ has a good minimal model over $Y$ around $W$. 
\end{proof}

\begin{prop}[cf.~{\cite[Proposition 3.2]{has-finite}}]\label{prop--crepantmmp}
Let $\pi \colon X \to Y$ be a projective surjective morphism from a normal analytic variety $X$ to a Stein space $Y$, and let $W \subset Y$ be a connected compact subset such that $\pi$ and $W$ satisfy (P). 
Let $(X,\Delta)$ be an lc pair such that $K_{X}+\Delta$ is $\pi$-pseudo-effective and $\pi$-abundant. 
Then, after shrinking $Y$ around $W$, there exist a dlt blow-up $\widetilde{f} \colon (\widetilde{X}, \widetilde{\Delta}) \to (X,\Delta)$, where $\widetilde{X}$ is $\mathbb{Q}$-factorial over $W$, and effective $\mathbb{R}$-Cartier divisors $\widetilde{G}$ and $\widetilde{H}$ on $\widetilde{X}$ satisfying the following properties. 
\makeatletter 
\renewcommand{\p@enumii}{III-} 
\makeatother
\begin{enumerate}[(I)]
\item \label{prop--crepantmmp-(I)}
$K_{\widetilde{X}}+\widetilde{\Delta}\sim_{\mathbb{R}}\widetilde{G}+\widetilde{H}$, 
\item \label{prop--crepantmmp-(II)}
${\rm Supp}\,\widetilde{G}\subset {\rm Supp}\,\llcorner \widetilde{\Delta} \lrcorner$, and
\item \label{prop--crepantmmp-(III)}
there exists a positive real number $t_{0}$ such that for any $t\in(0,t_{0}]$, the following properties hold: 
\begin{enumerate}[({III-}a)]
\item \label{prop--crepantmmp-(III-a)}
The pair $(\widetilde{X},\widetilde{\Delta}+t\widetilde{H})$ is dlt, $N_{\sigma}(K_{\widetilde{X}}+\widetilde{\Delta}+t\widetilde{H}; \widetilde{X}/Y,W)$ is well defined as an $\mathbb{R}$-divisor on $\widetilde{X}$, and the support of $N_{\sigma}(K_{\widetilde{X}}+\widetilde{\Delta}+t\widetilde{H}; \widetilde{X}/Y,W)$ does not depend on $t$, and
\item \label{prop--crepantmmp-(III-b)}
after shrinking $Y$ around $W$, the pair $(\widetilde{X},\widetilde{\Delta}-t\widetilde{G})$ has a good minimal model over $Y$ around $W$. 
\end{enumerate}
\end{enumerate}
\end{prop}

\begin{proof}
By Lemma \ref{thm--speciallogresol} and Lemma \ref{lem--mmp-genericCYfib}, we can apply \cite[Proof of Proposition 3.2]{has-finite} to the situation of Proposition \ref{prop--crepantmmp}. 
Note that we only need to deal with finite sequences of log MMP, and thus \cite[Proof of Proposition 3.2]{has-finite} works with no changes. 
\end{proof}

\begin{prop}[cf.~{\cite[Proposition 3.3]{has-finite}}]\label{prop--specialmmp}
Let $\pi \colon X \to Y$ be a projective surjective morphism from a normal analytic variety $X$ to a Stein space $Y$, and let $W \subset Y$ be a connected compact subset such that $\pi$ and $W$ satisfy (P). 
Let $(X,\Delta)$ be a dlt pair such that $X$ is $\mathbb{Q}$-factorial over $W$ and there are effective $\mathbb{R}$-Cartier divisors $G$ and $H$ on $X$ satisfying (\ref{prop--crepantmmp-(I)}), (\ref{prop--crepantmmp-(II)}), (\ref{prop--crepantmmp-(III)}), (\ref{prop--crepantmmp-(III-a)}), and (\ref{prop--crepantmmp-(III-b)}) in Proposition \ref{prop--crepantmmp}. 
Then, after shrinking $Y$ around $W$, there exist a real number $\lambda_{0}>0$, a bimeromorphic contraction $X \dashrightarrow X_{1}$ over $Y$ with the strict transforms $\Delta_{1}$ and $H_{1}$ of $\Delta$ and $H$ on $X_{1}$ respectively, and a sequence of steps of a $(K_{X_{1}}+\Delta_{1})$-MMP over $Y$ around $W$
$$X_{1}\dashrightarrow X_{2} \dashrightarrow  \cdots \dashrightarrow X_{i} \dashrightarrow \cdots$$
with the strict transforms $\Delta_{i}$ and $H_{i}$ of $\Delta$ and $H$ on $X_{i}$ respectively, which satisfy the following properties. 
\begin{enumerate}[(1)]
\item \label{prop--specialmmp-(1)}
The pair $(X,\Delta+\lambda_{0}H)$ is dlt and all lc centers of $(X,\Delta+\lambda_{0}H)$ are lc centers of $(X,\Delta)$, 
\item \label{prop--specialmmp-(2)}
$X \dashrightarrow X_{1}$ is a sequence of steps of a $(K_{X}+\Delta+\lambda_{0}H)$-MMP over $Y$ around $W$ that is represented by a bimeromorphic contraction over $Y$ and $(X_{1},\Delta_{1}+\lambda_{0}H_{1})$ is a good minimal model of $(X,\Delta+\lambda_{0}H)$ over $Y$, 
\item \label{prop--specialmmp-(3)}
$X_{1}\dashrightarrow \cdots \dashrightarrow X_{i}\dashrightarrow \cdots$ is a sequence of steps of a $(K_{X_{1}}+\Delta_{1})$-MMP over $Y$ around $W$ with scaling of $\lambda_{0}H_{1}$ such that if we define
$$\lambda_{i}={\rm inf}\set{\mu \in \mathbb{R}_{\geq0} \!|\! K_{X_{i}}+\Delta_{i}+\mu H_{i}\text{\rm \;is nef over $W$}}$$
for each $i\geq 1$, then we have ${\rm lim}_{i \to \infty}\lambda_{i}=0$, 
\item \label{prop--specialmmp-(4)}
for all $i\geq1$ and all positive real numbers $u\in [\lambda_{i},\lambda_{i-1}]$, after shrinking $Y$ around $W$, the pair $(X_{i},\Delta_{i}+u H_{i})$ is a good minimal model of $(X,\Delta+u H)$ and $(X_{1},\Delta_{1}+u H_{1})$ over $Y$ around $W$, and 
\item \label{prop--specialmmp-(5)}
for any $i\geq1$ and curve $C_{i}$ in the fiber of $X_{i} \to Y$ over $W$, if $C_{i}$ is contracted by the extremal contraction in the $i$-th step of the $(K_{X_{1}}+\Delta_{1})$-MMP, then $C_{i} \subset {\rm Supp}\,\llcorner \Delta_{i}\lrcorner$. 
\end{enumerate}
\end{prop}

\begin{proof}
The argument in the algebraic case \cite[Proof of Proposition 3.3]{has-finite} works with minor changes. 
We only write how to construct $X \dashrightarrow X_{1}$ and the sequence of steps of the log MMP $X_{1}\dashrightarrow  \cdots \dashrightarrow X_{i} \dashrightarrow \cdots$. 

If we fix $\lambda_{0} \in (0, t_{0})$, where $t_{0}$ is as in (\ref{prop--crepantmmp-(III)}) in Proposition \ref{prop--crepantmmp}, then $(X,\Delta+\lambda_{0}H)$ satisfies (\ref{prop--specialmmp-(1)}) of Proposition \ref{prop--specialmmp}. 
By (\ref{prop--crepantmmp-(I)}) in Proposition \ref{prop--crepantmmp}, for any $s \in \mathbb{R}_{>0}$, we have
$$K_{X}+\Delta+sH \sim_{\mathbb{R},\,Y}(1+s)\left(K_{X}+\Delta-\frac{s}{1+s}G\right).$$
By Theorem \ref{thm--termination-mmp} and (\ref{prop--crepantmmp-(III-b)}) in Proposition \ref{prop--crepantmmp}, for any $t \in (0,\lambda_{0}]$, there exists a Stein open subset $U_{t} \subset Y$ containing $W$ such that putting $X_{U_{t}}:=\pi^{-1}(U_{t})$ then there exists a sequence of steps of a $(K_{X_{U_{t}}}+\Delta|_{X_{U_{t}}}+tH|_{X_{U_{t}}})$-MMP 
$$(X_{U_{t}},\Delta|_{X_{U_{t}}}+tH|_{X_{U_{t}}}) \dashrightarrow (X'_{t},\Delta'_{t}+tH'_{t})$$
over $U_{t}$ around $W$ to a good minimal model, which we may assume that the MMP is represented by a bimeromorphic contraction over $U_{t}$.
By shrinking $Y$ to $U_{\lambda_{0}}$ and putting $X_{1}:=X'_{\lambda_{0}}$, $\Delta_{1}:=\Delta'_{\lambda_{0}}$, and $H_{1}:=H'_{\lambda_{0}}$, we get a bimeromorphic contraction over $Y$
$$(X,\Delta+\lambda_{0}H) \dashrightarrow (X_{1},\Delta_{1}+\lambda_{0}H_{1}),$$
where $(X_{1},\Delta_{1}+\lambda_{0}H_{1})$ is a good minimal model of $(X,\Delta+\lambda_{0}H)$ over $Y$ around $W$. 
By construction, $X \dashrightarrow X_{1}$ satisfies (\ref{prop--specialmmp-(2)}) of Proposition \ref{prop--specialmmp}. 
Let $\pi_{1}\colon X_{1} \to Y$ be the structure morphism. 
By (\ref{prop--crepantmmp-(III-a)}) in Proposition \ref{prop--crepantmmp} and Lemma \ref{lem--asymvanorder-discrepancy}, replacing the open sets $U_{t}\subset Y$ if necessary, we may assume that the induced bimeromorphic map $\pi_{1}^{-1}(U_{t})\dashrightarrow X'_{t}$ is small for all $t \in (0,\lambda_{0}]$. 
Then $(X'_{t},\Delta'_{t}+tH'_{t})$ is a good minimal model of $(\pi_{1}^{-1}(U_{t}), (\Delta_{1}+tH_{1})|_{\pi_{1}^{-1}(U_{t})})$ over $U_{t}$ around $W$. 
By Theorem \ref{thm--mmpsequence-2}, we get a sequence of steps of a $(K_{X_{1}}+\Delta_{1})$-MMP over $Y$ around $W$ with scaling of $\lambda_{0}H_{1}$ 
$$(X_{1},\Delta_{1})\dashrightarrow \cdots \dashrightarrow (X_{i},\Delta_{i}) \dashrightarrow \cdots$$
such that if we define
$$\lambda_{i}={\rm inf}\set{\mu \in \mathbb{R}_{\geq0} \!|\! K_{X_{i}}+\Delta_{i}+\mu H_{i}\text{\rm \;is nef over $W$}}$$
for each $i\geq 1$, then we have ${\rm lim}_{i \to \infty}\lambda_{i}=0$. 
By the same argument as in the algebraic case \cite[Proof of Proposition 3.3]{has-finite}, we can check that the MMP satisfies (\ref{prop--specialmmp-(3)})--(\ref{prop--specialmmp-(5)}) of Proposition \ref{prop--specialmmp}. 
We note that the property (\ref{prop--crepantmmp-(II)}) in Proposition \ref{prop--crepantmmp} is used to prove (\ref{prop--specialmmp-(5)}) of Proposition \ref{prop--specialmmp}. 
For details, see \cite[Proof of Proposition 3.3]{has-finite}. 
\end{proof}

\begin{thm}[cf.~{\cite[Theorem 3.4]{has-finite}}]\label{thm--mmp-ind1}
Let $\pi \colon X \to Y$ be a projective morphism from a normal analytic variety $X$ to a Stein space $Y$, and let $W \subset Y$ be a compact subset such that $\pi$ and $W$ satisfy (P). 
Let $(X,\Delta)$ be a dlt pair. 
Suppose that
\begin{itemize}
\item
$K_{X}+\Delta$ is $\pi$-pseudo-effective and $\pi$-abundant,
\item 
for any lc center $S$ of $(X,\Delta)$, the restriction $(K_{X}+\Delta)|_{S}$ is nef over $W$, and
\item
$\sigma_{P}(K_{X}+\Delta;X/Y)=0$ for all prime divisor $P$ over $X$ such that $a(P,X,\Delta)<0$, the center of $P$ on $X$ intersects an lc center of $(X,\Delta)$, and the image of $P$ on $Y$ intersects $W$. 
\end{itemize}
Then, after shrinking $Y$ around $W$, the dlt pair $(X,\Delta)$ has a log minimal model over $Y$ around $W$.  
\end{thm}

\begin{proof}
By Remark \ref{rem--mmp-reduction-basic}, we may assume that $\pi$ is a contraction and $W$ is connected. 
We closely follow \cite[Proof of Theorem 3.4]{has-finite}. 
We divide the proof into several steps. 

\begin{step3}\label{thm--mmp-ind1-step1}
In this step we will reduce Theorem \ref{thm--mmp-ind1} to a special termination of a log MMP over $Y$ around $W$ with scaling. 
We follow \cite[Step 1 in the proof of Theorem 3.4]{has-finite}. 

By Proposition \ref{prop--crepantmmp} and shrinking $Y$ around $W$, we get a dlt blow-up $(\widetilde{X},\widetilde{\Delta}) \to (X,\Delta)$ and effective $\mathbb{R}$-Cartier divisors $\widetilde{G}$ and $\widetilde{H}$ on $\widetilde{X}$ such that $\widetilde{X}$ is $\mathbb{Q}$-factorial over $W$ and the properties (\ref{prop--crepantmmp-(I)}), (\ref{prop--crepantmmp-(II)}), (\ref{prop--crepantmmp-(III)}), (\ref{prop--crepantmmp-(III-a)}), and (\ref{prop--crepantmmp-(III-b)}) in Proposition \ref{prop--crepantmmp} hold. 
Then it is easy to check that we may replace $(X,\Delta)$ by $(\widetilde{X},\widetilde{\Delta})$. 
Therefore, replacing $(X,\Delta)$ by $(\widetilde{X},\widetilde{\Delta})$, we may assume that $X$ is $\mathbb{Q}$-factorial over $W$ and there exist effective $\mathbb{R}$-Cartier divisors $G$ and $H$ on $X$ satisfying (\ref{prop--crepantmmp-(I)}), (\ref{prop--crepantmmp-(II)}), (\ref{prop--crepantmmp-(III)}), (\ref{prop--crepantmmp-(III-a)}), and (\ref{prop--crepantmmp-(III-b)}) in Proposition \ref{prop--crepantmmp}. 
By Proposition \ref{prop--specialmmp} and shrinking $Y$ around $W$, we obtain a real number $\lambda_{0} >0$, a bimeromorphic contraction 
$$X \dashrightarrow X_{1}$$
over $Y$ with the strict transforms $\Delta_{1}$ and $H_{1}$ of $\Delta$ and $H$ on $X_{1}$ respectively, and a sequence of steps of a $(K_{X_{1}}+\Delta_{1})$-MMP over $Y$ around $W$
$$X_{1}\dashrightarrow X_{2} \dashrightarrow  \cdots \dashrightarrow X_{i} \dashrightarrow \cdots$$
with the strict transforms $\Delta_{i}$ and $H_{i}$ of $\Delta$ and $H$ on $X_{i}$ respectively, which satisfy the following properties. 
\begin{enumerate}[(1)]
\item \label{thm--mmp-ind1-(1)}
The pair $(X,\Delta+\lambda_{0}H)$ is dlt and all lc centers of $(X,\Delta+\lambda_{0}H)$ are lc centers of $(X,\Delta)$, 
\item \label{thm--mmp-ind1-(2)}
$X \dashrightarrow X_{1}$ is a sequence of steps of a $(K_{X}+\Delta+\lambda_{0}H)$-MMP over $Y$ around $W$ that is represented by a bimeromorphic contraction over $Y$ and $(X_{1},\Delta_{1}+\lambda_{0}H_{1})$ is a good minimal model of $(X,\Delta+\lambda_{0}H)$ over $Y$, 
\item \label{thm--mmp-ind1-(3)}
$X_{1}\dashrightarrow \cdots \dashrightarrow X_{i}\dashrightarrow \cdots$ is a sequence of steps of a $(K_{X_{1}}+\Delta_{1})$-MMP over $Y$ around $W$ with scaling of $\lambda_{0}H_{1}$ such that if we define
$$\lambda_{i}={\rm inf}\set{\mu \in \mathbb{R}_{\geq0} \!|\! K_{X_{i}}+\Delta_{i}+\mu H_{i}\text{\rm \;is nef over $W$}}$$
for each $i\geq 1$, then we have ${\rm lim}_{i \to \infty}\lambda_{i}=0$, 
\item \label{thm--mmp-ind1-(4)}
for all $i\geq1$ and all positive real numbers $u\in [\lambda_{i},\lambda_{i-1}]$, after shrinking $Y$ around $W$, the pair $(X_{i},\Delta_{i}+u H_{i})$ is a good minimal model of $(X,\Delta+u H)$ and $(X_{1},\Delta_{1}+u H_{1})$ over $Y$ around $W$, and 
\item \label{thm--mmp-ind1-(5)}
for any $i\geq1$ and curve $C_{i}$ in the fiber of $X_{i} \to Y$ over $W$, if $C_{i}$ is contracted by the extremal contraction in the $i$-th step of the $(K_{X_{1}}+\Delta_{1})$-MMP, then $C_{i} \subset {\rm Supp}\,\llcorner \Delta_{i}\lrcorner$. 
\end{enumerate}
By the argument in the algebraic case (\cite[Step1 in the proof of Theorem 3.4]{has-finite}), we see that Theorem \ref{thm--mmp-ind1} follows from the termination of the $(K_{X_{1}}+\Delta_{1})$-MMP over $Y$ around $W$. 
By (\ref{thm--mmp-ind1-(5)}), we only need to prove that the $(K_{X_{1}}+\Delta_{1})$-MMP terminates on a neighborhood of $\lfloor \Delta_{1} \rfloor$. 
In the rest of the proof, we will carry out the argument of the special termination in Subsection \ref{subsec-special-termi}. 

Let $\{Y_{i}\}_{i \geq 1}$ and $\{\phi_{i}\colon X_{i} \dashrightarrow X'_{i}\}_{i\geq 1}$ be the data of the $(K_{X_{1}}+\Delta_{1})$-MMP. 
There exists $m>0$ such that for any lc center $S_{m}$ of $(X_{m},\Delta_{m})$ and  any $i \geq m$, the bimeromorphic map $X_{m}\times _{Y} Y_{i}\dashrightarrow X_{i}$ induces a bimeromorphic map $S_{m}\dashrightarrow S_{i}$ to an lc center $S_{i}$ of $(X_{i},\Delta_{i})$. 
For any $i\geq m$ and any lc center $S_{i}$ of $(X_{i},\Delta_{i})$, we define $\mathbb{R}$-divisors $\Delta_{S_{i}}$ and $H_{S_{i}}$ on $S_{i}$ by 
$$K_{S_{i}}+\Delta_{S_{i}}=(K_{X_{i}}+\Delta_{i})|_{S_{i}} \qquad {\rm and} \qquad H_{S_{i}}=H_{i}|_{S_{i}},$$
 respectively. Similarly, on every lc center $S$ of $(X,\Delta)$, we define $\mathbb{R}$-divisors $\Delta_{S}$ and $H_{S}$ by 
 $$K_{S}+\Delta_{S}=(K_{X}+\Delta)|_{S} \qquad {\rm and} \qquad H_{S}=H|_{S},$$ respectively. Then $H_{S}$ and $H_{S_{i}}$ are effective $\mathbb{R}$-Cartier divisors, and $(S_{i},\Delta_{S_{i}})$ is dlt for all $i\geq m$ and all $S_{i}$. 

From now on, we prove that for any integer $d \geq0$, there is $m_{d}\geq m$ such that  after $m_{d}$ steps, the non-biholomorphic locus of the $(K_{X_{1}}+\Delta_{1})$-MMP over $Y$ around $W$ does not intersect any $d$-dimensional lc center of $(X_{m_{d}},\Delta_{m_{d}})$ whose image on $Y$ intersects $W$. 
The case $d={\rm dim}\,X-1$ of this statement and (\ref{thm--mmp-ind1-(5)}) imply the termination of the $(K_{X_{1}}+\Delta_{1})$-MMP over $Y$ around $W$. 
Therefore, to prove Theorem \ref{thm--mmp-ind1} it is sufficient to prove the statement.  We will prove the statement by induction on $d$. 
In the case where $d=0$, we can take $m$ as $m_{0}$. 
We assume the existence of $m_{d-1}$ in the statement. 
By replacing $m$ with $m_{d-1}$, we may assume $m_{d-1}=m$. 
By Theorem \ref{thm--sp-ter-1} and replacing $m$ again, we may assume that for any $d$-dimensional lc center $S_{m}$ of $(X_{m},\Delta_{m})$ whose image on $Y$ intersects $W$, the following condition holds: For every $i \geq m$, if we shrink $Y$ to $Y_{i}$ so that the $(K_{X_{1}}+\Delta_{1})$-MMP $(X_{1},\Delta_{1}) \dashrightarrow (X_{i},\Delta_{i})$ is represented by a bimeromorphic contraction $X_{1}\dashrightarrow X_{i}$ over $Y$, then
the induced bimeromorphic map $S_{m}\dashrightarrow S_{i}$ is small and the strict transform of $\Delta_{S_{m}}$ (resp.~$H_{S_{m}}$) on $S_{i}$ is equal to $\Delta_{S_{i}}$ (resp.~$H_{S_{i}}$). 
Since ${\rm lim}_{i \to \infty}\lambda_{i}=0$, by Theorem \ref{thm--sp-ter-2}, to prove the existence of $m_{d}$ it is sufficient to prove that $(S_{m},\Delta_{S_{m}})$ has a log minimal model over $Y$ around $W$. 
\end{step3}

\begin{step3}\label{thm--mmp-ind1-step2}
In the rest of the proof, all indices $i$ are assumed to be greater than or equal to $m$, unless otherwise stated. 
We will show that for any $d$-dimensional lc center $S_{m}$ of $(X_{m},\Delta_{m})$  whose image on $Y$ intersects $W$, the pair $(S_{m},\Delta_{S_{m}})$ has a log minimal model, as required. 

In this step we apply \cite[Step 2 in the proof of Theorem 3.4]{has-finite} to our setting. 
Fix a $d$-dimensional lc center $S_{m}$ of $(X_{m},\Delta_{m})$ whose image on $Y$ intersects $W$. 
By replacing $Y$ with $Y_{m}$, we may assume that the $(K_{X_{1}}+\Delta_{1})$-MMP $(X_{1},\Delta_{1}) \dashrightarrow (X_{m},\Delta_{m})$ is represented by a bimeromorphic contraction $X_{1} \dashrightarrow X_{m}$ over $Y$. 
Then $X \dashrightarrow X_{m}$ is also a bimeromorphic contraction over $Y$. 
By construction of $X\dashrightarrow X_{m}$ in (\ref{thm--mmp-ind1-(1)})--(\ref{thm--mmp-ind1-(3)}) in Step \ref{thm--mmp-ind1-step1}, we can find an lc center $S$ of $(X,\Delta)$ such that the map $X\dashrightarrow X_{m}$ induces a birmeromorphic map $S\dashrightarrow S_{m}$. 
Using the lc center $S$, in the rest of this step, we shrink $Y$ appropriately and we establish the following variety, divisor, and inequalities. 

\begin{enumerate}[(a)]
\item \label{thm--mmp-ind1-step2-(a)}
A projective bimeromorphism $\psi\colon T\to S_{m}$ such that $T$ is $\mathbb{Q}$-factorial over $W$ and for any prime divisor $\bar{D}$ on $S$ whose image on $Y$ intersects $W$, if the strict inequality $a(\bar{D},S_{m}, \Delta_{S_{m}})<a(\bar{D},S,\Delta_{S})$ holds then $\bar{D}$ appears as a prime divisor on $T$, 
\item \label{thm--mmp-ind1-step2-(b)}
an effective $\mathbb{R}$-divisor $\Psi$ on $T$ defined by $\Psi=-\sum_{\substack {D}}a(D,S,\Delta_{S})D$, where $D$ runs over all prime divisors on $T$ whose image intersects $W$, 
\item \label{thm--mmp-ind1-step2-(c)}
 for every $i \geq m$, if we replace $Y$ by $Y_{i}$ so that the log MMP $(X_{1},\Delta_{1}) \dashrightarrow (X_{i},\Delta_{i})$ is represented by a bimeromorphic contraction $X_{1} \dashrightarrow X_{i}$ over $Y$, then
$$a(Q,S,\Delta_{S}+\lambda_{i}H_{S})\leq a(Q,S_{i},\Delta_{S_{i}}+\lambda_{i}H_{S_{i}})$$
holds for any prime divisors $Q$ over $S$ whose image on $Y$ intersects $W$, and 
\item \label{thm--mmp-ind1-step2-(d)}
$K_{T}+\Psi$ is $\mathbb{R}$-Cartier and the inequality $a(Q',S_{m},\Delta_{S_{m}})\leq a(Q',T,\Psi)$ holds for all prime divisors $Q'$ over $S_{m}$, in particular, the pair $(T,\Psi)$ is lc. 
\end{enumerate}
Because \cite[Step 2 in the proof of Theorem 3.4]{has-finite} works with minor changes, we only outline the proof. 

We first show (\ref{thm--mmp-ind1-step2-(c)}). 
We fix $i \geq m$ and we shrink $Y$ to $Y_{i}$ in this paragraph. 
By taking a resolution $\overline{X}\to X$ of $X$ which resolves the indeterminacy of $X\dashrightarrow X_{i}$, we can construct a common resolution $\overline{X}\to X$ and $\overline{X}\to {X_{i}}$ and a subvariety $\overline{S}\subset \overline{X}$ bimeromorphic to $S$ and $S_{i}$ such that the induced bimeromorphisms $\overline{S}\to S$ and $\overline{S}\to S_{i}$ form a common resolution of $S\dashrightarrow S_{i}$. 
By (\ref{thm--mmp-ind1-(4)}) in Step \ref{thm--mmp-ind1-step1}, comparing the coefficients of the pullbacks of $(K_{X}+\Delta+\lambda_{i}H)|_{S}$ and $(K_{X_{i}}+\Delta_{i}+\lambda_{i}H_{i})|_{S_{i}}$ to $\overline{S}$, we obtain
\begin{equation*}
a(Q,S,\Delta_{S}+\lambda_{i}H_{S})\leq a(Q,S_{i},\Delta_{S_{i}}+\lambda_{i}H_{S_{i}})
\end{equation*}
for all prime divisors $Q$ over $S$ whose image on $Y$ intersects $W$. 
We have proved (\ref{thm--mmp-ind1-step2-(c)}). 

Let $\widetilde{D}$ be a prime divisor on $S_{m}$ whose image on $Y$ intersects $W$. 
For every $i$, we put 
$$S_{m}^{(i)}:=S_{m}\times_{Y}Y_{i}, \quad S^{(i)}:=S\times_{Y}Y_{i}, \quad \Delta_{S_{m}}^{(i)}:=\Delta_{S_{m}}|_{S_{m}^{(i)}}, \quad {\rm and} \quad \Delta_{S}^{(i)}:=\Delta_{S}|_{S^{(i)}}.$$ 
Then there is an irreducible components $\widetilde{D}_{i}$ of $\widetilde{D}|_{S_{m}^{(i)}}$ such that the image of $\widetilde{D}_{i}$ on $Y$ intersects $W$ and
\begin{equation*}
\begin{split}
a(\widetilde{D}_{i}, S_{m}^{(i)},\Delta_{S_{m}}^{(i)})=&a(\widetilde{D},S_{m},\Delta_{S_{m}}),\quad {\rm and}\\ a(\widetilde{D}_{i}, S^{(i)},\Delta_{S}^{(i)}+\lambda_{i}H_{S}|_{S^{(i)}})=&a(\widetilde{D},S,\Delta_{S}+\lambda_{i}H_{S}). 
\end{split}
\end{equation*}
Since $(S_{m}^{(i)},\Delta_{S_{m}}^{(i)}) \dashrightarrow (S_{i},\Delta_{S_{i}})$ is small, we have 
$a(\widetilde{D}_{i}, S_{m}^{(i)},\Delta_{S_{m}}^{(i)})=a(\widetilde{D}_{i},S_{i},\Delta_{S_{i}}).$
From these facts and (\ref{thm--mmp-ind1-step2-(c)}), we have
\begin{equation*}
\begin{split}
a(\widetilde{D},S,\Delta_{S}+\lambda_{i}H_{S})=&a(\widetilde{D}_{i}, S^{(i)},\Delta_{S}^{(i)}+\lambda_{i}H_{S}|_{S^{(i)}})
\leq  a(\widetilde{D}_{i},S_{i},\Delta_{S_{i}}+\lambda_{i}H_{S_{i}})\\
\leq&  a(\widetilde{D}_{i},S_{i},\Delta_{S_{i}})=a(\widetilde{D}_{i}, S_{m}^{(i)},\Delta_{S_{m}}^{(i)})
=a(\widetilde{D},S_{m},\Delta_{S_{m}}).
\end{split}
\end{equation*}
On the other hand, by applying Lemma \ref{lem--dlt-lccenter-discrepancy} to $(X,\Delta) \dashrightarrow (X_{m},\Delta_{m})$, $S$, and $S_{m}$, we have
$a(\widetilde{D},S_{m},\Delta_{S_{m}}) \leq a(\widetilde{D},S,\Delta_{S}).$ 
Thus
$$a(\widetilde{D},S,\Delta_{S}+\lambda_{i}H_{S}) \leq a(\widetilde{D},S_{m},\Delta_{S_{m}}) \leq a(\widetilde{D},S,\Delta_{S})$$
for every $i$. 
Since ${\rm lim}_{i \to \infty}\lambda_{i}=0$ by (\ref{thm--mmp-ind1-(3)}) in Step \ref{thm--mmp-ind1-step1}, by considering the limit $i\to \infty$, the equality
\begin{equation*}\tag{$\star$}\label{thm--mmp-ind1-(star)}
a(\widetilde{D},S_{m},\Delta_{S_{m}})= a(\widetilde{D},S,\Delta_{S})
\end{equation*} 
holds for all prime divisors $\widetilde{D}$ on $S_{m}$ whose image on $Y$ intersects $W$. 
This argument shows that (\ref{thm--mmp-ind1-(star)}) holds true even if we shrink $Y$ to a Stein open subset containing $W$. 

In this paragraph we will show (\ref{thm--mmp-ind1-step2-(a)}). 
We set 
\begin{equation*} 
\mathcal{C}=\Set{ \bar{D} | \begin{array}{l}\!\!\text{$\bar{D}$ is a prime divisor on $S$ such that the image of $\bar{D}$ on $Y$}\\
\!\!\text{intersects $W$ and $a(\bar{D},S_{m}, \Delta_{S_{m}})<a(\bar{D},S,\Delta_{S})$} \end{array}\!\!}. \end{equation*} 
By (\ref{thm--mmp-ind1-(star)}), all elements of $\mathcal{C}$ are exceptional over $S_{m}$.  
By a basic property of discrepancies, we have $a(\bar{D},S_{m}, \Delta_{S_{m}}+\lambda_{m}H_{S_{m}}) \leq a(\bar{D},S_{m}, \Delta_{S_{m}})$. 
Combining this with (\ref{thm--mmp-ind1-step2-(c)}), we obtain
$$a(\bar{D},S,\Delta_{S}+\lambda_{m}H_{S}) \leq a(\bar{D},S_{m}, \Delta_{S_{m}})< a(\bar{D},S,\Delta_{S})\leq 0$$
for all $\bar{D}\in \mathcal{C}$. 
Since every element of $\mathcal{C}$ is a prime divisor on $S$, we see that all elements of $\mathcal{C}$ are components of $H_{S}$. 
Thus $ \mathcal{C}$ is a finite set, and furthermore, any $\bar{D}\in \mathcal{C}$ satisfies $-1<a(\bar{D},S,\Delta_{S}+\lambda_{m}H_{S})$. 
Thus, we obtain
$$-1<a(\bar{D},S_{m}, \Delta_{S_{m}})< 0$$ 
for all $\bar{D}\in \mathcal{C}$. 
By Lemma \ref{lem--extraction}, after shrinking $Y$ around $W$ and replacing $\mathcal{C}$ accordingly, we get a projective bimeromorphism $\psi\colon T\to S_{m}$ such that 
\begin{itemize}
\item
$T$ is $\mathbb{Q}$-factorial over $W$, 
\item
any element of $\mathcal{C}$ appears as a $\psi$-exceptional prime divisor on $T$, and 
\item
any 
$\psi$-exceptional prime divisor $\bar{E}$ on $T$ appears as a prime divisor on $S$ and it satisfies $a(\bar{E},S_{m}, \Delta_{S_{m}})<a(\bar{E},S,\Delta_{S})$. 
\end{itemize}
Then $\psi\colon T\to S_{m}$ is the desired bimeromorphism as in (\ref{thm--mmp-ind1-step2-(a)}). 
We note that these three properties hold even if we shrink $Y$ to a Stein open subset containing $W$. 

Next, we will show (\ref{thm--mmp-ind1-step2-(b)}). 
Let $D$ be a prime divisor on $T$ whose image on $Y$ intersects $W$. 
If $D$ is $\psi$-exceptional, then $a(D,S_{m}, \Delta_{S_{m}})<a(D,S,\Delta_{S})\leq 0$ by the definition of $\psi$. 
If $D$ is not $\psi$-exceptional, by (\ref{thm--mmp-ind1-(star)}) we obtain $a(D,S,\Delta_{S})= a(D,S_{m},\Delta_{S_{m}})\leq 0$. 
In any case, the relation
\begin{equation*}\tag{$\star$$\star$}\label{proof-thm--ind-1-(starstarstar)}
a(D,S_{m}, \Delta_{S_{m}})\leq a(D,S,\Delta_{S})\leq 0
\end{equation*} 
holds. 
By shrinking $Y$ around $W$, we may assume that only finitely many prime divisors $D$ on $T$ satisfy $a(D,S,\Delta_{S})<0$ because $a(D,S,\Delta_{S})=0$ when $D$ is a prime divisor on $S$ and not a component of $\Delta_{S}$. 
Therefore, we may define an $\mathbb{R}$-divisor $\Psi\geq0$ on $T$ by
$$\Psi=-\sum_{\substack {D}}a(D,S,\Delta_{S})D,$$
where $D$ runs over all prime divisors on $T$ whose image on $Y$ intersects $W$. 
This is the $\mathbb{R}$-divisor stated in (\ref{thm--mmp-ind1-step2-(b)}). 

Finally, we prove (\ref{thm--mmp-ind1-step2-(d)}). 
Since $T$ is $\mathbb{Q}$-factorial over $W$, there exist an Stein open subset $Y' \subset Y$ containing $W$ such that putting $S'$ and $T'$ as the inverse images of $Y'$ by $S \to Y$ and $T \to Y$ respectively, then $K_{T'}$ is $\mathbb{Q}$-Cartier and
$$\Psi|_{T'}=-\sum_{D}\left(\sum_{D|_{T'}=\bigcup_{j}Q_{D,\,j}}a(Q_{D,\,j},S',\Delta_{S}|_{S'})Q_{j,\,D}\right)$$ 
is $\mathbb{R}$-Cartier, where $D|_{T'}=\bigcup_{j}Q_{D,\,j}$ is the prime decomposition. 
Then the images of some prime divisors $Q_{D,\,j}$ on $Y'$ may not intersect $W$. 
However, by replacing $Y'$ with a smaller Stein open subset to remove such $Q_{D,\,j}$, we may assume that $\Psi|_{T'}$ is $\mathbb{R}$-Cartier. 
In particular, we may assume that $K_{T}+\Psi$ is $\mathbb{R}$-Cartier. 
Moreover we may assume that the image of any component of $K_{T}+\Psi-\psi^{*}(K_{S_{m}}+\Delta_{S_{m}})$ intersects $W$. 
By (\ref{proof-thm--ind-1-(starstarstar)}), we obtain $K_{T}+\Psi\leq \psi^{*}(K_{S_{m}}+\Delta_{S_{m}})$. 
From this, we have
\begin{equation*}
-1\leq a(Q',S_{m},\Delta_{S_{m}})\leq a(Q',T,\Psi) 
\end{equation*}
for any prime divisor $Q'$ over $S_{m}$. 
This shows (\ref{thm--mmp-ind1-step2-(d)}). 
\end{step3}

\begin{step3}\label{thm--mmp-ind1-step3}
In this step we follow \cite[Step 3 in the proof of Theorem 3.4]{has-finite}. 
We will prove that after shrinking $Y$ around $W$ the lc pair $(T,\Psi)$ has a log minimal model over $Y$ around $W$. 

To apply Lemma \ref{lem--bir-relation} to the bimeromorphic map $(S,\Delta_{S})\dashrightarrow (T,\Psi)$, we will prove the following claim. 

\hypertarget{thm--mmp-ind1-step3-claim}{\begin{claim*}
Let $\tilde{Q}$ be a prime divisor over $S$ whose image on $Y$ intersects $W$.
Then the following two statements hold:
\begin{itemize}
\item
If $\tilde{Q}$ is a prime divisor on $S$, then $a(\tilde{Q},S,\Delta_{S})\leq a(\tilde{Q},T,\Psi)$, and
\item
if $\tilde{Q}$ is a prime divisor on $T$, then $a(\tilde{Q},T,\Psi)\leq a(\tilde{Q},S,\Delta_{S})$. 
\end{itemize}
\end{claim*}}

\begin{proof}[Proof of Claim]
Let $\tilde{Q}$ be a prime divisor over $S$ whose image on $Y$ intersects $W$. 
When $\tilde{Q}$ appears as a prime divisor on $T$ or a prime divisor on $S$ which is not exceptional over $T$, by (\ref{thm--mmp-ind1-step2-(b)}) in Step \ref{thm--mmp-ind1-step2} we have $$a(\tilde{Q},T,\Psi)=-{\rm coeff}_{\tilde{Q}}(\Psi)=a(\tilde{Q},S,\Delta_{S}).$$ 
Therefore, the second assertion of the claim holds. 
To prove the first assertion of the claim, we may assume that $\tilde{Q}$ appears as a prime divisor on $S$ that is exceptional over $T$. 
In this case, (\ref{thm--mmp-ind1-step2-(a)}) in Step \ref{thm--mmp-ind1-step2} shows $a(\tilde{Q},S,\Delta_{S})\leq a(\tilde{Q},S_{m},\Delta_{S_{m}})$, then (\ref{thm--mmp-ind1-step2-(d)}) in Step \ref{thm--mmp-ind1-step2} shows
$$a(\tilde{Q},S,\Delta_{S})\leq a(\tilde{Q},S_{m},\Delta_{m})\leq a(\tilde{Q},T,\Psi).$$
Therefore, the first assertion of the claim also holds. 
In this way, the claim holds. 
\end{proof}
By the second condition of Theorem \ref{thm--mmp-ind1} and Lemma \ref{lem--bir-relation}, after shrinking $Y$ around $W$ the lc pair $(T,\Psi)$ has a log minimal model over $Y$ around $W$. 
\end{step3}

\begin{step3}\label{thm--mmp-ind1-step4}
In this step we follow \cite[Step 4 in the proof of Theorem 3.4]{has-finite}. 
We will prove that after shrinking $Y$ around $W$ the lc pair $(S_{m},\Delta_{S_{m}})$ has a log minimal model over $Y$ around $W$. 

For every $i$, we put 
\begin{equation*}
\begin{split}
S_{m}^{(i)}:=&S_{m}\times_{Y}Y_{i}, \quad S^{(i)}:=S\times_{Y}Y_{i}, \quad T^{(i)}:=T \times_{Y}Y_{i}, \\
 \Delta_{S_{m}}^{(i)}:=&\Delta_{S_{m}}|_{S_{m}^{(i)}}, \quad \Delta_{S}^{(i)}:=\Delta_{S}|_{S^{(i)}}, \quad
H_{S_{m}}^{(i)}:=H_{S_{m}}|_{S_{m}^{(i)}}, \quad {\rm and} \quad H_{S}^{(i)}:=H_{S}|_{S^{(i)}}
\end{split}
\end{equation*} 
For each $i$, the pair $(S_{i},\Delta_{S_{i}}+\lambda_{i}H_{S_{i}})$ is a weak lc model of $(S_{m}^{(i)}, \Delta_{S_{m}}^{(i)}+\lambda_{i}H_{S_{m}}^{(i)})$ because the bimeromorphic map $S_{m}^{(i)}\dashrightarrow S_{i}$ is small, $\Delta_{S_{i}}+\lambda_{i}H_{S_{i}}$ is equal to the strict transform of $\Delta_{S_{m}}^{(i)}+\lambda_{i}H_{S_{m}}^{(i)}$ on $S_{i}$, and the divisor $K_{S_{i}}+\Delta_{S_{i}}+\lambda_{i}H_{S_{i}}$ is nef over $W$. 
Pick any prime divisor $D$ on $T$ whose image on $Y$ intersects $W$. 
For every $i$, there is an irreducible components $D^{(i)}$ of $D|_{S_{m}^{(i)}}$ such that the image of $D^{(i)}$ on $Y$ intersects $W$ and
\begin{equation*}
\begin{split}
a(D^{(i)}, S_{m}^{(i)},\Delta_{S_{m}}^{(i)}+\lambda_{i}H_{S_{m}}^{(i)})=&a(D,S_{m},\Delta_{S_{m}}+\lambda_{i}H_{S_{m}}),\quad {\rm and}\\ a(D^{(i)}, S^{(i)},\Delta_{S}^{(i)}+\lambda_{i}H_{S}^{(i)})=&a(D,S,\Delta_{S}+\lambda_{i}H_{S}). 
\end{split}
\end{equation*}
By Lemma \ref{lem--asymvanorder-discrepancy} and (\ref{thm--mmp-ind1-step2-(c)}) in Step \ref{thm--mmp-ind1-step2}, we obtain
\begin{equation*}
\begin{split}
\sigma_{D^{(i)}}(K_{S_{m}^{(i)}}+\Delta_{S_{m}}^{(i)}+\lambda_{i}H_{S_{m}}^{(i)})=a(D^{(i)},S_{i},\Delta_{S_{i}}+\lambda_{i}H_{S_{i}} )-a(D^{(i)}, S_{m}^{(i)},\Delta_{S_{m}}^{(i)}+\lambda_{i}H_{S_{m}}^{(i)})&\\
\geq a(D^{(i)}, S^{(i)},\Delta_{S}^{(i)}+\lambda_{i}H_{S}^{(i)})-a(D^{(i)}, S_{m}^{(i)},\Delta_{S_{m}}^{(i)}+\lambda_{i}H_{S_{m}}^{(i)})&.
\end{split}
\end{equation*}
By these relations and Theorem \ref{thm--asymvanorder-pseudoeff-basic} (\ref{thm--asymvanorder-pseudoeff-basic-(6)}), we obtain
$$\sigma_{D}(K_{S_{m}}+\Delta_{S_{m}}+\lambda_{i}H_{S_{m}})\geq a(D,S, \Delta_{S}+\lambda_{i}H_{S})-a(D, S_{m}, \Delta_{S_{m}}+\lambda_{i}H_{S_{m}}).$$ 
By (\ref{thm--mmp-ind1-step2-(b)}) in Step \ref{thm--mmp-ind1-step2}, we have $a(D,S,\Delta_{S})=a(D,T,\Psi)$. 
By Theorem \ref{thm--asymvanorder-pseudoeff-basic} (\ref{thm--asymvanorder-pseudoeff-basic-(4)}), using the fact that ${\rm lim}_{i \to \infty}\lambda_{i}=0$ (see (\ref{thm--mmp-ind1-(3)}) in Step \ref{thm--mmp-ind1-step1}) and taking the limit $i\to \infty$, we obtain  
\begin{equation*}
\begin{split}
\sigma_{D}(K_{S_{m}}+\Delta_{S_{m}})\geq&\underset{i\to \infty}{\rm lim}\bigl(a(D,S, \Delta_{S}+\lambda_{i}H_{S})-a(D, S_{m}, \Delta_{S_{m}}+\lambda_{i}H_{S_{m}})\bigr)\\
=&a(D,T,\Psi)-a(D, S_{m}, \Delta_{S_{m}})\\
\geq & 0.
\end{split}
\end{equation*}
Here, the last inequality holds by (\ref{thm--mmp-ind1-step2-(d)}) in Step \ref{thm--mmp-ind1-step2}. 
By Theorem \ref{thm--minmodel-asymvanorder-boundary} and the existence of a log minimal model of $(T, \Psi)$ over a neighborhood of $W$ as discussed in Step \ref{thm--mmp-ind1-step3}, after shrinking $Y$ around $W$ the lc pair $(S_{m},\Delta_{S_{m}})$ has a log minimal model.  
\end{step3}

By Step \ref{thm--mmp-ind1-step1} and Step \ref{thm--mmp-ind1-step4}, we complete the argument of the special termination as in Subsection \ref{subsec-special-termi}. 
Therefore, after shrinking $Y$ around $W$, the dlt pair $(X,\Delta)$ has a log minimal model over $Y$ around $W$. 
We complete the proof. 
\end{proof}

\subsection{Proof of main results}

In this subsection we prove the main results of this paper. 

\begin{thm}[cf.~{\cite[Theorem 3.5]{has-finite}}]\label{thm--mmp-ind2}
Let $\pi \colon X \to Y$ be a projective morphism from a normal analytic variety $X$ to a Stein space $Y$, and let $W \subset Y$ be a compact subset such that $\pi$ and $W$ satisfy (P). 
Let $(X,\Delta)$ be an lc pair. 
Let $A$ be an effective $\mathbb{R}$-divisor on $X$ such that $(X,\Delta+A)$ is an lc pair and $K_{X}+\Delta+A$ is nef over $W$.
Then no infinite sequence of steps of a $(K_{X}+\Delta)$-MMP over $Y$ around $W$ with scaling of $A$
$$(X_{0}:=X,\Delta_{0}:=\Delta) \dashrightarrow (X_{1},\Delta_{1}) \dashrightarrow\cdots \dashrightarrow (X_{i},\Delta_{i})\dashrightarrow \cdots$$
satisfies the following properties.
\begin{itemize}
\item
If we define $\lambda_{i}={\rm inf}\set{\mu \in \mathbb{R}_{\geq0} \!|\! K_{X_{i}}+\Delta_{i}+\mu A_{i}\text{\rm \, is nef over }W}$, where $A_{i}$ is the strict transform of $A$ on $X_{i}$, then ${\rm lim}_{i\to\infty}\lambda_{i}=0$, and  
\item
there are infinitely many $i$ such that $(X_{i},\Delta_{i})$ is log abundant over $Y$ around $W$. 
\end{itemize}
\end{thm}

\begin{proof}
The argument in \cite[Proof of Theorem 3.5]{has-finite} works in our situation since we may use the special termination (Subsection \ref{subsec-special-termi}), the lift of MMP (Subsection \ref{subsection--lift-mmp}), and Theorem \ref{thm--mmp-ind1}. 
\end{proof}

\begin{thm}[cf.~{\cite[Theorem 1.5]{hashizumehu}}]\label{thm--mmp-ample+eff}
Let $\pi \colon X \to Y$ be a projective morphism from a normal analytic variety $X$ to a Stein space $Y$, and let $W \subset Y$ be a compact subset such that $\pi$ and $W$ satisfy (P). 
Let $(X,B+A)$ be an lc pair, where $B$ is an effective $\mathbb{R}$-divisor on $X$ and $A$ is an effective $\pi$-ample $\mathbb{R}$-divisor on $X$.  
Let $f \colon (\tilde{X},\tilde{B}) \to (X,B)$ be a dlt blow-up of $(X,B)$, and we put $\tilde{\Gamma}=\tilde{B}+f^{*}A$. 
Let $\tilde{H}$ be a $(\pi \circ f)$-ample $\mathbb{R}$-divisor on $\tilde{X}$ such that $(\tilde{X},\tilde{\Gamma}+\tilde{H})$ is lc and $K_{\tilde{X}}+\tilde{\Gamma}+\tilde{H}$ is nef over $W$. 
Then there exists a sequence of steps of a $(K_{\tilde{X}}+\tilde{\Gamma})$-MMP over $Y$ around $W$ with scaling of $\tilde{H}$
$$(\tilde{X}_{0}:=\tilde{X},\tilde{\Gamma}_{0}:=\tilde{\Gamma}) \dashrightarrow (\tilde{X}_{1},\tilde{\Gamma}_{1}) \dashrightarrow\cdots \dashrightarrow (\tilde{X}_{m},\tilde{\Gamma}_{m})$$
such that after shrinking $Y$ around $W$ the lc pair $(\tilde{X}_{m},\tilde{\Gamma}_{m})$ is a log minimal model or a Mori fiber space of $(X,\Delta+A)$ over $Y$ around $W$.  
\end{thm}

\begin{proof}
Let 
$$(\tilde{X}_{0}:=\tilde{X},\tilde{\Gamma}_{0}:=\tilde{\Gamma}) \dashrightarrow (\tilde{X}_{1},\tilde{\Gamma}_{1}) \dashrightarrow\cdots \dashrightarrow (\tilde{X}_{i},\tilde{\Gamma}_{i})\dashrightarrow \cdots$$
be a sequence of steps of a $(K_{\tilde{X}}+\tilde{\Gamma})$-MMP over $Y$ around $W$ with scaling of $\tilde{H}$. 
By the argument as in \cite{has-nonvan-gpair}, we see that all $(\tilde{X}_{i},\tilde{\Gamma}_{i})$ are log abundant over $Y$ around $W$. 
Moreover, if we define 
$$\lambda_{i}={\rm inf}\set{\mu \in \mathbb{R}_{\geq0} \!|\! K_{\tilde{X}_{i}}+\tilde{\Gamma}_{i}+\mu \tilde{H}_{i}\text{\rm \, is nef over }W},$$ 
where $\tilde{H}_{i}$ is the strict transform of $\tilde{H}$ on $\tilde{H}_{i}$, then ${\rm lim}_{i\to\infty}\lambda_{i}=0$. 
By Theorem \ref{thm--mmp-ind2}, the $(K_{\tilde{X}}+\tilde{\Gamma})$-MMP terminates. 
Therefore Theorem \ref{thm--mmp-ample+eff} holds. 
\end{proof}

\begin{thm}[cf.~{\cite[Theorem 1.1]{has-mmp}}, {\cite[Theorem 1.1]{birkar-flip}}, {\cite[Theorem 1.6]{haconxu-lcc}}]\label{thm--mmp-negativetrivial}
Let $\pi \colon X \to Y$ be a projective morphism from a normal analytic variety $X$ to a Stein space $Y$, and let $W \subset Y$ be a compact subset such that $\pi$ and $W$ satisfy (P). 
Let $(X,B)$ be a dlt pair and $A$ an effective $\mathbb{R}$-Cartier divisor on $X$ such that $(X,B+A)$ is lc and $K_{X}+B+A\sim_{\mathbb{R},\,Y}0$.  
Let $H$ be a $\pi$-ample $\mathbb{R}$-divisor on $X$ such that $(X,B+H)$ is lc and $K_{X}+B+H$ is nef over $W$. 
Then there exists a sequence of steps of a $(K_{X}+B)$-MMP over $Y$ around $W$ with scaling of $H$
$$(X_{0}:=X,B_{0}:=B) \dashrightarrow (X_{1},B_{1}) \dashrightarrow\cdots \dashrightarrow (X_{m},B_{m})$$
such that after shrinking $Y$ around $W$ the lc pair $(X_{m},B_{m})$ is a log minimal model or a Mori fiber space of $(X,B)$ over $Y$ around $W$.  
\end{thm}

\begin{proof}
The argument in \cite[Remark 3.7]{has-finite} works with no changes. 
\end{proof}



\end{document}